%% file: konsoib_arxiv2.tex
\documentclass[reqno]{amsart}
\usepackage{amssymb,amsmath}
\usepackage{verbatim}
\usepackage{mathrsfs}
\usepackage{esint}
\usepackage[all]{xy}
\usepackage[pdftex]{graphicx}
\usepackage[pdftex]{hyperref}

\title[Tropical and non-Archimedean limits of volume forms]
{Tropical and non-Archimedean limits of degenerating families of volume forms}
\author{S{\'e}bastien Boucksom
  \and
  Mattias Jonsson}
\address{CMLS, \'Ecole polytechnique\\
CNRS, Universit\'e Paris-Saclay\\ 
 91128 Palaiseau Cedex\\
  France}
\email{sebastien.boucksom@polytechnique.edu}
\address{Dept of Mathematics\\
  University of Michigan\\
  Ann Arbor, MI 48109-1043\\
  USA}
\address{Mathematical Sciences\\
  Chalmers University of Technology
  and University of Gothenburg\\
  SE-412 96 G\"oteborg\\
  Sweden}
\email{mattiasj@umich.edu}

\date\today


\include{preamb}
\begin{document}
\begin{abstract}
  We study the asymptotic behavior of volume forms on a degenerating family of 
  compact complex manifolds. Under rather general conditions, we prove that the 
  volume forms converge in a natural sense to a Lebesgue-type measure on a certain 
  simplicial complex.
  In particular, this provides a measure-theoretic version of a conjecture by 
  Kontsevich--Soibelman and Gross--Wilson, bearing on maximal degenerations 
  of Calabi--Yau manifolds. 
\end{abstract}

\maketitle
\setcounter{tocdepth}{1}
\tableofcontents 
%
%
%
%
%
%
\section*{Introduction} 
As is well-known, there is a natural bijection between (smooth,
positive) volume forms on a complex manifold and smooth Hermitian
metrics on its canonical bundle. Consequently, the data of a smooth
family $(\nu_t)_{t\in\DD^*}$ of volume forms on a holomorphic family
$(X_t)_{t\in\DD^*}$ of compact complex manifolds is equivalent to
that of a proper holomorphic submersion $\pi\colon X\to\DD^*$
together with a smooth metric $\p$ on the relative canonical bundle
$K_{X/\DD^*}$. 

We say that the family $(\nu_t)$ has 
\emph{analytic singularities} at $t=0$ if the following conditions
hold:
\begin{itemize}
\item[(i)]
  $\pi\colon X\to\DD^*$ is meromorphic at $0\in\DD$ in the sense
  that it extends to a proper, flat map
  $\pi\colon\cX\to\DD$, with $\cX$ normal;
\item[(ii)]
  $\cX$ can be chosen so that $K_{X/\DD^*}$ extends to a $\Q$-line
  bundle $\cL$ on $\cX$, and $\p$ extends continuously to $\cL$.
\end{itemize}
When~(i) holds, we call $\cX$ a \emph{model} of $X$. 
Using resolution of singularities, 
we can always choose $\cX$ as an \emph{snc model}, that is, 
$\cX$ is smooth and $\cX_0=\sum_{i\in I} b_i E_i$ has simple normal
crossing support. To $\cX$ is then associated a dual complex $\D(\cX)$, 
with one vertex $e_i$ for each $E_i$, and a face $\sigma$ 
for each connected component $Y$ of a non-empty intersection 
$E_J=\bigcap_{i\in J} E_i$ with $J\subset I$. 

In the spirit of the Morgan-Shalen topological compactification 
of affine varieties~\cite{MS}, we introduce a natural ``hybrid'' space
\begin{equation*}
  \cX^\hyb:=X\coprod\D(\cX)
\end{equation*}
associated to $\cX$; it is equipped with a topology 
defined in terms of a tropicalization map $X\to\Delta(\cX)$,
measuring the logarithmic rate of convergence of 
local coordinates compatible with $\cX_0$.

Our first main result says that, after normalizing to unit mass,
the volume forms $\nu_t$ admit a ``tropical'' limit inside $\cX^\hyb$.
\begin{ThmA} 
  Let $(\nu_t)_{t\in\DD^*}$ be a family of volume forms on a 
  holomorphic family $X\to\DD^*$ of compact complex
  manifolds, with analytic singularities at $t=0$. 
  The asymptotic behavior of the total mass of $\nu_t$ is then given by
  \begin{equation*}
    \nu_t(X_t)\sim c|t|^{2\kappa_{\min}}(\log|t|^{-1})^d
  \end{equation*}
  with $c\in\R_+^*$, $\kappa_{\min}\in\Q$ and $d\in\N^*$, where $d\le n:=\dim X_t$. 
  Further, given any snc model $\cX\to\DD$
  of $X\to\DD^*$ such that $K_{X/\DD^*}$ extends to a $\Q$-line bundle
  on $\cL$ on $\cX$ and $\p$ extends to a continuous metric on $\cL$, 
  the rescaled measures
  \begin{equation*}
    \mu_t:=\frac{\nu_t}{|t|^{2\kappa_{\min}}(2\pi\log|t|^{-1})^d},
  \end{equation*}
  viewed as measures on $\cX^\hyb$, converge weakly to a 
  Lebesgue type measure $\mu_0$ on a $d$-dimensional subcomplex 
  $\D(\cL)$ of $\D(\cX)$. 
\end{ThmA}
The invariant $\kappa_{\min}$ and the subcomplex $\D(\cL)$ only depend on $\cL$ (and not on the metric on $\cL$). Consider the \emph{logarithmic relative canonical bundle} 
$$
K^\lo_{\cX/\DD}:=K_\cX+\cX_{0,\red}-\pi^*(K_{\DD}+[0])=K_{\cX/\DD}+\cX_{0,\red}-\cX_0
$$
and write $K^{\lo}_{\cX/\DD}=\cL+\sum_{i\in I} a_i E_i$ with $a_i\in\Q$. Setting $\kappa_i:=a_i/b_i$, we then have $\kappa_{\min}=\min_{i\in I}\kappa_i$, and $\D(\cL)$ is the subcomplex of $\D(\cX)$ whose vertices $e_i$ correspond to those $i\in I$ achieving the minimum. 

On the other hand, the limit measure $\mu_0$ does depend on $\p$;
it is given by 
$$
\mu_0=\sum_{\sigma}\left(\int_{Y_\sigma}\Res_{Y_\sigma}(\p)\right)b_\sigma^{-1}\la_\sigma. 
$$
Here, $\sigma$ ranges over the $d$-dimensional faces of $\D(\cL)$, with corresponding strata $Y_\sigma\subset\cX_0$, $\Res_{Y_\sigma}(\p)$ is a naturally defined residual positive measure on $Y_\sigma$, $\la_\sigma$ is the Lebesgue measure of $\sigma$ normalized by its natural integral affine structure, and $b_\sigma\in\Z_{>0}$ is an arithmetic coefficient. 

\medskip
The study of the asymptotics of integrals is a very classical subject and has been 
pursued by many people; see for example the book~\cite{AGZV}. 
The assertions in Theorem~A are closely related to results by Chambert-Loir 
and Tschinkel (who also worked over general local fields and in an adelic setting). 
Specifically, the estimate for $\nu_t(X_t)$, suitably averaged over $t$, is essentially
equivalent to~\cite[Theorem~1.2]{CLT10}.
It also appears in~\cite[\S3.1]{KS01} and is exploited in~\cite{BHJ2}.

The convergence result for the measures $\mu_t$ is also closely related 
to~\cite[Corollary~4.8]{CLT10}, where, however, 
the limit measure lives on $\cX_0$ and not on 
$\D(\cX)$.\footnote{A.~Chambert-Loir has pointed out
  that~\cite[Corollary~4.8]{CLT10} is sufficiently precise, so that when applying it to 
  toric blowups of $\cX$ one can see the form of the limit measure $\mu_0$ 
  in Theorem~A.}
The main new feature of Theorem~A is the precise and explicit convergence of the measure 
$\mu_t$ to a ``tropical'' limit $\mu_0$, living on a simplicial complex.

\medskip
The following examples illustrate Theorem~A. First consider the subvariety
\begin{equation*}
  \cX
  :=\{(z_0^{n+1}+\dots+z_n^{n+1})+\e tz_0\cdot\ldots\cdot z_n=0\}\subset\C\times\P^n,
\end{equation*}
where $0<\e\ll1$.
Write $X:=\mathrm{pr}_1^{-1}(\C^*)$.
The fiber $X_t$ over $t\in\DD^*$ is a Calabi-Yau manifold, and 
we can choose a nonvanishing holomorphic $n$-form $\eta_t$ on $X_t$
to define a smooth metric $\p$ on $K_{X/\DD^*}$ 
that extends continuously to $\cL=K_{\cX/\DD}$. 
In the terminology of Theorem~A we have 
$\nu_t:=2^{-n}i^{n^2}\eta_t\wedge\overline\eta_t$.
Here $\cX_0$ is smooth, so $\D(\cX)$ is a single point.
Thus $\nu_t(X_t)\sim c$ for some $c>0$, and the limit measure 
$\mu_0$ is a point mass.

Now consider instead 
\begin{equation*}
  \cX
  :=\{t\e(z_0^{n+1}+\dots+z_n^{n+1})+z_0\cdot\ldots\cdot z_n=0\}\subset\C\times\P^n.
\end{equation*}
In this case, $\D(\cL)=\D(\cX)$ is a union of $(n+1)$ simplices of dimension $n$, and
topologically a sphere. We have $\nu_t(X_t)\sim c(\log|t|^{-1})^n$ and the limit measure
is a weighted sum of Lebesgue measures on each simplex. In fact, it is clear
by symmetry that the weights are equal; this also follows from Theorem~C below.

\medskip
We also prove a logarithmic version of Theorem~A, for a log smooth klt pair $(X,B)$,
and a metric $\p$ on $K_{(X,B)/\DD^*}$, see Theorem~\ref{T301}. 

\medskip 
The space $\cX^\hyb$ and the measure $\mu_0$ depend on the choice of snc model $\cX$. 
We obtain a more canonical situation by considering all possible snc models 
simultaneously. 
Namely, the set of snc models of $X$ is directed, 
and in~\S\ref{S315} we define a locally compact (Hausdorff) topological space 
\begin{equation*}
  X^\hyb:=\varprojlim_\cX \cX^\hyb,
\end{equation*}
fibering over $\DD$, with central fiber
$X^\hyb_0:=\varprojlim\D(\cX)$.
For any $\cX$, the dual complex $\Delta(\cX)$ embeds
in the central fiber $X^\hyb_0$ of $X^\hyb$.
\begin{CorB}
  With assumptions and notation as in Theorem~A, 
  the measures $\mu_t$, viewed as measures on $X^\hyb$, 
  converge weakly to a measure $\mu_0$.
  Further, $\mu_0$ is a Lebesgue type measure on a 
  $d$-dimensional complex in $X^\hyb_0$.
\end{CorB}

\medskip  
Now consider the case when $X\to\DD^*$ is projective. As we now 
explain, the central fiber of $X^\hyb$ is then a \emph{non-Archimedean} space.
Namely, $X$ induces a smooth projective variety $X_K$ over the non-Archimedean
field $K$ of complex formal Laurent series, to which we can associate
a Berkovich analytification $X_K^\an$. 
Similarly, any projective snc model $\cX\to\DD$ of $X$ induces a
projective model $\cX_R$ over the valuation ring $R$ of $K$. 
The dual complex $\D(\cX)$ then has a canonical realization as a 
compact $\Z$-PA subspace $\Sk(\cX)\subset X_K^\an$, the \emph{skeleton}
of $\cX$. In fact, it is well known (see~\eg\cite{siminag}) that 
there is a homeomorphism $X_K^\an\simto\varprojlim_\cX\Sk(\cX)$, 
so we can identify the central fiber of the space $X^\hyb$ with
the analytification $X_K^\an$. In fact, as shown in Appendix~\ref{S319}, using
ideas from~\cite{BerkHodge}, we can view the restriction of 
$X^\hyb\to\DD$ to a closed subdisc $\overline{\DD}_r$ as the analytification of the base
change of $X$ to a suitable Banach ring $A_r$. 

\medskip
Assuming $X\to\DD^*$ is projective, we can describe the limit measure $\mu_0$ and its
support $\Sk(\cL)\simeq\Delta(\cL)$ inside $X_K^\an$ in more
detail. The skeleton $\Sk(\cL)$ is of purely non-Archimedean nature,
and can be seen as a mild generalization of the
\emph{Kontsevich--Soibelman skeleton} 
introduced in~\cite{KS06} and studied in~\cite{MN,NX13,NX16}. 
The \emph{skeletal measure} $\mu_0$, on the other hand, depends on
both Archimedean and non-Archimedean data.
Namely, it is supported on the skeleton $\Sk(\cL)$, but depends on
the choice of metric on the restriction of the line bundle $\cL$ to the central
fiber $\cX_0$ (viewed as a complex space) of any snc model $\cX$.

We also study both the skeleton and the skeletal measure in the more
general case when the model $\cX$ is allowed to have mild (dlt) singularities. 

\medskip
One major motivation for studying the above general setting comes from
degenerations of Calabi--Yau manifolds. Thus suppose $X\to\DD^*$ is a
projective holomorphic submersion, meromorphic at $0\in\DD$, such that 
$K_{X/\DD^*}=\cO_X$. 
Any trivializing section $\eta\in H^0(X,K_{X/\DD^*})$ then defines a
family $\eta_t:=\eta|_{X_t}$ of trivializations of $K_{X_t}$, and
hence a smooth family of volume forms $\nu_t:=|\eta_t|^2$ 
with analytic singularities at $t=0$.
Indeed, for any snc model $\cX\to\DD$, $\eta$ extends to a nowhere vanishing section 
of $\cL:=\cO_\cX$, and $\p:=\log|\eta|$ defines a smooth metric on $\cL$.

The total mass $\nu_t(X_t)=\int_{X_t}|\eta_t|^2$ is then nothing 
but the $L^2$ (or Hodge) metric on the direct image of 
$K_{X/\DD^*}$, whose asymptotic behavior at $t=0$ is described 
in a very precise way by Schmid's nilpotent orbit 
theorem~\cite[Theorem 4.9]{Sch} (compare for 
instance~\cite[Proposition 2.1]{GTZ13b}). 

On the other hand, the skeleton $\Sk(\cL)$ described above coincides
in the current context with the Kontsevich--Soibelman skeleton
$\Sk(X)$~\cite{KS06,MN,NX13}. Its dimension $d$, which features as the
exponent of the log term in the asymptotics of the mass, measures how
``bad'' the degeneration is. 
Further, the family $X\to\DD^*$ 
admits a relative minimal model $\cX$, with certain
mild (dlt) singularities~\cite{KNX}, and the essential skeleton can
be identified with the dual complex of $\cX$~\cite{NX13}. In
particular, $d=0$ if and only if $X$ can filled in with a central
fiber $\cX_0$ which is a Calabi--Yau variety with klt singularities. 

At the other end of the spectrum, $d=n=\dim X_t$ if and only if $X$ 
is maximally degenerate, \ie a ``large complex structure limit''. 
In that case, the essential skeleton $\Sk(X)$ is shown to be a 
pseudomanifold in~\cite{NX13}. Building on this, we prove:
\begin{ThmC} 
  Let $X\to\DD^*$ be a smooth projective family
  of Calabi--Yau varieties, meromorphic at $0\in\DD$. Assume that $X$
  is maximally degenerate and has semistable reduction. 
  Then the skeletal measure $\mu_0$ 
  is a multiple of the integral affine Lebesgue measure on $\Sk(X)$.
\end{ThmC} 
This theorem also holds in the purely non-Archimedean setting of 
Calabi--Yau varieties defined over the field of Laurent series.
The semistable reduction condition means that $X$ admits an snc model 
$\cX$ with $\cX_0$ reduced. This condition is always 
satisfied after a finite base change.

\medskip  
Theorem~C describes measure-theoretic degenerations of Calabi--Yau
varieties. Let us briefly discuss the case of \emph{metric}
degenerations.
Consider a smooth projective family $X\to\DD^*$ of Calabi--Yau
varieties, meromorphic at $0\in\DD$, and suppose the
family is polarized, that is, we are given a relative ample line
bundle $A$ on $X$. By Yau's theorem~\cite{Yau78}, each fiber $X_t$ carries a
unique Ricci-flat K\"ahler metric $\om_t$ in the cohomology 
class of $A_t$.

By~\cite{Wan,Tos15,Taka}, the diameter $D_t$ of $(X_t,\om_t)$ 
remains bounded if and only if $d=0$, that is, $X$ admits a
model $\cX$ such that $\cX_0$ has klt singularities.
In this case, it is shown in~\cite{RZ11,RZ13}, 
building in part on~\cite{DS}, that $(X_t,\om_t)$ 
converges in the Gromov-Hausdorff sense to the 
Calabi--Yau variety $\cX_0$, endowed with the metric completion 
of its singular Ricci-flat K\"ahler metric in the sense of~\cite{EGZ}. 

The maximally degenerate case $d=n$ is the object of the \emph{Kontsevich--Soibelman
  conjecture}~\cite{KS06}\footnote{Essentially the same 
  conjecture was stated independently by Gross--Wilson~\cite{GW00} and
  Todorov.}, which states that $(X_t,D_t^{-2}\om_t)$ (which has
diameter one) converges in the Gromov-Hausdorff sense to the essential
skeleton $\Sk(X)$ endowed with a piecewise smooth metric of
Monge-Amp\`ere type, \ie locally given as the Hessian of a convex
function satisfying a real Monge-Amp\`ere equation. 
This conjecture has been verified for abelian varieties see
e.g.~\cite{OdakaCollapse} but is largely open in general.
The ``mirror'' situation, when one fixes the complex structure and 
degenerates the cohomology class of the Ricci-flat K\"ahler metric
(along a line segment in the K\"ahler cone), is better 
understood~\cite{GW00, Tos09, Tos10,GTZ13a,GTZ13b,HT14,TWY14}.
By performing a ``hyper-K\"ahler rotation'', this implies a version of 
the Kontsevich--Soibelman conjecture for special cases of Type~III 
degenerations of K3 surfaces~\cite{GW00}. 

\smallskip
Theorems~A and~C indicate a possible approach to the 
Kontsevich--Soibelman conjecture. Indeed, recall that the metric 
$\omega_t$ for $t\in\DD^*$ is constructed as the curvature form of a smooth
metric $\phi_t$ on $A_t$, where $\phi_t$ in turn is 
obtained as a solution of the complex Monge-Amp\`ere equation
$\MA(\phi_t)=\mu_t$.

On the central fiber $X^\hyb_0=X_K^\an$ of $X^\hyb$, it was shown 
in~\cite{nama} that there exists a metric on the line bundle
$A_K^\an$, unique up to scaling, solving the 
non-Archimedean Monge-Amp\`ere equation
$\MA(\phi_0)=\mu_0$ (at least when $X$ is defined over an algebraic
curve). It is now tempting to approach the Kontsevich--Soibelman
conjecture by studying the behavior of $\phi_t$ as $t\to0$.
However, this seems to be a delicate issue since there is no a priori
reason why the weak continuity at $t=0$ of $t\mapsto\mu_t$ would imply continuity
of the solutions $t\mapsto\phi_t$.

\medskip
Instead of Calabi-Yau manifolds, it would be interesting to study degenerating 
families $X\to\DD^*$ of canonically polarized projective manifolds, where the 
metric on $K_{X_t}$ would be the K\"ahler-Einstein metric or the Bergman metric,
and prove versions of Theorems~A and~C in this context.

\medskip
The paper is organized as follows. 
After recalling various facts in~\S\ref{S313} we define
in~\S\ref{sec:hyb} the hybrid space $\cX^\hyb$ 
associated to an SNC model $\cX$.
The proof of Theorem~A is given in~\S\ref{S314}.
In~\S\ref{S315} we define the space $X^\hyb$ associated to 
a degeneration as an inverse limit of the spaces $\cX^\hyb$,
and prove Corollary~B. 
Various notions of skeleta are defined and studied in~\S\ref{S316},
and in~\S\ref{sec:skeletal} we formalize the notion of a residually 
metrized model of the canonical bundle, and associate to such an
object a positive measure on the relevant Berkovich space.
Degenerations of Calabi--Yau varieties are studied in~\S\ref{S317} where
we prove Theorem~C. In~\S\ref{S312} we study various extensions,
and in Appendix~\ref{S308} we recall the Berkovich analytification
of a scheme over a Banach ring.

\medskip\noindent
\textbf{Acknowledgement}.\
We are very grateful to Johannes Nicaise and Chenyang Xu for explaining 
the behavior of Poincar\'e residues in the present context.
We also thank Vladimir Berkovich, Antoine Chambert-Loir, 
Antoine Ducros and Charles Favre 
for useful comments leading up to this work, Bernard Teissier for
help with the Hironaka flattening theorem, 
and Matt Baker and Valentino Tosatti for comments on a preliminary version
of this manuscript. Finally we thank the referees for useful comments.
Boucksom was supported by the ANR project GRACK\@. 
Jonsson was supported by NSF grants DMS-1266207 and DMS-1600011, 
a grant from the Knut and Alice Wallenberg foundation
and a grant from the United States---Israel 
Binational Science Foundation.
%
%
%
%
\section{Preliminaries}\label{S313}
The goal of this section is to fix conventions and notation for
metrics and measures, and to recall a few basic facts on 
integral affine structures. We also make a few calculations regarding 
tropicalizations that will be useful in the proof of Theorem~A.
%
%
%
%
\subsection{Metrics}
We use additive notation for line bundles and metrics over an analytic space $X$, both in the complex and non-Archimedean setting. This amounts to the following two rules:
\begin{itemize}
\item[(i)] if for $i=1,2$, $\phi_i$ is a metric on a line bundle $L_i$ and $a_i\in\Z$, then $a_1\phi_1+a_2\phi_2$ is a metric on $a_1L_1+a_2L_2$;
\item[(ii)] a metric on the trivial line bundle $\cO_X$ is of the form $|\cdot|e^{-\phi}$ for a function $\phi$ on $X$, and we identify the metric with $\phi$. 
\end{itemize}
If $s$ is a section of a line bundle $L$ on $X$, then $\log|s|$ stands
for the corresponding (possibly singular) metric on $L$ in which $s$
has length 1. 
For any metric $\phi$ on $L$, the above rules imply that $\log|s|-\phi$ is a function on $X$, and
\begin{equation*}
  |s|_\phi:=|s|e^{-\phi}=\exp(\log|s|-\phi)
\end{equation*}
is the pointwise length of $s$ in the metric $\phi$. 

A metric on a $\Q$-line bundle $L$ is a collection $(\phi_m)_m$
of metrics on $mL$, for $m$ sufficiently divisible, such that 
$\phi_{jm}=j\phi_m$. 

The line bundle $\cO_X(D)$ associated to any Cartier divisor 
$D$ on $X$ comes with a canonical singular metric $\phi_D$, 
smooth outside $D$. This fact extends to $\Q$-divisors, 
by interpreting $\phi_D$ as a metric on a $\Q$-line bundle.
In the complex case at least, the curvature current of $\phi_D$, 
correctly normalized, coincides with the integration current on $D$.
%
%
%
%
\subsection{Measures and forms}\label{S310}
Any finite-dimensional real vector space $V$ comes equipped with a
Lebesgue (or Haar) measure $\la$, uniquely defined up to a multiplicative constant.
Any lattice $\Lambda\subset V$ allows us to normalize $\la$ by
$\la(V/\Lambda)=1$.

To any top-dimensional differential form $\omega$ on a $C^\infty$ manifold $X$ is associated a positive measure
$|\omega|$ on $X$. For example, if $\Lambda\subset V$ is a 
lattice as above, $m_1,\dots,m_n$ is a basis of the dual lattice,
then $|dm_1\wedge\dots\wedge dm_n|$ is a Lebesgue measure on $V$
normalized by $\Lambda$. 

If $X$ is a complex manifold of dimension $n$, and $\Omega$ is a section of $K_X$, 
that is, a holomorphic $n$-form, 
we define $|\Omega|^2$ as the positive measure
\begin{equation*}
  |\Omega|^2:=\frac{i^{n^2}}{2^n}|\Omega\wedge\bar\Omega|.
\end{equation*}
The normalization is chosen so that the measure associated to the form $dz=dx+idy$
on $\C$ is Lebesgue measure $|dz|^2=|dx\wedge dy|$ on $\C\simeq\R^2$.

This construction induces a natural bijection between smooth metrics on the canonical bundle $K_X$ and (smooth, positive) volume forms on $X$, which associates to a smooth metric $\psi$ on $K_X$ the volume form $e^{2\psi}$ locally defined by 
\begin{equation*}
  e^{2\psi}
  :=\frac{i^{n^2}|\Omega\wedge\bar\Omega|}{2^n|\Omega|^2_\psi}
  =\frac{|\Omega|^2}{|\Omega|^2 e^{-2\psi}}
\end{equation*}
for any local section $\Omega$ of $K_X$. If $\psi'$ is another metric on $K_X$, then 
$$
e^{2\psi'}=e^{2(\psi'-\psi)} e^{2\psi},
$$
where $e^{2(\psi'-\psi)}$ is the usual exponential of the smooth
function $2(\psi'-\psi)\in C^\infty(X)$. This can be used to make
sense of $e^{2\psi}$ as a positive measure for any (possibly singular)
metric $\psi$ on $K_X$. 
Similarly, $e^{2\p/m}$ is a volume form for every metric $\p$ on
$mK_X$, $m\in\Z$. 

Now assume $(X,B)$ is a \emph{pair} in the sense of the Minimal
Model Program, \ie $X$ is a normal complex space and $B$ is a  (not
necessarily effective) $\Q$-Weil divisor on $X$ such that 
\begin{equation*}
  K_{(X,B)}:=K_X+B
\end{equation*}
is a $\Q$-line bundle. Denote by $\phi_B$ the canonical singular metric on
$B|_{X_\reg}$, viewed as a $\Q$-line bundle. If $\psi$ is smooth
metric on the $\Q$-line bundle $K_{(X,B)}$, then $\psi-\phi_B$ is a smooth
metric on $K_{X_\reg\setminus B}$, and $e^{2(\psi-\phi_B)}$ is thus a
volume form on $X_\reg\setminus B$.\footnote{Here and in what follows,
  we write $X\setminus D$ for the complement of the support of a (not
  necessarily reduced) divisor $D$ in a complex space $X$.} 

A pair $(X,B)$ is \emph{subklt} if for some
(or, equivalently, any) log resolution $\rho\colon X'\to X$ of $(X,B)$, the unique 
$\Q$-divisor $B'$ such that $\rho^*K_{(X,B)}=K_{(X',B')}$ and $\rho_*B'=B$ has 
coefficients $<1$. The pair $(X,B)$ is \emph{klt} if $B$ is further effective.  

\begin{lem}\label{lem:subklt} 
  For any continuous metric $\psi$ on $K_{(X,B)}$, $(X,B)$ is subklt if and only if the 
  measure $e^{2(\psi-\phi_B)}$ has locally finite mass near each point of $X$. 
\end{lem}
\begin{proof} 
  With the above notation it is immediate to check that 
  $$
  \rho^*e^{2(\psi-\phi_B)}=e^{2(\rho^*\psi-\phi_{B'})}.
  $$ 
  We are thus reduced to a log smooth pair $(X',B')$, \ie $X'$ is smooth
  and $B'$ has snc support, and the proof is then trivial. 
\end{proof}
When $(X,B)$ is subklt, we may thus view $e^{2(\psi-\phi_B)}$ as a finite positive 
(Radon) measure on $X$, putting no mass on Zariski closed subsets. 
Such measures are called \emph{adapted} in~\cite{EGZ,BBEGZ}. 
%
%
\subsection{Integral piecewise affine spaces}\label{S301} 
The following discussion roughly follows~\cite[p.59]{KKMS} and~\cite[\S1]{BerkContr}. 

If $P$ is a rational polytope in $\R^n$, that is,
the convex hull of a finite subset of $\Q^n$, 
denote by $M_P\subset C^0(P)$ the finitely generated free abelian group obtained by restricting to $P$ affine functions with
coefficients in $\Z$ (constant term included). Denote by $1_P$ the constant function on $P$ with value $1$, and set 
$$
\vec{M}_P:=M_P/M_P\cap\Q 1_P.
$$ 
Denote also by $b_P\in\N$ the greatest integer such that 
$b_P^{-1}1_P\in M_P$. 

The data of $(P,M_P)$ modulo homeomorphism is called an
\emph{(abstract) $\Z$-polytope}. The functions in $M_P$ are called
\emph{integral affine}, or $\Z$-affine.

The evaluation map defines a canonical realization 
$P\hookrightarrow (M_P)^\vee_\R$ as a codimension one rational polytope, 
with tangent space $T_P$ identified with $(\vec{M}_P)^\vee_\R$. 
Further, the lattice $T_{P,\Z}:=\Hom(\vec{M}_P,\Z)\subset T_P$ 
yields a normalized Lebesgue measure $\la_P$ on $P$.

The main example for us is as follows. 
\begin{lem}\label{lem:integral} Given $b_0,\dots,b_p\in\N^*$, view
$$
\sigma=\left\{w\in\R_+^{p+1}\mid\sum_{i=0}^p b_i w_i=1\right\}
$$ 
as a $\Z$-simplex. Then $b_\sigma=\gcd(b_i)$, and 
$$
\vol(\sigma)=\frac{b_\sigma}{p!\prod_i b_i}.
$$
\end{lem}
\begin{proof} 
  Note that $T_{\sigma,\Z}=\{w\in\Z^{p+1}\mid\sum_i b_i w_i=0\}$. 
  The linear isomorphism $\phi\colon\R^{p+1}\to\R^{p+1}$ given by 
  $\phi(w_j)=(b_j w_j)$ takes $\sigma$ to the standard simplex
  $$
  \sigma'=\{w'\in\R_+^{p+1}\mid\sum_i w'_j=1\},
  $$
  and hence
  $$
  [T_{\sigma',\Z}\colon\phi(T_{\sigma,\Z})]\vol(\sigma)=\vol(\sigma')=\frac{1}{p!}.
  $$ 
  Write $T_{\sigma',\Z}$ as the kernel of $\chi\colon\Z^{p+1}\to\Z$ defined
  by $\chi(w')=\sum_i w'_i$. Then
  $\phi(T_{\sigma,\Z})=\ker\chi\cap\phi(\Z^{p+1})$, $\chi(\phi(\Z^{p+1}))=\gcd(b_i)\Z$, 
  and the exact sequence
  $$
  0\to\frac{\ker\chi}{\ker\chi\cap\phi(\Z^{p+1})}\to
  \frac{\Z^{p+1}}{\phi(\Z^{p+1})}\to\frac{\Z}{\chi(\phi(\Z^{p+1}))}\to 0
  $$
  gives as desired 
  $$
  [T_{\sigma',\Z}\colon\phi(T_{\sigma,\Z})]=\frac{\prod_i b_i}{\gcd(b_i)}.
  $$ 
  Finally, the first assertion is clear. 
\end{proof}
\begin{rmk}\label{R301}
  By setting $w_0=b_0^{-1}(1-\sum_{i=1}^pb_iw_i)$, we can identify $\sigma$
  with the simplex $\sum_1^p b_iw_i\le 1$ in $\R_+^p$. The normalized 
  Lebesgue measure on $\sigma$ is then given by 
  $\la_\sigma=b_\sigma^{-1}|dw_1\wedge\dots\wedge dw_p|$.
\end{rmk}
A \emph{compact rational polyhedron} $K$ in $\R^n$ is a finite union
of rational polytopes $P_i$, which may then be arranged so that
$P_i\cap P_j$ is either empty or a common face of
$P_i$ and $P_j$. We then say that $(P_i)$ is a \emph{subdivision} of
$K$, and call the subdivision \emph{simplicial} if each $P_i$ is a
simplex. A continuous function on $K$ is \emph{integral piecewise affine}
($\Z$-PA for short) if $f|_{P_i}\in M_{P_i}$ for some subdivision of
$K$.  These functions form a subgroup $\PA_\Z(K)\subset C^0(K)$, 
and the data of $(K,\PA_\Z(K))$ modulo homeomorphism is called a 
\emph{compact $\Z$-PA space}. 

The \emph{normalized Lebesgue measure} of $K$ is defined as
$$
\la_K=\sum_{\dim P_i=\dim K}{\bf 1}_{P_i}\la_{P_i}
$$
for some (and hence any) subdivision into $\Z$-polytopes. 

Note that a $\Z$-polytope $P$ can be regarded as a $\Z$-PA space
and that $M_P\subset\PA_\Z(P)$.
%
%
\subsection{Tropicalizations and polar coordinates}\label{S302}
The material in this section is surely well known, but we include the
details for lack of a suitable reference. The calculations here are used in the 
proof of Theorem~\ref{thm:genconv} (which implies Theorem~A).

Let $N\simeq\Z^{p+1}$ be a lattice, $M=\Hom(N,\Z)$ the dual lattice,
$\C[M]$ the semigroup ring and $T=\Spec\C[M]=N\otimes\C^*$ the algebraic torus.
A basis for $N$ induces a dual basis
$(m_0,\dots,m_p)$ for $M$ and elements $z_i\in\C[M]$, $0\le i\le p$,
such that $\C[M]=\C[z_0^{\pm1},\dots,z_p^{\pm1}]$ and $T\simeq(\C^*)^{p+1}$.

Let $\Omega\in H^0(T,K_T)$ be the $T$-invariant global section
given in coordinates by 
\begin{equation*}
  \Omega=\frac{dz_0}{z_0}\wedge\dots\wedge\frac{dz_p}{z_p}.
\end{equation*}
Note that $\Omega$ is independent of the choice of coordinates, up
to a sign. Its associated measure 
\begin{equation*}
  \rho:=|\Omega|^2
\end{equation*}
is $T$-invariant, and hence a Haar measure on $T$.

We can write this measure in (logarithmic) polar coordinates via the canonical 
\emph{tropicalization map} $L\colon T\to N_\R$, given in the basis above by 
\begin{equation*}
  L=(-\log|z_0|,\dots,-\log|z_p|).
\end{equation*}
Note that $L$ sits in the exact sequence $1\to K\to T\to N_\R$ obtained by tensoring with $N$ the exact sequence $1\to S^1\to\C^*\to\R\to 0$ induced by $z\mapsto-\log|z|$. In particular, $K=N\otimes S^1\simeq(S^1)^{p+1}$ is a compact torus, and $L\colon T\to N_\R$ is a principal $K$-bundle. 

On the one hand, let $\omega$ be the  translation invariant real $(p+1)$-form on the
tropical torus $N_\R\simeq\R^{p+1}$ given by
\begin{equation*}
  \omega=dm_0\wedge\dots\wedge dm_p.
\end{equation*}
This form is again independent of the choice of basis, up to a sign,
and its associated measure $\la:=|\omega|$ is the Lebesgue (or Haar) measure
on $N_\R$ normalized by $N$.

On the other hand, since $L:T\to N_\R$ is a principal $K$-bundle, each fiber $K_w=L^{-1}(w)$ has a unique $K$-invariant
probability measure $\rho_w$. Then $\rho$ has a fiber decomposition
$$
\rho=(2\pi)^{p+1}\la(dw)\otimes\rho_w,
$$
\ie
\begin{equation}\label{e302}
  \int_Tf\,d\rho
  =(2\pi)^{p+1}\int_{N_\R}\left(\int_{K_w}f\,d\rho_w\right)\la(dw),
\end{equation}
for any $f\in C^0_c(T)$. Concretely, we can use logarithmic polar coordinates on $T$:
\begin{equation*}
  z_j=\exp(-w_j+2\pi i\theta_j)
\end{equation*}
for $0\le j\le p$; then $\rho_w=|d\theta_0\wedge\dots\wedge
d\theta_p|$,
and 
\begin{equation*}
  \rho
  =\left|\frac{dz_0}{z_0}\wedge\dots\wedge\frac{dz_p}{z_p}\right|^2
  =(2\pi)^{p+1}|dw_1\wedge\dots\wedge dw_n|\otimes\rho_w.
\end{equation*}

\medskip
We will need the same analysis on certain subgroups of $T$. 
Fix an element $m\in M$ and let $\chi=\chi^m\colon T\to\C^*$
be the corresponding character. Let $b\in\Z_{>0}$ be the largest
integer such that $b^{-1}m\in M$.
In the bases above, we can write 
$m=\sum_{i=0}^pb_im_i$ and $\chi=\prod_iz_i^{b_i}$, 
where $b_i\in\Z$; then $b=\gcd_ib_i$.
On the other hand, we can pick a basis such that 
$m=bm_0$ and $\chi=z_0^b$. This is useful for computations.

For $t\in\C^*$, $T_t:=\chi^{-1}(t)$ is a complex manifold 
with $b$ connected components. Note that $T':=X_1$ is an algebraic subgroup of $T$
and that $T_t$ is a torsor for $T'$ for any $t\in\C^*$. The $T$-invariant $(p+1)$-form $\Omega$ induces in a canonical way a $T'$-invariant $p$-form $\Omega_t$ on $T_t$, obtained as the restriction to $T_t$ of any choice of holomorphic $p$-form $\Omega'$ on $T$ such that 
$\frac{d\chi}{\chi}\wedge\Omega'=\Omega$. 
In general coordinates as above, we can pick
\begin{equation*}
  \Omega'
  =\frac1{\#J}\sum_{j\in J}\frac{(-1)^j}{b_j}
  \frac{dz_0}{z_0}\wedge\dots\wedge\widehat{\frac{dz_j}{z_j}}
  \wedge\dots\wedge\frac{dz_p}{z_p},
\end{equation*}
where $J=\{j\mid b_j\ne 0\}$. 
In special coordinates, so that $m=bm_0$ and $\chi=z_0^b$,
we then have $\Omega'=\frac1b\frac{dz_1}{z_1}\wedge\dots\wedge\frac{dz_p}{z_p}$, and hence 
\begin{equation*}
  T_t=\bigcup_{u^b=t}\{z_0=u\}
  \qand
  \Omega_t
  =\frac1b\frac{dz_1}{z_1}\wedge\dots\wedge\frac{dz_p}{z_p}
  \bigg|_{T_t}.
\end{equation*}

Note that $\rho_1:=|\Omega_1|^2$ is Haar measure on $T'$, whereas
$\rho_t:=|\Omega_t|^2$ is a $T'$-invariant measure on $T_t$.
In the special case $p=0$, $T_t$ consists of $b$ points, and $\rho_t$
gives mass $\frac1{b^2}$ to each of them.

\smallskip  
Next we study the analogous situation in the tropical torus $N_\R$. 
Viewing $m$ as a linear form on $N_\R$, set $H_s:=m^{-1}(s)$ for $s\in\R$. 
The lattice $N'=\Ker m\subset N$ defines an integral affine 
structure on $H_s$, and hence a normalized Lebesgue measure $\la_s$. Note that
\begin{equation*}
  |\omega'|_{H_s}|=\frac1b\la_s
\end{equation*}
for any choice of $p$-form $\omega'$ on $N_\R$ such that 
$dm\wedge\omega'=\omega$. In general coordinates, we pick
\begin{equation*}
  \omega'
  =\frac1{\#J}\sum_{j\in J}\frac{(-1)^j}{b_j}
  dm_0\wedge\dots\wedge\widehat{dm_j}
  \wedge\dots\wedge dm_p,
\end{equation*}
where $J=\{j\mid b_j\ne 0\}$. 
In special coordinates, $\omega'=\frac1b dm_1\wedge\dots\wedge dm_p$.

\smallskip  
Finally we describe $\rho_t$ in polar coordinates. 
The tropicalization map $L\colon T\to N_\R$ induces a principal $T'\cap K$-bundle $T_t\to H_s$ with $s=-\log|t|$, and hence an invariant probability measure on $\rho_{t,w}$ on each fiber $K_{t,w}:=T_t\cap K_w$. We claim that 
$$
\rho_t=\frac{(2\pi)^p}{b}\la_s(dw)\otimes\rho_{t,w},
$$
\ie
\begin{equation}\label{e301}
  \int_{T_t}f\,d\rho_t=\frac{(2\pi)^p}{b}\int_{H_s}\left(\int_{K_{t,w}}f\,\rho_{t,w}\right)\la_s(dw),
\end{equation}
for any $f\in C^0_c(T_t)$, where $s=\log|t|^{-1}$.

The proof is essentially the same as that of~\eqref{e302}.
We work in special coordinates, so that $\chi=z_0^b$ and
$m=bm_0$. 
Then $T_t=\{z_0^b=t\}$ has $b$ connected components $T_t^{(l)}$, $1\le
l\le b$, and 
\begin{equation*}
  \rho_t
  =|\Omega_t|^2
  =\frac1{b^2}\left|\frac{dz_1}{z_1}\wedge\dots\wedge\frac{dz_p}{z_p}\right|^2.
\end{equation*}
The restriction of the tropicalization map to $T_t^{(l)}$ amounts to the change of coordinates
$z_j=u_j^{(l)}\exp(-w_j+2\pi i\theta_j)$ for $1\le j\le p$, where the 
$u_j^{(l)}$ are constants with $|u_j^{(l)}|=1$. In these coordinates,
\begin{equation*}
  \rho_t|_{T_t^{(l)}}
  =\frac{(2\pi)^p}{b^2}|dm_1\wedge\dots\wedge dm_n|\otimes
  |d\theta_1\wedge\dots\wedge d\theta_p|.
\end{equation*}
Here $\frac1b|d\theta_1\wedge\dots\wedge d\theta_p|$ induces the
measure $\rho_{t,w}$ on $K_{t,w}$, whereas 
$|dm_1\wedge\dots\wedge dm_s|$ is Lebesgue measure $\la_s$ on $H_s$.
Hence~\eqref{e301} follows.
%
%
%
%
\section{The hybrid space associated to an snc model}\label{sec:hyb} 
In this section, we show how to perform a topological surgery in a complex manifold, replacing a simple normal crossing divisor with its dual complex. Our construction is similar to the one used by Morgan-Shalen in~\cite[\S I.3]{MS}, and can even be traced back to the pioneering work of Bergman~\cite{Berg}. 
%
%
%
%
\subsection{The dual complex}\label{sec:dual}
 Let $D$ be an effective divisor with simple normal crossing (snc) support in a 
complex manifold $\cX$. By definition, $D=\sum_{i\in I} b_i E_i$ with $b_i\in\N^*$ 
and $(E_i)_{i\in I}$ a finite family of smooth irreducible divisors such that
$$
E_J:=\bigcap_{i\in J} E_i
$$
is either empty or smooth of codimension $|J|$ (with finitely many connected components)
for each $\emptyset\ne J\subset I$. 
A connected component $Y$ of a non-empty $E_J$ is called a
\emph{stratum}. Together with $\cX\setminus D=E_\emptyset$, the locally closed
submanifolds $\mathring{Y}:=Y\setminus\bigcup_{i\in I\setminus J} E_i$
define a partition of $\cX$.

The \emph{dual complex} $\D(D)$ is the simplicial complex\footnote{This is understood in the slightly generalized sense that the intersection of two faces is a \emph{union} of common faces.} defined as follows: to each stratum $Y$ corresponds a simplex 
$$
\sigma_Y=\left\{w\in\R_+^J\mid\sum_{i\in J} b_i w_i=1\right\}, 
$$
and $\sigma_{Y}$ is a face of $\sigma_{Y'}$ if and only if $Y'\subset Y$. 
This description equips $\D(D)$ with an integral affine structure, 
by which we mean a compatible choice of integral affine structures on each 
simplex $\sigma$. This further induces a $\Z$-PA structure on $\D(D)$. 

We write $Y_\sigma$ for the stratum of a face $\sigma$. Each point $\xi\in D$
belongs to $\mathring{Y_\xi}$ for a unique stratum $Y_\xi$, obtained as
the connected component of $E_{J_\xi}$ containing $\xi$, 
with $J_\xi=\{i\in I\mid \xi\in E_i\}$. 
We denote by $\sigma_\xi:=\sigma_{Y_\xi}$ the corresponding face of $\D(D)$. 
%
%
%
%
\subsection{The hybrid topology}\label{sec:hybtop}
Next we define a natural topology on the disjoint union
$$
\cX^\hyb:=(\cX\setminus D)\coprod\D(D).
$$
Consider a connected open set $\cU\subset \cX$ meeting $D$ and 
local coordinates $z=(z_0,\dots,z_n)$ on $\cU$. 
We say that the pair $(\cU,z)$ 
is \emph{adapted} (to $D$) if the following conditions hold:
\begin{itemize}
\item[(i)] 
  if $E_0,\dots,E_p$ are the irreducible components of $D$
  intersecting $\cU$, then
  $\cU\cap E_0\cap\dots\cap E_p$, if nonempty, equals 
  $\cU\cap\mathring{Y}$ 
  for a component $Y$ of $E_0\cap\dots\cap E_p$;
\item[(ii)] 
  $z_i$ is an equation of $E_i\cap \cU$ with $|z_i|<1$, $0\le i\le p$.
\end{itemize}
We call $Y=Y_\cU$ the stratum of $\cU$, and denote by 
$$
\sigma_\cU=\left\{w\in\R^{p+1}\mid\sum_{i=0}^p b_i w_i=1\right\}
$$ 
the corresponding face of $\D(D)$. 
The function $f_{\cU,z}:=\prod_{i=0}^p z_i^{b_i}$ is an equation of $D$
in $\cU$, with $|f_{\cU,z}|<1$, 
and we get a continuous map $\Log_{\cU}\colon \cU\setminus D\to\sigma_Y$ by setting
$$
\Log_{\cU}=\left(\frac{\log|z_i|}{\log|f_\cU|}\right)_{0\le i\le p}.
$$
For any two adapted coordinate charts $(\cU,z)$, $(\cU',z')$, with the same stratum $Y$,
we have $z'_i=u_i z_i$ with $u_i$ nonvanishing on $\cU\cap \cU'$, for
$i=0,\dots,p$ (after a possible reindexing); it follows that
\begin{equation}\label{equ:Log'}
  \Log_{\cU'}=\Log_{\cU}+O\left(\frac{1}{\log|f_{\cU,z}|^{-1}}\right)
\end{equation}
locally uniformly on $\cU\cap \cU'$. We next show how to globalize this construction. 
\begin{prop}\label{prop:globlog} 
  There exists an open neighborhood
  $\cV\subset \cX$ of $D$ and a continuous map 
  $\Log_\cV\colon \cV\setminus D\to\D(D)$ such that for each adapted
  coordinate chart $(\cU,z)$ with $\cU\subset \cV$ we have 
  $\Log_\cV(\cU\setminus D)\subset\sigma_\cU$ and
  \begin{equation}\label{equ:globlog}
    \Log_\cV=\Log_{\cU}+O\left(\frac{1}{\log|f_{\cU,z}|^{-1}}\right)
  \end{equation}
uniformly on compact subsets of $\cU$. 
\end{prop}
This will be accomplished by means of a partition of unity, using the following elementary special case of~\cite[Theorem 5.7]{Cle}.  
\begin{lem}\label{lem:adapted} 
  There exists a family $((\cV_\a,z_\a))_{\a\in A}$ of adapted
  coordinate charts, such that $(\cV_\a)_\a$ forms a locally finite
  covering of $D$ and such that the strata $Y_\a$ of the $\cV_\a$ satisfy
  \begin{equation}\label{equ:interadapted}
    \bigcap_{\b\in B} \cV_\b\ne\emptyset\Longrightarrow\bigcap_{\b\in B} Y_\b\ne\emptyset
  \end{equation}
  for every finite $B\subset A$.
\end{lem}
\begin{proof}[Proof of Proposition~\ref{prop:globlog}] 
  Pick an open
  cover $(\cV_\a)_\a$ as in Lemma~\ref{lem:adapted}, and denote by
  $\Log_\a\colon \cV_\a\setminus D\to\sigma_\a$ the corresponding maps. 
  Set $\cV:=\bigcup_\a \cV_\a$, and pick a partition of unity $(\chi_\a)$
  subordinate to $(\cV_\a)$. 
  We claim that for each $\xi\in \cV$ there exists an open neighborhood
  $W$ of $\xi$ 
  and a face $\sigma_W$ of $\D(D)$ such that 
  $$
  W\cap\supp\chi_\a\ne\emptyset\Longrightarrow\sigma_\a\subset\sigma_W
  $$
  for any $\a\in A$. Indeed, using~\eqref{equ:interadapted} it is easy to see that 
  \begin{equation*}
    W:=\bigcap_{\a\mid \xi\in \cU_\a} \cV_\a\setminus
    \bigcup_{\a\mid\xi\notin\supp\chi_\b}\supp\chi_\b
  \end{equation*}
  satisfies this property. By convexity of $\sigma_W$, it follows that
  $\Log_\cV:=\sum_\a\chi_\a\Log_{\cV_\a}$ is well-defined on 
  $W\setminus D$, and hence yields a continuous map 
  $\Log_\cV\colon \cV\setminus D\to\D(D)$. 
  The last property is a direct consequence of~\eqref{equ:Log'}. 
\end{proof}
We extend the previous map as 
$$
\Log_\cV\colon \cV^\hyb:=(\cV\setminus D)\cup\D(D)\to\D(D)
$$
by setting $\Log_\cV=\id$ on $\D(D)$.
\begin{defi}\label{defi:hybtop} 
  The \emph{hybrid topology} on $\cX^\hyb:=(\cX\setminus D)\cup\D(D)$ 
  is defined as the coarsest topology such that:
  \begin{itemize}
  \item[(i)] 
    $\cX\setminus D\hookrightarrow \cX^\hyb$ is an open embedding;
  \item[(ii)] 
    For every open neighborhood $\cV$ of $D$ in $\cX$, the 
    set $(\cV\setminus D)\cup\D(\cX)$ is open in $\cX^\hyb$;
  \item[(iii)] 
    $\Log_\cV\colon \cV^\hyb\to\D(D)$ is continuous. 
  \end{itemize}
\end{defi}
Using~\eqref{equ:globlog}, this definition is easily seen to be
independent of the choice of map $\Log_\cV$. If $D$ is compact and $K\subset
\cX$ is a compact neighborhood of $D$, then one easily checks that the
corresponding subset $K^\hyb=(K\setminus D)\cup\D(D)$ is compact
(Hausdorff). 
When $D=b_0E_0$ has only one irreducible component, 
$K^\hyb$ is simply the Tychonoff one-point 
compactification of $K\setminus D$. 
\begin{ex}\label{E301}
  Set $\cX=\DD^2$ and $D=E_0+E_1$ the union of the coordinate
  axes, with coordinates $(z_0,z_1)$. Then $\cU=\cX$ is itself an
  adapted coordinate chart. In these coordinates, 
  $\Log_{\cU}\colon \cU\setminus D\to\sigma_\cU$ becomes the
  map $(\DD^*)^2\to[0,1]$ 
  sending $(z_0,z_1)$ to $\log|z_1|/\log|z_0z_1|$. 
  As a consequence, given $\zeta\in\R_+^*$ and $0<\e\ll1$, 
  the closure in $\cX^\hyb$ of the closed subset 
  \begin{equation*}
    F_\e:=\{0<|z_0|,|z_1|\le\e, |z_0|^{\zeta+\e}\le|z_1|\le|z_0|^{\zeta-\e}\}\subset\DD^2
  \end{equation*}
  is given by $\bF_\e=F_\e\cup I_\e$, where
  $I_\e:=\{w\in[0,1]\mid \frac{\zeta-\e}{1+\zeta-\e}\le w\le\frac{\zeta+\e}{1+\zeta+\e}\}$.
  Further, the sets $\bF_\e$, for $0<\e\ll1$ form a basis of closed neighborhoods of 
  the point $\frac{\zeta}{1+\zeta}\in[0,1]$ in $\cX^\hyb$.
  See Figure~\ref{F301}.
\end{ex}
\begin{figure}
 \includegraphics[width=12cm]{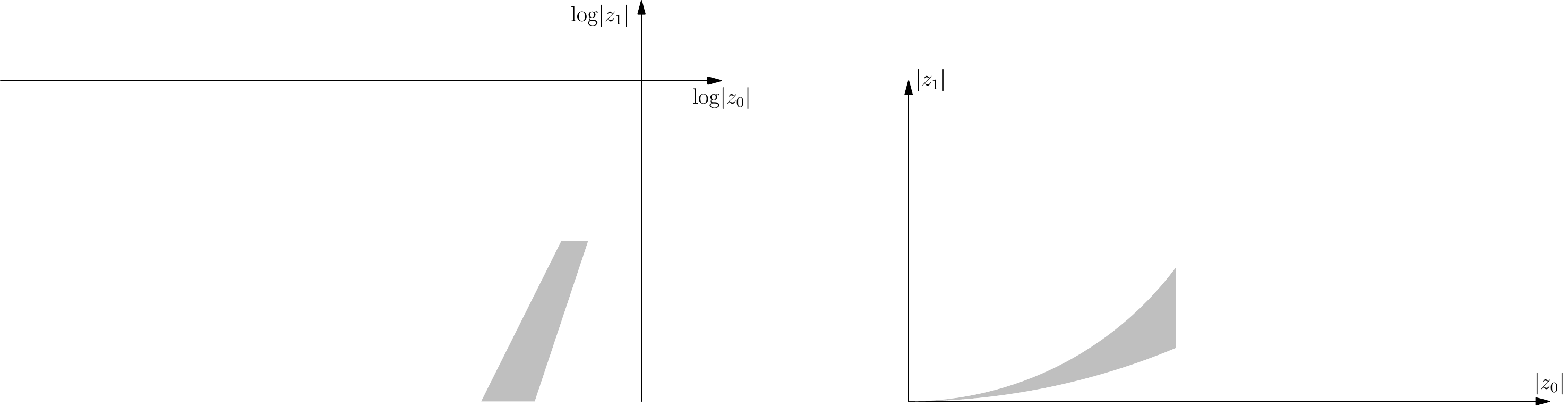}
  \caption{The figure shows the closed subset $F_\e$ in Example~\ref{E301}.}\label{F301}
\end{figure}
%
%
%
%
\section{Proof of Theorem~A}\label{S314}
In this section, we describe in more detail the objects involved in
Theorem A, and then provide a proof. 
We work purely in the complex analytic category here. 
%
%
%
%
\subsection{Residual measures}\label{sec:rescplex}
Let $\pi\colon\cX\to\DD$ be an snc degeneration, \ie a proper, surjective
holomorphic map from a connected complex manifold to the
unit disc in $\C$, whose restriction to $X:=\pi^{-1}(\DD^*)$ is a submersion
and such that $\cX_0:=\pi^{-1}(0)=\sum_{i\in I} b_i E_i$ has snc support. 
Note that $X_t:=\pi^{-1}(t)$ is non-singular for $t\in\DD^*$.
The \emph{dual complex} $\D(\cX)$ is defined as that of $\cX_0$;
it is equipped with its natural $\Z$-PA structure. 
The \emph{logarithmic canonical bundle} of $\cX$ is 
$$
K^\lo_\cX:=K_\cX+\cX_{0,\red}. 
$$
Setting $K^\lo_{\DD}:=K_\DD+[0]$, we define the \emph{relative logarithmic canonical bundle} as
$$
K^\lo_{\cX/\DD}:=K^\lo_\cX-\pi^*K^\lo_{\DD}=K_{\cX/\DD}+\cX_{0,\red}-\cX_0.
$$
Now suppose we are given a $\Q$-line bundle $\cL$ on $\cX$ 
extending $K_{X/\DD^*}$. We then have a unique decomposition
\begin{equation*}
  K^\lo_{\cX/\DD}=\cL+\sum_{i\in I}a_i E_i
\end{equation*}
with $a_i\in\Q$. Set $\kappa_i:=a_i/b_i$ and $\kappa_{\min}:=\min_i\kappa_i$. 

\begin{defi} We denote by $\D(\cL)$ the subcomplex of $\D(\cX)$ such that a face $\sigma$ of $\D(\cX)$ is in $\D(\cL)$ if and only if each vertex of $\sigma$ achieves $\min_i\kappa_i$. 
\end{defi}
In general, $\D(\cL)$ is neither connected nor pure dimensional.
We say that a face of $\D(\cL)$ is \emph{maximal} if it is not 
contained in a larger face of $\D(\cL)$.
\begin{lem}
  Let $Y\subset\cX_0$ be a stratum corresponding to face $\sigma$ of $\D(\cX)$, 
  and denote by $J\subset I$ the set of irreducible 
  components $E_i$ cutting out $Y$. Then 
  \begin{equation*}
    B^\cL_Y:=\sum_{i\notin J}(1-(a_i-\kappa_{\min}b_i))E_i|_Y
  \end{equation*}
  is a $\Q$-divisor on $Y$ with snc support, and we have a canonical identification 
  \begin{equation*}
    \cL|_Y=K_{(Y,B^\cL_Y)}:=K_Y+B^\cL_Y
  \end{equation*}
  as $\Q$-line bundles. If we further assume that $\sigma$ is a
  maximal face of $\D(\cL)$, then $B^\cL_Y$ has coefficients $<1$,
  so the pair $(Y,B^\cL_Y)$ is subklt.
\end{lem}
\begin{proof} 
  The first point is a simple consequence of the triviality of the
  normal bundle 
  $\cO_{\cX_0}(\cX_0)$ together with the adjunction formula 
  \begin{equation*}
    K_Y=(K_\cX+\sum_{i\in J} E_i)|_Y,
  \end{equation*}
  canonically realized by Poincar\'e residues once an order on $J$ has
  been chosen. 
  When $\sigma$ is a maximal face of $\D(\cL)$, each $E_i$ meeting $Y$ 
  properly satisfies $\kappa_i>\kappa_{\min}$, which implies that $B^\cL_Y$
  has coefficients $<1$.
\end{proof}

If $\p$ is a continuous metric on $\cL$, $\p|_Y$ may thus be viewed 
as a metric on $K_{(Y,B^\cL_Y)}$. 
When $\sigma$ is a maximal face of $\Delta(\cL)$,
the pair $(Y,B^\cL_Y)$ is subklt, 
and Lemma~\ref{lem:subklt} applies. 
This leads to the following notion.
\begin{defi} 
  Let $Y$ be a stratum corresponding to a maximal face of
  $\D(\cL)$. The \emph{residual measure} on $Y$ of a continuous metric
  $\p$ on $\cL$ is the (finite) positive measure on $Y$ defined by
  \begin{equation*}
    \Res_Y(\p):=\exp\left(2(\p|_Y-\phi_{B^\cL_Y})\right). 
  \end{equation*}
\end{defi}
This measure can be more explicitly described as follows. 
At each point $\xi\in Y$, pick local coordinates $(z_0,\dots,z_n)$ such
that $z_0,\dots,z_p$ are local equations for the components
$E_0,\dots,E_p$ of $\cX_0$ that pass through $\xi$, 
indexed so that $J=\{0,\dots,d\}$, where $0\le d\le p$,
and such that $t=\prod_{j=0}^pz_j^{b_j}$
The logarithmic form
$$
\Omega:=\frac{dz_0}{z_0}\wedge\dots\wedge\frac{dz_p}{z_p}\wedge dz_{p+1}\wedge\dots\wedge dz_n
$$
is a local trivialization of $K^\lo_{\cX}$, and hence induces a local trivialization 
$\Omega^\rel=\Omega\otimes(dt/t)^{-1}$ of $K^\lo_{\cX/\DD}$. We may then view $\tau:=\prod_{i=0}^p z_i^{a_i}\Omega^\rel$ as a local $\Q$-generator of $\cL$. Under the identification $\cL|_Y=K_{(Y,B^\cL_Y)}$, we have 
\begin{equation*}
  \tau|_Y=\prod_{i=d+1}^p z_i^{a_i-\kappa_{\min} b_i}\Res_Y(\Omega)
\end{equation*}
with
\begin{equation*}
  \Res_Y(\Omega)=\frac{dz_{d+1}}{z_{d+1}}\wedge\dots\wedge\frac{dz_p}{z_p}\wedge dz_{p+1}\wedge\dots\wedge dz_n\bigg|_Y.
\end{equation*}
We infer
\begin{equation}\label{equ:rescplex}
  \Res_Y(\p)=|\tau|^{-2}_\p\prod_{i=d+1}^p|z_i|^{2(a_i-\kappa_{\min} b_i-1)}
  \bigg|\bigwedge_{i=d+1}^ndz_i\bigg|^2.
\end{equation}
%
%
%
%
\subsection{Statement and first reductions}\label{S320}
It will be convenient to introduce the quantity
\begin{equation*}
  \la(t):=(\log|t|^{-1})^{-1},
\end{equation*}
for $t\in\DD^*$. Note that $\la(t)\to0$ as $t\to 0$. 

Let $\cX^\hyb:=X\coprod\D(\cX)$ be the locally compact 
hybrid space constructed in~\S\ref{sec:hyb}. It comes with a
proper map $\pi\colon\cX^\hyb\to\D$ extending 
$\pi\colon X\to\DD^*$ and such that $\D(\cX)=\pi^{-1}(0)$.
The next result implies Theorem A in the introduction. 
\begin{thm}\label{thm:genconv} 
  Let $\pi\colon\cX\to\DD$ be an snc degeneration, 
  $\cL$ a $\Q$-line bundle on $\cX$ extending $K_{X/\DD^*}$, 
  and $\p$ a continuous metric on $\cL$. 
  Define $\kappa_{\min}$ as above, and set $d:=\dim\D(\cL)$. 
  Then, viewed as measures on $\cX^\hyb$,  
  \begin{equation*}
    \mu_t:=\frac{\la(t)^d}{(2\pi)^d|t|^{2\kappa_{\min}}}e^{2\p_t}
  \end{equation*}
  converges weakly to 
  \begin{equation*}
    \mu_0:=\sum_{\sigma}\left(\int_{Y_\sigma}\Res_{Y_\sigma}(\p)\right)b_\sigma^{-1}\la_\sigma, 
  \end{equation*}
  where $\sigma$ ranges over the $d$-dimensional faces of $\D(\cL)$. 
  Here $\la_\sigma$ denotes normalized Lebesgue measure on $\sigma$
  and $b_\sigma=\gcd_{i\in J}b_i$, where $\cX_0=\sum_i b_iE_i$
  and $E_i$, $i\in J$ are the divisors defining $\sigma$.
\end{thm}
We start by making a few reductions.
First, we may---and will---assume in what follows that
$\kappa_{\min}=0$. Indeed, $t$ defines a nonvanishing section of
$\cO_\cX(\cX_0)$, and hence a smooth metric $\log|t|$, so we may replace 
$\cL$ and $\p$ with $\cL-\kappa_{\min}\cX_0$ and
$\p-\kappa_{\min}\log|t|$, respectively, and end up with $\kappa_{\min}=0$.

Since $\min_i a_i/b_i=\kappa_{\min}=0$, we then have $a_i\ge 0$, with
equality if and only if $E_i$ corresponds to a vertex of $\D(\cL)$. 

Next we reduce the assertion of Theorem~\ref{thm:genconv} to a local problem. Let $Y\subset\cX_0$ be the  stratum of an arbitrary face $\sigma$ of $\D(\cX)$, and denote by $E_0,\dots,E_p$ the components of $\cX_0$ cutting out $Y$, ordered so that 
$$
\kappa_0=\dots=\kappa_q<\kappa_{q+1}\le\dots\le\kappa_p.
$$
We can then make the identification
$$
\sigma=\left\{w\in\R_+^{p+1}\mid b\cdot w=1\right\}
$$
with $b=(b_0,\dots,b_p)\in\Z_{>0}^{p+1}$. 
Set $b'=(b_0,\dots,b_q)\in\Z_{>0}^{q+1}$ and 
$$
\sigma':=\left\{w'\in\R_+^{q+1}\mid b'\cdot w'=1\right\}.
$$
Then $\sigma'$ is a face of $\sigma$ under the embedding 
$\R_+^{q+1}\hookrightarrow\R_+^{p+1}$ given by $w'\to(w',0)$.
Let $Y'\supset Y$ be the corresponding stratum of $\cX_0$.

Note that $\sigma$ contains a face of $\D(\cL)$ if and only if $\kappa_0=0$; in that case, the face is unique, equal to $\sigma'$ (which then implies $q\le d$). 

Pick $x\in\mathring{Y}$, and choose local coordinates
$z=(z_0,\dots,z_n)$ at $x$ such that $z_i$ is a local equation of
$E_i$ for $0\le i\le p$ and 
$$
t=\prod_{i=0}^p z_i^{b_i}
$$ 
We may assume that $z$ is defined on a polydisc $\cU\simeq\DD(r)^{p+1}\times\DD^{n-p}$ with $0<r\ll 1$. Decompose
$$
z=(z_0,\dots,z_n)\in\cU\simeq\DD(r)^{p+1}\times\DD^{n-p}
$$ 
as 
$$
z=(z',z'',y)\in\DD(r)^{q+1}\times\DD(r)^{p-q}\times\DD^{n-p},
$$
where we view $y$ as a point of $\cU\cap Y\simeq\DD^{n-p}$, and $(z'',y)$ as
a point of 
$\cU\cap Y'\simeq\DD(r)^{p-q}\times\DD^{n-p}$. 

The coordinate chart $(\cU,z)$ is adapted to $\cX_0$ in the sense of~\S\ref{sec:hybtop}, with 
$$
\Log_{\cU}\colon\cU\setminus\cX_0\to\sigma
$$ 
given by 
$$
\Log_{\cU}=\left(\frac{\log|z_i|}{\log|t|}\right)_{0\le i\le p}. 
$$
We aim to establish the following result. 
\begin{lem}\label{lem:gencv} 
  Pick $\chi\in C^0_c(\cU)$. If $\kappa_0=0$ and $q=d$, then
  \begin{equation*}
    \lim_{t\to0}(\Log_{\cU})_*(\chi \mu_t)
    =\left(\int_{Y'}\chi\Res_{Y'}(\p)\right)b_{\sigma'}^{-1}\la_{\sigma'}
  \end{equation*}
  in the weak topology of measures on $\sigma$, 
  with $\sigma'$ the unique $d$-dimensional face of $\D(\cL)$ 
  contained in $\sigma$. 
  Otherwise (\ie if $\kappa_0>0$ or $q<d$) 
  $(\Log_{\cU})_*(\chi \mu_t)\to 0$. 
\end{lem}
Granted this result, let us show how to prove
Theorem~\ref{thm:genconv}. 
For $0<r\ll1$, $\cV:=\pi^{-1}(\overline{\DD}_r)\subset\cX$ is an
compact neighborhood
of $\cX_0$ with a map $\Log_\cV\colon\cV^\hyb\to\D(\cX)$ as in
Proposition~\ref{prop:globlog}.
We will use
\begin{lem}\label{L306}
  Let $\mu_t$, $t\in\DD_r$ be a family of probability measures on $\cX^\hyb$
  such that $\mu_t$ is supported on $\cX_t$. Then $\lim_{t\to0}\mu_t=\mu_0$ 
  if and only if $\lim_{t\to0}(\Log_\cV)_*\mu_t=\mu_0$.
  Here the limits are in the sense of weak convergence of measures on
  $\cX^\hyb$ and $\D(\cX)$, respectively.
\end{lem}
By Lemma~\ref{L306} we must show that
\begin{equation*}
  (\Log_\cV)_*\mu_t\to\mu_0
  =\sum_{\sigma'}\left(\int_{Y'}\Res_{Y'}(\p)\right)b_{\sigma'}^{-1}\la_{\sigma'},
\end{equation*}
where $\sigma'$ ranges
over $d$-dimensional simplices in $\D(\cL)$.
But this is easily seen to follow from Lemma~\ref{lem:gencv}, 
using a partition of unity argument 
as in the proof of Proposition~\ref{prop:globlog}. 
\begin{proof}[Proof of Lemma~\ref{L306}]
  The direct implication follows from the continuity of $\Log_\cV$. 
  For the reverse implication, assume that 
  $\lim_{t\to0}(\Log_\cV)_*\mu_t=\mu_0$ and 
  consider the following three subsets of $C^0(\cV)$:
  $A_1$ is the set of functions of the form $\Log_\cV^*\f$, 
  where $\f\in C^0(\D(\cX))$;
  $A_2$ is the set of functions of the form 
  $\pi^*g$, where $g\in C^0(\DD_r)$; and
  $A_3=C^0_c(\cV\setminus\D(\cX))$ together with the constant function 1.
  Then the real vector space $A\subset C^0(\cV)$ spanned 
  by functions of the form $f_1f_2f_3$, with $f_i\in A_i$ is 
  easily seen to be an $\R$-algebra that separates points and contains all
  constant functions. By the Stone-Weierstrass Theorem, $A$ is
  dense in $C^0(\cV)$, so it suffices to prove that 
  $\lim\int f\mu_t=\int f\mu_0$ for $f\in A$. By linearity,
  we may assume $f=f_1f_2f_3$ with $f_i\in A_i$. 
  We may further assume $f_3=1$. Write $f_1=\Log_\cV^*\f$ and
  $f_2=\pi^*g$. Then
  \begin{multline*}
    \lim_{t\to0}\int_{X_t}f\mu_t
    =\lim_{t\to0}g(t)\int_{X_t}\f\circ\Log_\cV\mu_t\\
    =\lim_{t\to0}g(t)\int_{\D(\cX)}\f\ (\Log_\cV)_*\mu_t
    =g(0)\int_{\D(\cX)}\f\mu_0
    =\int f\mu_0,
  \end{multline*}
  which completes the proof.
\end{proof}
%
%
%
%
\subsection{Proof of Lemma~\ref{lem:gencv}}\label{S318}
As in~\S\ref{sec:rescplex}, we introduce the logarithmic form
$$
\Omega:=\frac{dz_0}{z_0}\wedge\dots\wedge\frac{dz_p}{z_p}\wedge dz_{p+1}\wedge\dots\wedge dz_n, 
$$
and the corresponding local trivialization $\Omega^\rel=\Omega\otimes(dt/t)^{-1}$ of $K^\lo_{\cX/\DD}$. The restriction $\Omega_t$ of $\Omega^\rel$ to the fiber 
$U_t:=\cX_t\cap\cU$ is a trivializing section of $K_{U_t}$, explicitly given by
\begin{equation*}
  \Omega_t
  =\frac1{p+1}\sum_{j=0}^p\frac{(-1)^j}{b_j}
  \frac{dz_0}{z_0}\wedge\dots\wedge\widehat{\frac{dz_j}{z_j}}
  \wedge\dots\wedge\frac{dz_p}{z_p}
  \wedge dz_{p+1}\wedge\dots\wedge dz_n\bigg|_{U_t}. 
\end{equation*}

For $t\in\DD^*$ close to 0, consider the map
$\Log_t\colon U_t\to\sigma\times(Y\cap\cU)$ 
defined by
\begin{equation*}
  \Log_t
  =(\Log_{\cU},y)
  =\left(\frac{\log|z_0|}{\log|t|},\dots,\frac{\log|z_p|}{\log|t|},z_{p+1},\dots,z_n\right).
\end{equation*}
Note the similarity to the situation considered in~\S\ref{S302}. 
More precisely, view $U:=\cU\cap X$ as embedded in $T\times\C^{n-p}$,
where $T=(\C^*)^{p+1}$, and consider the character
$\chi=\prod_{i=0}^pz_i^{b_i}$ on $T$. 
If $L\colon T\to\R^{p+1}$ is the tropicalization map, then
\begin{equation*}
  \Log_t=(\la(t)^{-1}L(z',z''),y).
\end{equation*}
Each fiber $\Log_t^{-1}(w,y)$ is a torsor for the 
(possibly disconnected) compact Lie group 
\begin{equation*}
  K=\left\{\theta\in(\R/\Z)^{p+1}\mid \sum_i b_i\theta_i=0\right\};
\end{equation*}
hence carries a unique $K$-invariant probability measure $\rho_{t,w,y}$.

The analysis in~\S\ref{S302} now gives the following 
expression for the volume form $|\Omega_t|^2$ on $U_t$
in logarithmic polar coordinates:
\begin{lem}\label{lem:fibermeas} 
For $h\in C^0_c(\cU)$ and $t\in\DD^*$ close to 0, we have
  \begin{equation}\label{equ:omt}
    \int_{U_t}h|\Omega_t|^2
    =(2\pi)^p\la(t)^{-p}\int_{\sigma\times(Y\cap\cU)}
    b_\sigma^{-1}\la_\sigma(dw)\otimes|dy|^2\int_{\Log_t^{-1}(w,y)}h\,\rho_{t,w,y},
  \end{equation}
  where $dy:=dz_{p+1}\wedge\dots\wedge dz_n$.
\end{lem}
As before, view
$\tau:=\prod_{i=0}^p z_i^{a_i}\Omega^\rel$ as a local
$\Q$-generator of $\cL$, and set
\begin{equation*}
  g:=-\log|\tau|_\p\in C^0(\cU).
\end{equation*}
By definition, we have
$\mu_t=(2\pi)^{-d}\la(t)^d|\Omega_t|^2/|\Omega_t|^2_{\psi_t}$, and
hence
  \begin{multline}\label{equ:normmeas}
    (2\pi)^{d-p}\la(t)^{p-d}|t|^{-2\kappa_0}\int\limits_{U_t}h\mu_t\\
    =\int\limits_{\sigma\times(Y\cap\cU)}|t|^{2\sum_{i=q+1}^p b_i w_i(\kappa_i-\kappa_0)}
    b_\sigma^{-1}\la_\sigma(dw)\otimes|dy|^2 
    \int he^{2g}\,\rho_{t,w,y}\\
    =\int\limits_{\sigma\times(Y\cap\cU)}
    e^{-2\la(t)^{-1}\sum_{i=q+1}^p b_i w_i(\kappa_i-\kappa_0)}
    b_\sigma^{-1}\la_\sigma\otimes|dy|^2 
    \int he^{2g}\,\rho_{t,w,y}
  \end{multline}
for every $h\in C^0_c(\cU)$, thanks to Lemma~\ref{lem:fibermeas}. 

We use the following change of variables.
For $t\in\DD^*$, consider the polytope
\begin{equation*}
  \sigma_t:=
  \{(w',x'')\in\R_+^{q+1}\times\R_+^{p-q}\mid 
  b'\cdot w'=1,\ b''\cdot x''\le\la(t)^{-1}\}
  \subset\sigma'\times\R_+^{p-q}
  \subset\R_+^{p+1},
\end{equation*}
where $b'=(b_0,\dots,b_q)$ and $b''=(b_{q+1},\dots,b_p)$.

\begin{lem}\label{L301}
  The continuous map $Q_t\colon\sigma_t\to\sigma$ defined by 
  \begin{equation*}
    Q_t(w',x'')=\left(\left(1-\la(t)b''\cdot x''\right)w',\la(t) x''\right)
  \end{equation*}
  restricts to a homeomorphism between the interior
  of $\sigma_t$ and the interior of $\sigma$.
  Further, its inverse maps the Lebesgue measure 
  $b_\sigma^{-1}\la_\sigma$ on $\sigma$ to the measure
  \begin{equation*}
    (Q_t^{-1})_*b_\sigma^{-1}\la_\sigma
    =\left(1-\la(t) b''\cdot x''\right)^q\la(t)^{p-q} b_{\sigma'}^{-1}\la'_{\sigma'}\otimes|dx''|,
  \end{equation*}
  on $\sigma_t$, where 
  $|dx''|$ is Lebesgue measure on $\R^{p-q}$ normalized by $\Z^{p-q}$.
\end{lem}
\begin{proof}
  The first statement is elementary. To prove the second, we must make
  sure to handle the ``multiplicities'' $b_\sigma$ and $b_{\sigma'}$ correctly.
  Parametrize the interior of $\sigma$ by coordinates 
  $(w_1,\dots,w_p)$ using $w_0=b_0^{-1}(1-\sum_1^pb_iw_i)$.
  By Remark~\ref{R301} we have
  \begin{equation*}
    b_\sigma\la_\sigma=|dw_1\wedge\dots\wedge dw_p|
  \end{equation*}
  Similarly, we parametrize the interiors of $\sigma'$ and $\sigma_t$
  using coordinates $(w_1,\dots,w_q)$ and 
  $(w_1,\dots,w_q,x''_{q+1},\dots,x''_p)$, respectively.
  Then  
  \begin{equation*}
    b_{\sigma'}\la_{\sigma'}=|dw_1\wedge\dots\wedge dw_q|.
  \end{equation*}
  The required formula now follows from an elementary computation.
\end{proof}
Using the map $Q_t$ and the fact that $\kappa_i-\kappa_0>0$
for $i>q$, it is easy to see that 
\begin{equation*}
  \int_\sigma
    e^{-2\la(t)^{-1}\sum_{i=q+1}^p b_i w_i(\kappa_i-\kappa_0)}
    \la_{\sigma}(dw)
  =O(\la(t)^{p-q}).
\end{equation*}
By~\eqref{equ:normmeas}, it follows that
\begin{equation}\label{equ:massK}
  \mu_t(U_t)=O(\la(t)^{d-q}|t|^{2\kappa_0}),  
\end{equation}
and hence $\mu_t(U_t)\to 0$ unless $\kappa_0=0$ and $q=d$, 
which we henceforth assume. Given $\f\in C^0(\sigma)$, our goal is now to show 
\begin{equation}\label{equ:cvut}
  \int_{U_t}(\f\circ\Log_{\cU})\chi\,\mu_t
  \to\left(\int_{\sigma'}\f b_{\sigma'}^{-1}\la_{\sigma'}\right). 
  \left(\int_{Y'}\chi\Res_{Y'}(\p)\right).
\end{equation}
Let us first express both sides of~\eqref{equ:cvut}
in logarithmic polar coordinates. 
We start by the left-hand side.
Set $f:=\chi e^{2g}\in C^0(\cU)$. 
By~\eqref{equ:normmeas} and Lemma~\ref{L301} we have
\begin{multline}\label{equ:normbis}
  (2\pi)^{d-p}\int_{U_t}(\f\circ\Log_{\cU})\chi\,\mu_t\\
  =\la(t)^{d-p}
  \int\limits_{\sigma\times(Y\cap\cU)}
  \f(w)
  e^{-2\la(t)^{-1}a''\cdot w''}
  b_\sigma^{-1}\la_\sigma(dw)\otimes|dy|^2 
  \int f\,\rho_{t,w,y}\\
  =\int\limits_{\sigma'\times\R_+^{p-d}\times(Y\cap\cU)}
  H_t(w',x'')
  b_{\sigma'}^{-1}\la_{\sigma'}(dw')\otimes|dx''|\otimes|dy|^2
  \int f\,\rho_{t,w',x'',y},
\end{multline}
where
\begin{equation*}
  H_t(w',x'')
  =\mathbf{1}_{\sigma_t}
  \f(Q_t(w',x''))e^{-2a''\cdot x''}(1-\la(t)b''\cdot x'')^d,
\end{equation*}
and $\rho_{t,w',x'',y}$ is the same measure as $\rho_{t,w,y}$
via the identification $Q_t(w',x'')=w$.

Note that $\lim_{t\to 0}Q_t(w',x'')=(w',0)$, so 
\begin{equation*}
  \lim_{t\to0}H_t(w',x'')
  =\mathbf{1}_{\sigma'\times\R_+^{p-d}}
  \f(w',0)e^{-2a''\cdot x''}. 
\end{equation*}
Consider the tropicalization map
\begin{equation*}
  S\colon Y'\cap\cU\to\R_+^{p-d}\times(Y\cap\cU)
\end{equation*}
given by $S=(-\log|z_{d+1}|,\dots,-\log|z_p|,y)$.
Each fiber $S^{-1}(x'',y)$ is a torsor for the 
compact torus $(\R/\Z)^{p-d}$ and hence carries a 
unique invariant probability measure $\rho_{x'',y}$.
As $t\to0$, the probability measure $\rho_{t,w',x'',y}$ 
converges weakly to $\rho_{x'',y}$ for any $w'\in\sigma'$.

By dominated convergence it follows that 
\begin{multline}\label{e303}
  \lim_{t\to0}(2\pi)^{d-p}\int_{U_t}(\f\circ\Log_{\cU})\chi\,d\mu_t\\
  =\int_{\sigma'\times\R_+^{p-d}\times(Y\cap\cU)}
  \f(w',0)e^{-2a''\cdot x''}b_{\sigma'}^{-1}\la_{\sigma'}(dw')\otimes|dx''|\otimes|dy|^2
  \int f\,\rho_{x'',y}\\
  =\left(\int_{\sigma'}\f b_{\sigma'}^{-1}\la_{\sigma'}\right)
  \left(\int_{\R_+^{p-d}}e^{-2a''\cdot x''}|dx''|\int_{Y\cap\cU}|dy|^2
    \int f\,\rho_{x'',y}\right).
\end{multline}
It only remains to compare the second factor of~\eqref{e303} 
to the second factor in~\eqref{equ:cvut}.
To this end, we again use logarithmic polar coordinates.
We have 
\begin{equation}\label{equ:rescplexbis}
  \chi\Res_{Y'}(\p)
  =f\prod_{i=d+1}^p|z_i|^{2a_i-2}|dz''|^2\otimes|dy|^2.
\end{equation}
For $d<j\le p$, set $z_j=e^{-x_j+2\pi i\theta_j}$  
with $x''\in\R_+^{p-d}$ and $\theta''\in(\R/\Z)^{p-d}$. 
Then 
\begin{equation}\label{e310}
  \int_{Y'\cap\cU}\chi \Res_{Y'}(\p)
  =(2\pi)^{p-d}\int_{\R_+^{p-d}}e^{-2a''\cdot x''}|dx''|
  \int_{Y\cap\cU}|dy|^2\int f\,\rho_{x'',y},
\end{equation}
which completes the proof of~\eqref{equ:cvut}, and hence of Theorem~\ref{thm:genconv}. 
%
%
%
%
\section{The limit hybrid model}\label{S315}
Let $\pi\colon X\to\DD^*$ be a proper submersion, with $X$ a connected
complex manifold. Assume that $\pi$ is meromorphic over $0\in\DD$
in the sense that it admits a \emph{model} $\pi\colon\cX\to\DD$, that is,
$\cX$ is a normal complex space, $\pi$ is a flat proper map, and we are given an
isomorphism $X\simeq\pi^{-1}(\DD^*)$ over $\DD^*$. We say that $\cX$ is an 
\emph{snc model} (of $X$) if $\cX$ is smooth and the Cartier divisor 
$\cX_0:=\pi^{-1}(0)$ has simple normal crossing support. Such models always exist by 
Hironaka's theorem.

To any snc model $\cX$ we can associate as in~\S\ref{sec:hyb} 
a hybrid space $\cX^\hyb$, that of course depends on $\cX$.
In this section we define a canonical hybrid space $X^\hyb$,
obtained as the inverse limit of the $\cX^\hyb$, that does not have
this defect. We then prove Theorem~B from the introduction.

In the projective case, we show that the both the central fiber
$X^\hyb_0$ and the closed subset $X^\hyb_{\overline{\DD_r}}$ can be viewed
as analytifications in the sense of Berkovich.
%
%
\subsection{Snc models and simple blowups}
Given any two models $\cX$, $\cX'$ of $X$, there is a canonical
bimeromorphic map $\cX'\dashrightarrow\cX$, and we say that $\cX'$
\emph{dominates} $\cX$ if this map is a morphism.
Any two models $\cX$, $\cX'$ are dominated by a third, for instance the 
normalization of the graph of $\cX\dashrightarrow\cX'$.
By Hironaka's theorem, any model is dominated by an snc model. 
Thus the set of models forms a directed set,
in which snc models are cofinal. 

Suppose $\cX$ is an snc model and that 
$\cX'$ is another model that dominates $\cX$ via $\rho\colon\cX'\to\cX$.
As in~\cite[Definition~22]{KS06} we say that $\rho$ is
a \emph{simple blowup} if it is a blowup along a 
smooth, connected complex subspace $W$ of $\cX_0$ meeting transversely (or not at all) every irreducible component of $\cX_0$ that does not contain it. In this case, $\cX'$ is also an snc model.
\begin{lem}\label{L303}
  Suppose $\cX$ and $\cX'$ are snc models, and that $\cX'$ dominates
  $\cX$ via $\rho\colon\cX'\to\cX$. 
  Then there exists a third snc model $\cX''$
  dominating $\cX'$, such that the induced map 
  $\cX''\to\cX$ is a composition of simple blowups.
\end{lem}
We are grateful to Bernard Teissier for help with the following argument.
\begin{proof}
  By Hironaka's version of the Chow theorem (in turn a
  consequence of the flattening theorem),
  see~\cite[Corollary~2]{HirFlat},
  there exists a complex manifold $\cX''$ and a projective
  bimeromorphic morphism $\cX''\to\cX$ such that $\cX''$
  dominates $\cX'$. 
  Since $\cX'\to\cX$ is an isomorphism above $X$, the construction
  in~\cite{HirFlat} further guarantees that $\cX''\to\cX$ is an
  isomorphism above $X$. Indeed, the proof proceeds by blowing up
  well-chosen smooth centers contained in the
  non-flat locus of $\cX'\to\cX$, see D\'efinition~4.4.3~(2) in \loccit
  
  We may therefore assume that $\cX'\to\cX$ itself is projective,
  and more precisely the blowup of an ideal $I$ cosupported on $\cX_0$.
  By the principalization theorem for ideals, there exists a
  projective bimeromorphic morphism $\cX''\to\cX$ that is a 
  composition of simple blowups, such that the pullback of $I$
  to $\cX''$ is a principal ideal, see~\cite[Theorem~3.45]{KollarRes}
  or~\cite[Theorem~2.0.3]{WlodAnalytic}.
  In particular, $\cX''$ dominates $\cX'$.
\end{proof}
%
%
\subsection{Induced maps between dual complexes}\label{S303}
Suppose $\cX'$ and $\cX$ are snc models with $\cX'$ dominating 
$\cX$ via $\rho\colon\cX'\to\cX$. There is then an integral affine map
\begin{equation*}
  r_{\cX\cX'}\colon\Delta(\cX')\to\Delta(\cX),
\end{equation*}
defined as follows.
Consider any simplex $\sigma'$ of $\Delta(\cX')$ and let $Y'$ be the 
corresponding stratum. There exists a unique minimal stratum $Y$
of $\cX_0$ such that $\rho(Y')\subset Y$. Let $\sigma=\sigma_Y$ be
the corresponding simplex. Let $E_i$, $0\le i\le p$ 
(resp.\ $E'_j$, $0\le j\le p'$) be the irreducible
components of $\cX_0$ cutting out $Y$ (resp.\ $Y'$).
Then
\begin{equation*}
  \rho^*E_i=\sum_{j=0}^{p'} a_{ij}E'_j, 
\end{equation*}
for $0\le i\le p$, where $a_{ij}\in\Z_{>0}$.

We can realize the simplex $\sigma$ (resp.\ $\sigma'$)
as the subset $\{\sum_{i=0}^pb_iw_i=1\}\subset\R_+^{p+1}$ 
(resp.\ $\{\sum_{j=0}^{p'}b'_jw'_j=1\}\subset\R_+^{p'+1}$),
where $b_i$ (resp.\ $b'_j$) is the multiplicity of $E_i$ in $\cX_0$
(resp.\ of $E'_j$ in $\cX'_0$).
The restriction of $r_{\cX\cX'}$ to $\sigma'$ is then given by 
\begin{equation}\label{e305}
  w_i=\sum_{j=0}^{p'} a_{ij}w'_j.
\end{equation}
for $0\le i\le p$.
It is clear that $r_{\cX\cX'}$ defines a continuous, integral affine
map from $\D(\cX')$ to $\D(\cX)$. Further, if $\cX$, $\cX'$ and
$\cX''$ are snc models with $\cX''$ dominating $\cX'$, and 
$\cX'$ dominating $\cX$, then $r_{\cX\cX'}\circ r_{\cX'\cX''}=r_{\cX\cX''}$.

In general, it may happen that $\rho(Y')$ is a strict subvariety of $Y$,
and the linear map defining $r_{\cX\cX'}|_{\sigma'}$ could fail to be injective or surjective.
\begin{defi}
  With notation as above, we say that 
  $\sigma'$ is \emph{active} for $r_{\cX\cX'}$ 
  if the restriction $\rho|_{Y'}\colon Y'\to Y$ is a bimeromorphic morphism
  and the $\Q$-linear map defining $r_{\cX\cX'}|_{\sigma'}$ is an isomorphism.
  In this case, $\sigma'$ and $\sigma$ have the same dimension, and 
  $r_{\cX\cX'}$ maps $\sigma'$ homeomorphically 
  onto a $\Z$-subsimplex of $\sigma$ of the same dimension.
\end{defi}
Denote by $A_{\cX\cX'}$ the union of all simplices in 
$\D(\cX')$ that are active for $r_{\cX\cX'}$.
Our goal in this subsection is to prove the following result.
\begin{prop}\label{P301}
  Let $\cX$ and $\cX'$ be snc models, with $\cX'$ dominating $\cX$.
  Then $r_{\cX\cX'}$ maps $A_{\cX\cX'}$ homeomorphically onto $\D(\cX)$.
\end{prop}
\begin{cor}\label{C401}
  The images under $r_{\cX\cX'}$ of the active simplices in $\D(\cX')$
  form a simplicial $\Z$-subdivision of $\D(\cX)$.
  As a consequence, there exists a unique, $\Z$-PA map 
  $i_{\cX'\cX}\colon\D(\cX)\to\D(\cX')$ such that $i_{\cX'\cX}(\D(\cX))=A_{\cX\cX'}$ and
  $r_{\cX\cX'}\circ i_{\cX'\cX}=\id$.
\end{cor}
When $\pi$, $\pi'$ and $\rho$ are projective, one can prove Proposition~\ref{P301} 
using the algebraic tool of valuations. Here we follow an ad hoc approach,
based on Lemma~\ref{L303}.
\begin{lem}\label{L401}
  Suppose $\cX$, $\cX'$ and $\cX''$ are snc models, with 
  $\cX'$ dominating $\cX$ and $\cX''$ dominating $\cX'$.
  Let $\sigma''$ be a simplex of $\D(\cX'')$, and let $\sigma'$ be the 
  smallest simplex of $\D(\cX')$ containing $r_{\cX'\cX''}(\sigma'')$.
  Then $\sigma''$ is active for $r_{\cX\cX''}$ iff $\sigma''$ is active for
  $r_{\cX'\cX''}$ and $\sigma'$ is active for $r_{\cX\cX'}$.
  As a consequence, $A_{\cX\cX''}=A_{\cX'\cX''}\cap r_{\cX'\cX''}^{-1}(A_{\cX\cX'})$.
\end{lem}
\begin{proof}
  To ease notation, set $r':=r_{\cX'\cX''}$ and $r:=r_{\cX\cX'}$.
  Let $\sigma$ be the smallest simplex of $\D(\cX)$ containing $r(\sigma')$.
  Write $Y$, $Y'$ and $Y''$ for the strata of $\cX_0$, $\cX'_0$ and $\cX''_0$
  corresponding to $\sigma$, $\sigma'$ and $\sigma''$, respectively.
  The restrictions $r'|_{\sigma''}\colon\sigma''\to\sigma'$
  and $r_{\sigma'}\colon\sigma'\to\sigma$ are given by $\Q$-linear maps,
  and we have induced morphisms $Y''\to Y'$ and $Y'\to Y$.

  First suppose that $\sigma''$ is active for $r'$ and $\sigma'$ 
  is active for $r$. Then $r'|_{\sigma''}$ and $r|_{\sigma'}$ are given 
  by $\Q$-linear isomorphisms; hence so is the composition 
  $r_{\cX\cX''}|_{\sigma''}$.
  Similarly, the maps $Y''\to Y'$ and $Y'\to Y$ are bimeromorphic morphisms; 
  hence so is the composition $Y''\to Y$.
  It follows that $\sigma''$ is active for $r_{\cX\cX''}$.

  Conversely, suppose $\sigma''$ is active for $r_{\cX\cX''}$. 
  Since the map $Y''\to Y$ is a bimeromorphic morphism, 
  the map $Y''\to Y'$ (resp.\ $Y'\to Y$) must be injective (resp.\ surjective).
  In particular, $\dim Y''\le\dim Y'$ and $\dim Y\le\dim Y'$. 
  Similarly, since the $\Q$-linear map defining 
  $r_{\cX\cX''}|_{\sigma''}=r|_{\sigma'}\circ r'|_{\sigma''}$ is an isomorphism,
  the $\Q$-linear map defining $r'|_{\sigma''}$ (resp.\ $r|_{\sigma'}$) must be
  injective (resp.\ surjective). 
  In particular, $\dim\sigma''\le\dim\sigma'$ and $\dim\sigma\le\dim\sigma'$. 
  Now 
  \begin{equation*}
    \dim Y''+\dim\sigma''=\dim Y'+\dim\sigma'=\dim Y+\dim\sigma=n-1,
  \end{equation*}
  so we infer that 
  $\dim Y''=\dim Y'=\dim Y$ and $\dim\sigma=\dim\sigma'=\dim\sigma''$.
  This further implies that the maps $Y''\to Y'$ and $Y'\to Y$ are 
  bimeromorphic morphisms, and that the $\Q$-linear maps 
  defining $r'|_{\sigma''}$ and $r|_{\sigma'}$ are isomorphisms. 
  Hence $\sigma''$ and $\sigma'$ are active for $r'$ and $r$, respectively.
\end{proof}
\begin{lem}\label{L304}
  Suppose $\cX$, $\cX'$ and $\cX''$ are snc models, with 
  $\cX'$ dominating $\cX$ and $\cX''$ dominating $\cX'$.
  \begin{itemize}
  \item[(a)]
    If $r_{\cX\cX''}\colon A_{\cX\cX''}\to\Delta(\cX)$ is surjective, then 
    so is $r_{\cX\cX'}\colon A_{\cX\cX'}\to\Delta(\cX)$.
  \item[(b)]
    If $r_{\cX\cX''}\colon A_{\cX\cX''}\to\Delta(\cX)$ is injective and 
    $r_{\cX'\cX''}\colon A_{\cX'\cX''}\to\Delta(\cX')$ is surjective, then 
    $r_{\cX\cX'}\colon A_{\cX\cX'}\to\Delta(\cX)$ is injective.
  \item[(c)]
    If $r_{\cX\cX'}\colon A_{\cX\cX'}\to\Delta(\cX)$ and 
    $r_{\cX'\cX''}\colon A_{\cX'\cX''}\to\Delta(\cX')$ are both surjective, then 
    so is $r_{\cX\cX''}\colon A_{\cX\cX''}\to\Delta(\cX)$.
  \item[(d)]
    If $r_{\cX\cX'}\colon A_{\cX\cX'}\to\Delta(\cX)$ and 
    $r_{\cX'\cX''}\colon A_{\cX'\cX''}\to\Delta(\cX')$ are both injective, then 
    so is $r_{\cX\cX''}\colon A_{\cX\cX''}\to\Delta(\cX)$.
  \end{itemize}
\end{lem}
\begin{proof}
  This is formal consequence of the relations
  $r_{\cX\cX''}=r_{\cX\cX'}\circ r_{\cX'\cX''}$ and 
  $A_{\cX\cX''}=A_{\cX'\cX''}\cap r_{\cX'\cX''}^{-1}(A_{\cX\cX'})$.
  For example, let us prove~(a).
  Pick any point $w\in\D(\cX)$. The assumption implies that we can find
  $w''\in A_{\cX\cX''}$ with $r_{\cX\cX''}(w'')=w$. Then $w':=r_{\cX'\cX''}(w'')\in A_{\cX\cX'}$
  and $r_{\cX\cX'}(w')=w$.
  Thus~(a) holds. The proofs of~(b)--(d) are similar and left to the reader.
\end{proof}
\begin{lem}\label{L302}
  The assertions of Proposition~\ref{P301} hold when 
  $\rho$ is a simple blowup.
\end{lem}
\begin{proof}
  This is well known (see~\eg~\cite[p.381]{KS06}) but we supply a proof
  for the convenience of the reader.
  To simplify notation, we set $r:=r_{\cX\cX'}$, 
  $A:=A_{\cX\cX'}$, $\D:=\D(\cX)$ and $\D':=\D(\cX')$.

  Let $W$ be the center of the blowup $\rho$, and 
  $Z$ the smallest stratum of $\cX_0$ containing $W$.
  Let $E_i$, $i\in I$ be the irreducible components of $\cX_0$,
  $J\subset I$ the subset such that $Z$ is an component of $E_J$,
  and $\sigma_Z$ the simplex defined by $Z$. 
  Let $E'_i$, $i\in I$ be the strict transform of $E_i$ to $\cX'$.
  Finally, let $E'$ be the exceptional divisor
  of $\rho$. It corresponds to a vertex $v'=v'_{E'}$ of $\D'$. 

  First assume $W\subsetneq Z$.
  In this case, $\D'$ is obtained from $\D$ by ``raising a tent over the 
  simplex $\sigma_Z$''. Let us be more precise.
  Consider a simplex $\sigma$ of $\D$, corresponding to a
  stratum $Y$ of $\cX_0$. 
  By the definition of a simple blowup, $W$ meets every irreducible component 
  of $\cX_0$ transversely (if at all). It follows that $Y$ cannot be contained in $W$, 
  so $\rho$ is a biholomorphism above a general point of $Y$.
  Thus the strict transform $Y'$ of $Y$ defines a stratum of $\cX'_0$
  as well as a simplex $\sigma'$ of $\D'$, whose vertices
  correspond to the strict transforms of the vertices of $\sigma$.
  In this case, $r$ maps $\sigma'$ onto $\sigma$, 
  and $\rho\colon Y'\to Y$ is a bimeromorphic morphism,
  so $\sigma'$ is active for $r$. 

  This proves that $r\colon A\to\D$ is surjective.
  To prove injectivity, consider a stratum $Y'$ of $\cX'_0$, 
  with corresponding simplex $\sigma'$ of $\D'$. 
  If $Y'$ is not contained in $E'$,
  then $\rho$ is a biholomorphism at the general point of $Y'$,
  $Y:=\rho(Y')$ is a stratum of $\cX_0$ of the same dimension
  as $Y'$, and $Y'$ is the strict transform of $Y$. Thus we are in the
  situation above. 
  On the other hand, if $Y'$ is contained in $E'$, then there exist 
  irreducible components $E_i$, $i\in J$ of $\cX_0$, having 
  strict transforms $E'_i$, $i\in J$, such that $\sigma'$ has 
  $v'$ and $v'_i$, $i\in J$ as vertices.
  Since $W$ is not a stratum of $\cX_0$, the smallest stratum $Y$
  containing $\rho(Y')$ is cut out by $E_i$, $i\in J$.
  It follows that $r$ maps the simplex $\sigma'$ onto the lower-dimensional simplex $\sigma$,
  so $\sigma'$ is not active for $r$. 
  Hence $r\colon A\to\D$ is injective.

  \smallskip
  Now assume $W=Z$ is stratum of $\cX_0$, defining a 
  simplex $\sigma$ with vertices $v_i$, $i\in J$.
  In this case, $\D'$ is obtained from $\D$ by a barycentric subdivision 
  of the simplex $\sigma_Z$. 
  Again, let us be more precise. The same argument as above shows that if $Y$ is
  a stratum of $\cX_0$ that is not contained in $W$, and $Y'$
  is the strict transform, then the simplex $\sigma'_{Y'}$ 
  is active for $r$ and $r(\sigma'_{Y'})=\sigma_Y$.
  Further, $\sigma'_{Y'}$ is the unique simplex in $\cX'_0$ that is active for $r$ and whose
  image under $r$ meets the interior of $\sigma_Y$.

  It remains to consider strata of $\cX_0$ contained in $Z$.
  This becomes a toroidal calculation.
  Let $Y$ be such a stratum, cut out by $E_i$, $i\in K$, where
  $J\subset K$. Then $\rho^{-1}(Y)$ consists of $|J|$ strata
  $Y'_i$, $i\in J$, each cut out by $E'$ and $E'_j$, $j\in K\setminus\{i\}$.
  The restriction $\rho|_{Y'_i}\colon Y'_i\to Y$ is a bimeromorphic morphism,
  and the corresponding simplex $\sigma'_i$ is active for $r$ and maps
  homeomorphically onto a simplex contained in $\sigma_Y$.
  Further, these simplices $r(\sigma'_i)$ have disjoint interiors and 
  cover $\sigma_Y$.
  Finally, if $Y'$ is a stratum of $\cX'_0$ contained in
  $E=\rho^{-1}(Z)$, then $Y=\rho(Y')$ is a stratum contained in $Z$,
  hence $Y'=Y'_i$ is one of the strata above.
  This completes the proof.
\end{proof}
\begin{proof}[Proof of Proposition~\ref{P301}]
  Since $r_{\cX\cX'}$ is continuous, $A_{\cX\cX'}$ is compact, and 
  $\D(\cX)$ is Hausdorff, it suffices to prove that 
  $r_{\cX\cX'}\colon A_{\cX\cX'}\to\D(\cX)$ is bijective.
  
  Using Lemma~\ref{L304}~(c)--(d) and Lemma~\ref{L302}, one proves by 
  induction on the 
  number of blowups that $r_{\cX\cX'}\colon A_{\cX\cX'}\to\D(\cX)$ is bijective when
  $\cX'\to\cX$ is a composition of simple blowups. 

  Now consider the general case. 
  Using Lemma~\ref{L303} we find an snc model $\cX''$ 
  dominating both $\cX$ and $\cX'$ and such that the morphism
  $\cX''\to\cX$ is a composition of simple blowups. 
  Thus $r_{\cX\cX''}\colon A_{\cX\cX''}\to\D(\cX)$ is bijective.
  By Lemma~\ref{L304}~(a), it follows that 
  $r_{\cX\cX'}\colon A_{\cX\cX''}\to\D(\cX)$ is surjective.
  Since $\cX$ and $\cX'$ were arbitrary snc models with $\cX'$ 
  dominating $\cX$, it follows that $r_{\cX'\cX''}\colon A_{\cX'\cX''}\to\D(\cX)$ 
  is also surjective. It now follows from Lemma~\ref{L304}~(b) that
  $r_{\cX\cX'}\colon A_{\cX\cX''}\to\D(\cX)$ is injective, which completes the proof.
\end{proof}
%
%
\subsection{Induced maps between hybrid spaces}\label{S304}
To any snc model $\cX$ of $X$ we associated in~\S\ref{sec:hyb} 
a hybrid space $\cX^\hyb=X\coprod\D(\cX)$. Let us briefly recall
the topology on $\cX^\hyb$ in the present context.
Extend $\pi\colon X\to\DD^*$ to a map
\begin{equation*}
  \pi\colon\cX^\hyb\to\DD
\end{equation*}
by declaring $\pi=0$ on $\D(\cX)$.
For $0<r\le1$, define $\cX_{\DD_r}:=\pi^{-1}(\DD_r)$.  
The construction in~\S\ref{sec:hyb} yields, for $0<r\ll1$, a tropicalization map 
\begin{equation*}
  \log_\cX\colon\cX_{\DD_r}\to\D(\cX)
\end{equation*}
uniquely defined up to an additive error term of size $O((\log|t|)^{-1})$.
The topology on $\cX^\hyb$ is the coarsest one such that 
$\log_\cX$ is continuous, $\pi$ is continuous, and the inclusion
$X\subset\cX^\hyb$ is an open embedding.

Now suppose $\cX'$ and $\cX$ are snc models, with $\cX'$ dominating 
$\cX$ via $\rho\colon\cX'\to\cX$. 
Define the map $\rho^\hyb\colon\cX'^\hyb\to\cX^\hyb$ to be the
identity on $X\subset\cX'$ and equal to the map $r_{\cX\cX'}$ on
$\D(\cX')$ defined in~\S\ref{S303}.
\begin{prop}\label{P302}
  The map $\rho^\hyb$ is continuous and surjective.
  Further, we have 
  \begin{equation}\label{e304}
    \Log_\cX\circ\rho^\hyb=r_{\cX\cX'}\circ\Log_{\cX'}+O((\log|t|)^{-1})
  \end{equation}
  on $X_{\DD^*_r}$ for $0<r\ll1$.
\end{prop}
\begin{proof}
  Surjectivity follows from Proposition~\ref{P301}, and continuity
  from~\eqref{e304} after unwinding the definitions.
  It remains to establish~\eqref{e304}. Consider any point
  $\xi'\in\cX_0$ and set $\xi=\pi(\xi')$.
  We can find adapted coordinate charts $(\cU',z')$ at $\xi'$ on
  $\cX'$ and $(\cU,z)$ at $\xi$ on $\cX$ such that 
  $\rho(\cU')\subset\cU$ and such that the following holds:
  $t=\prod_{i=0}^pz_i^{b_i}$ in $\cU$, 
  $t=\prod_{j=0}^{p'}(z'_j)^{b'_j}$ in $\cU'$ and
  $\rho^*z_i=\prod_j(z'_j)^{a_{ij}}$. 
  Since the map $r_{\cX\cX'}$ is given by~\eqref{e305},
  the result now follows from Proposition~\ref{prop:globlog}.
\end{proof}
%
%
\subsection{The limit hybrid space}\label{S305}
Proposition~\ref{P302} allows us to introduce
\begin{defi}
  The \emph{hybrid space associated to $X$} is the topological space
  \begin{equation*}
    X^\hyb:=\varprojlim_\cX\cX^\hyb,
  \end{equation*}
  where $\cX$ runs over all snc models of $X$.
\end{defi}
Here $X^\hyb$ is equipped with the inverse limit topology. 
The maps $\pi\colon\cX^\hyb\to\DD$ define a continuous and proper map
\begin{equation*}
  \pi\colon X^\hyb\to\DD
\end{equation*}
We can identify $X$ with the open subset $\pi^{-1}(\DD^*)$.
Similarly, the compact subset $X^\hyb_0:=\pi^{-1}(0)$ can 
be identified with $\varprojlim_\cX\D(\cX)$.
For every snc model $\cX$ we have, by the definition of the inverse
limit, a continuous proper map $r_\cX\colon X^\hyb\to\cX^\hyb$.
We also have an embedding $i_\cX\colon\D(\cX)\to X^\hyb$
of $\D(\cX)$ onto a closed subset of $X^\hyb_0$. It satisfies
$r_\cX\circ i_\cX=\id$ on $\D(\cX)$.
\begin{rmk}
  It is not clear how to define a map 
  $\Log\colon X^\hyb\to\varprojlim_\cX\D(\cX)$,
  since each tropicalization map
  $\Log_\cX$ is only defined on $X_{\DD^*(r)}$, where $r=r_\cX$
  depends on $\cX$. 
  See~\S\ref{S307} for a substitute in the projective case.
\end{rmk}
%
%
\subsection{Convergence of measures}\label{S306}
For any locally compact Hausdorff space $Z$, let $\cM(Z)$ denote the space 
of signed Radon measures on $Z$. By definition we have 
$X^\hyb=\varprojlim_\cX\cX^\hyb$, and this induces a 
homeomorphism 
\begin{equation*}
  \cM(X^\hyb)\simto\varprojlim_\cX\cM(\cX^\hyb).
\end{equation*}
Theorem~\ref{thm:genconv} now implies the following result, which is
equivalent to Corollary~B in the introduction.
\begin{cor}
  Let $\pi\colon X\to\DD^*$ be a proper submersion that is meromorphic
  at $0\in\DD$, and let $\p$ be a continuous metric on
  $K_{X/\DD^*}$ with analytic singularities.
  Then there exists a positive measure $\mu_0$ on $X^\hyb_0$ such that if 
  $\mu_t:=\frac{\la(t)^de^{2\p_t}}{|t|^{2\kappa_{\min}}(2\pi)^d}$,
  then $\lim_{t\to0}\mu_t=\mu_0$ in the sense of weak convergence of
  measures on $X^\hyb$.
  Further, there exists a snc model $\cX\to\DD$ and a $\Q$-line bundle
  $\cL$ on $\cX$ extending $K_{X/\DD^*}$ such that $\psi$ extends to a
  smooth metric on $\cL$, and 
  \begin{equation*}
    \mu_0:=\sum_{\sigma}\left(\int_{Y_\sigma}\Res_{Y_\sigma}(\p)\right)b_\sigma^{-1}\la_\sigma, 
  \end{equation*}
  where $\sigma$ ranges over the $d$-dimensional faces of $\D(\cL)$. 
  Here $\la_\sigma$ denotes normalized Lebesgue measure on $\sigma$
  and $b_\sigma=\gcd_{i\in J}b_i$, where $\cX_0=\sum_i b_iE_i$
  and $E_i$, $i\in J$ are the divisors defining $\sigma$.
\end{cor}
%
%
\subsection{The projective case}\label{S307}
Now consider the case when $X\to\DD^*$ is projective.\footnote{In the projective case, the existence of the spaces 
  $\cX^\hyb$ and $X^\hyb$ was observed by Kontsevich and Soibelman,
see~\cite[p.383]{KS06}.}
As we now explain, we can then view $X^\hyb$ and its central fiber as analytic
spaces.

The projectivity assumption 
means that $X$ can be viewed as a smooth subspace $\P^N\times\DD^*$, defined by homogeneous polynomials with
coefficients that are holomorphic functions on $\DD^*$ and 
meromorphic at $0\in\DD$. 

We can view 
these coefficients as complex formal Laurent series, 
that is, elements of the field $K:=\C\lau{\unipar}$. Given $r\in(0,1)$,
this field admits a natural non-Archimedean absolute value that is 
trivial on $\C^*$ and normalized by $|\unipar|=r$. In other words, 
we have $|\sum_ja_j t^j|=r^{\min\{j\mid a_j\ne0\}}$.

Further, the equations defining $X$ now define a smooth projective
variety $X_K$ over the field $K$. To this variety we can associate a
non-Archimedean space $X_K^\an$, namely the Berkovich analytification
of $X_K$ with respect to non-Archimedean norm on $K$.
This is a connected and locally connected compact (Hausdorff) space.

We claim that $X^\hyb_0$ is homeomorphic on $X_K^\an$. 
To see this, we note that, for the same reasons as above, every
projective snc model $\cX\to\DD$ of $X$ defines a projective 
snc model $\cX_R$ of $X_K$ over the valuation ring 
$R=\C\cro{\unipar}$ of $K$. 
Further, the dual complex $\D(\cX)$ of $\cX$ can be identified with 
the dual complex $\D(\cX_R)$ of $\cX_R$. Now, there exists a canonical
retraction map $r_\cX\colon X_K\to\D(\cX_K)$, and we have 
\begin{equation}\label{e306}
  X_K^\an\simto\varprojlim_{\cX\ \text{projective snc}}\D(\cX).
\end{equation}
This was announced in~\cite[Theorem~10, p.383]{KS06}; 
see~\eg~\cite[Corollary~3.2]{siminag} for details.
On the other hand, Lemma~\ref{L303} implies that in 
$X^\hyb_0=\varprojlim_\cX\D(\cX)$, we may take the limit over 
projective snc models. This implies that $X^\hyb_0\simeq X_K^\an$.

\medskip
Next we analyze the space $X^\hyb$ itself, using~Appendix~\ref{S308}. 
Fix $0<r<1$ and consider the Banach ring 
\begin{equation*}
  A_r:=\left\{f=\sum_{\a\in\Z} c_\a\unipar^\a\in\C\lau{\unipar}\ \bigg|\ 
    \|f\|_{\hyb}:=\sum_{\a\in\Z}\|c_\a\|_{\hyb}r^\a<+\infty\right\},
\end{equation*}
where $\|\cdot\|_\hyb$ is the maximum of the usual norm 
and the trivial norm on $\C$.
The Berkovich spectrum $\cM(A_r)$ of $A_r$ is homeomorphic to $\overline{\DD}_r$.

Every function that is holomorphic on $\DD^*$ and meromorphic at
$0\in\DD$ defines an element of $A_r$.
Hence we can define the base change $X_{A_r}\subset\P^N_{A_r}$ using 
the same homogeneous equations as above. 
Then $X_{A_r}$ is a scheme of finite
type over $A_r$, so its analytification $X_{A_r}^\An$ is a compact
Hausdorff space with a continuous map $\pi_r$ onto 
$(\Spec A_r)^\An=\cM(A_r)\simeq\overline{\DD}_r$.
(In Appendix~\ref{S319}, this analytification is denoted by $X^\hyb$,
but here we use $X_{A_r}^\An$ for clarity.)
We have a homeomorphism 
\begin{equation}\label{e307}
  \tau^*\colon\pi_r^{-1}(\overline{\DD}^*_r)
  \simto X_{\overline{\DD}^*_r}
  \simto X^\hyb_{\overline{\DD}^*_r}
\end{equation}
and another homeomorphism 
\begin{equation}\label{e308}
  \tau_0\colon\pi^{-1}(0)\simto X^\hyb_0\simto X_K^\an. 
\end{equation} 
We glue these together to a map $\tau\colon X_{A_r}^\An\to X^\hyb_{\overline{\DD}_r}$.
\begin{prop}
  The map $\tau\colon X_{A_r}^\An\to X^\hyb_{\overline{\DD}_r}$ is
  homeomorphism.
\end{prop}
\begin{proof}
  It follows from~\eqref{e307} and~\eqref{e308} that $\tau$ is a
  bijection.
  Since $X_{A_r}^\An$ is compact and $X_{\overline{\DD}_r}^\hyb$ is
  Hausdorff, it only remains to prove that $\tau$ is continuous.
  It suffices to show that the corresponding map 
  $\tau_\cX\colon X_{A_r}^\An\to\cX^\hyb_{\overline{\DD}_r}$ is continuous for
  a given snc model $\cX$.
  For this, in turn, it suffices to show that $\Log_\cX\circ\tau_\cX$
  is continuous near the central fiber.
  
  Consider a coordinate chart $(\cU,z)$ adapted to $\cX_0$ in the
  sense of~\S\ref{sec:hybtop}. Let $E_0,\dots,E_p$ be the irreducible
  components of $\cX_0$ intersecting $\cU$. 
  Let $\hat\cU\subset X_{A_r}^\An$ be the set of seminorms satisfying 
  $|z_i|<1$ for $0\le i\le p$. Then we have 
  \begin{align*}
    \Log_\cX\circ\tau_\cX
    &=\left(\frac{\log|z_i|_\infty}{\log|t|_\infty}\right)_{0\le i\le p}+O((\log|t|_\infty)^{-1})\\
    &=(\log|z_i|^{-1})_{0\le i\le p}+O((\log|t|_\infty)^{-1})
  \end{align*}
  on $\hat\cU\setminus\pi^{-1}(0)$.
  Now the function $(\log|z_i|^{-1})_i$ is continuous on $\hat\cU$
  with values in the simplex 
  $\sigma=\R_+^{p+1}\cap\{\sum_0^pb_iw_i=1\}\subset\D(\cX)$.
  This completes the proof, since we can cover a neighborhood
  of the central fiber in $X_{A_r}^\An$ with sets of the type $\hat{\cU}$.
\end{proof}
%
%
%
%
\section{Berkovich spaces and skeleta}\label{S316}
Our goal in this section and the next is to study the limit measure $\mu_0$ 
appearing in Corollary~B in more detail. This measure lives on a Berkovich space
and its support has an integral piecewise affine structure.

In this section we undertake a fairly general study of metrics on the 
canonical bundle of a projective variety 
defined over a discretely valued field of residue characteristic zero.
To such a metric is associated a skeleton, a subset of the
underlying Berkovich space. In the setting of Corollary~B, the 
skeleton will be the support of the measure $\mu_0$.

The material here has overlap with~\cite{MN,NX13} and also 
draws on~\cite{TemkinMetric},  but we present some details for the convenience of
the reader.

\medskip
Until further notice, $X$ denotes a smooth, proper, geometrically connected
variety over the field $K:=k\lau{\unipar}$ of formal Laurent series with coefficients in
an algebraically closed field $k$ of characteristic $0$. We set
$n:=\dim X$ and denote by $X^\an$ the Berkovich analytification of $X$
with respect the non-Archimedean absolute value $|\cdot|=r^{\ord_0}$
on $K$, for some fixed $r\in(0,1)$.

While $X^\an$ comes equipped with a structure sheaf, we shall merely
consider it as a topological space. Since $X$ is proper, $X^\an$ is
compact. There is a natural continuous surjective map $X^\an\to X$
such that the preimage of a (scheme) point $\xi\in X$ is identified
with the set of real-valued valuations\footnote{Here we use additive
  terminology; the multiplicative norm associated to $v$ is $r^v$.} 
$v$ on the residue field of $\xi$ satisfying $v|_{k^*}\equiv 0$ and 
$v(\unipar)=1$.
In particular, the preimage $X^{\val}$ of the generic point of $X$ 
consists of real-valued valuations of the function field $F(X)$.
%
%
%
%
\subsection{Models}
Set $S:=\Spec k\cro{\unipar}$. Following the convention of~\cite{MN}, we
define a \emph{model} of $X$ to be a normal separated scheme $\cX$, 
flat and of finite type (but possibly non-proper) over $S$, together
with an identification of the generic fiber of the structure morphism
$\pi\colon\cX\to S$ with $X$. 

For any two models $\cX$, $\cX'$, the identifications of the generic fibers
with $X$ induces a unique birational map $\cX'\dashrightarrow\cX$. 
We say that $\cX'$ \emph{dominates} $\cX$ if this map is a proper morphism.
Any two models can be dominated by a third.

For any model $\cX$ and every irreducible component $E$ of $\cX_0$, 
we set $b_E:=\ord_E(\unipar)$, and view the divisorial valuation 
$$
v_E:=b_E^{-1}\ord_{E}
$$
as an element of $X^{\val}\subset X^\an$. 
The set of such points is a dense subset $X^\div\subset X^\an$. 

We usually denote by $\cX_0=\sum_{i\in I} b_i E_i$ the irreducible
decomposition of the central fiber, and write $E_J:=\bigcap_{i\in J}
E_i$ for $J\subset I$. We say that $\cX$ is \emph{snc} if ($\cX$ is
regular and) $\cX_{0,\red}$ has simple normal crossing support. Since
$k$ has characteristic $0$, this means that each non-empty $E_J$ is
smooth over $k$, of codimension $|J|$ in $\cX$.

More generally, a model $\cX$ is \emph{toroidal} if $\cX\setminus\cX_0\subset\cX$ is a strict toroidal embedding in the sense of~\cite{KKMS}, \ie is formally isomorphic, at each closed point of $\cX_0$, to the inclusion of $\G_{m,k}^{n+1}$ in a toric $k$-variety, and such that each $E_i$ is normal (which then implies that each non-empty $E_J$ is normal).

Every model $\cX$ contains a largest snc Zariski open subset $\cX_\snc\subset\cX$. By Hironaka's theorem, $\cX$ is dominated by an snc model $\cX'$ such that the induced birational morphism $\mu\colon\cX'\to\cX$ is projective, and an isomorphism over $\cX_\snc$. 

If $\cX$ is a model of $X$, the set $\cX^\an\subset X^\an$ of
semivaluations that admit a center (or reduction)
on $\cX_0$ is a closed subset; it can be viewed as the generic fiber
of a suitable formal scheme~\cite[2.2.2]{MN}.
By the valuative criterion of properness, we have $\cX^\an=\cX'^\an$ for
each proper morphism of models $\cX'\to\cX$, and $\cX^\an=X^\an$ if
$\cX$ is proper (over $S$, that is). The reduction map 
$c_\cX\colon\cX^\an\to\cX_0$, taking a semivaluation to
its center, is anticontinuous.\footnote{Anticontinuity means that the inverse image of an open set is closed.}

The set $\cX^\div:=\cX^\an\cap X^\div$ consists of all 
divisorial valuations $v$ on $F(\cX)=F(X)$ that are centered on $\cX_0$, 
trivial on $k$ and such that $v(\unipar)=1$. 
%
%
%
%
\subsection{Model metrics}
If $L$ is a line bundle on $X$, a \emph{model} $\cL$ of $L$ is a
$\Q$-line bundle $\cL$ on a \emph{proper} model $\cX$, together with
an identification $\cL|_X=L$. It defines a \emph{model metric}
$\phi_\cL$ on the Berkovich analytification $L^\an$ of $L$. 
If $\cL'$ is another model of $L$, determined on a proper 
model $\cX'$ of $X$, then $\phi_\cL=\phi_{\cL'}$ 
if and only if the pull-backs of $\cL$ and $\cL'$ to some higher model $\cX''$ coincide. 

A model of $\cO_X$ is given by a $\Q$-Cartier divisor $D$ supported on
the central fiber of a proper model $\cX$; the corresponding model
metric will then be identified with the \emph{model function} $\phi_D\colon
X^\an\to\R$ defined by $\phi_D(v)=v(D)$. It satisfies
\begin{equation}\label{equ:minmod}
\inf_{X^\an}\phi_D=\min_E\phi_D(v_E)
\end{equation}
where $E$ runs over the irreducible components of $\cX_0$. 
%
%
%
%
\subsection{Log canonical divisors}
If $\cX$ is a regular model, $\pi\colon\cX\to S$ is a locally complete
intersection morphism, 
so the dualizing sheaf $\om_{\cX/S}$ is a well-defined line bundle
(see~\cite[\S4.1]{MN} for a more detailed discussion). For an
arbitrary (normal) model, we may thus introduce the \emph{relative
  canonical divisor (class)} $K_{\cX/S}$ as the Weil divisor class on
$\cX$ such that $\cO_{\cX_\reg}(K_{\cX/S})=\om_{\cX_\reg/S}$. We then define: 
\begin{itemize} 
\item[(i)] the \emph{canonical divisor} $K_\cX:=K_{\cX/S}+\pi^*K_S$;
\item[(ii)] the \emph{log canonical divisor} $K_\cX^\lo:=K_\cX+\cX_{0,\red}$; 
\item[(iii)] the \emph{relative log canonical divisor}
$$
K^\lo_{\cX/S}:=K^\lo_\cX-\pi^*K^\lo_S=K_{\cX/S}+\cX_{0,\red}-\cX_0. 
$$
\end{itemize}
Note that $K_\cX^\lo$ is $\Q$-Cartier if and only if $K_{\cX/S}^\lo$ is $\Q$-Cartier.
\begin{ex}\label{ex:log} 
Assume that $\cX$ is snc, and write as above $\cX_0=\sum_{i\in I} b_i E_i$. Pick a closed point $\xi\in\cX_0$, and denote by $J=\{0,\dots,p\}\subset I$ the set of components of $\cX_0$ passing through $\xi$. We may choose a regular system of parameters $z_0,\dots,z_n\in\cO_{\cX,\xi}$ such that $z_i$ is a local equation of $E_i$ for $0\le i\le p$, \ie $\unipar=u z_0^{b_0}\dots z_p^{b_p}$ for some unit $u\in\cO^*_{\cX,\xi}$. The logarithmic form 
$$
\Omega:=\frac{dz_0}{z_0}\wedge\dots\wedge\frac{dz_p}{z_p}\wedge dz_{p+1}\wedge\dots\wedge dz_n
$$
is then a local generator of $K^\lo_{\cX}$, and induces a local generator 
$$
\Omega^\rel:=\Omega\otimes(\frac{d\unipar}{\unipar})^{-1}
$$
of $K^\lo_{\cX/S}$.
 \end{ex}

\begin{rmk} When $\cX$ is snc, $\cO_\cX(K^\lo_{\cX/S})$ coincides with the relative logarithmic dualizing sheaf $\om_{\cX^+/S^+}$ of~\cite[(3.2.2)]{NX13}. When $\cX$ is regular, $\cO_\cX(K_\cX)$ is described in~\cite[Appendix A]{dFEM} as the determinant of the locally free sheaf $\Omega'_{\cX/k}\subset\Omega^1_{\cX/k}$ of \emph{special differentials}, corresponding to derivations $D$ of $\cO_\cX$ such that $D(f)=f'(\unipar)d\unipar$ for $f\in k\cro{\unipar}$. 
\end{rmk}
%
%
%
%
\subsection{Log discrepancies}
We refer to~\cite{dFKX},~\cite[\S2.2]{NX13} and~\cite{KNX} for more details and references on what follows. 

Let $\cX$ be a model with $K^\lo_\cX$ $\Q$-Cartier, and recall that $\cX^\div$ denotes the set of divisorial valuations $v$ on $\cX$ such that $v(\unipar)=1$. We define the \emph{log discrepancy} $A_\cX(v)$ as the log discrepancy of $v$ with respect to the pair $(\cX,\cX_{0,\red})$, in the usual sense of the Minimal Model Program. 

The log discrepancy function $A_\cX\colon\cX^\div\to\Q$ is characterized by the following property: if $\cX'$ is a model over $\cX$ with proper birational morphism $\rho\colon\cX'\to\cX$, then 
\begin{equation}\label{equ:logdisc}
K^\lo_{\cX'}=\rho^*K^\lo_\cX+\sum_E b_E A_\cX(v_E)E, 
\end{equation}
with $E$ running over the irreducible components of $\cX'_0$. 

We say that a model $\cX$ is \emph{log canonical} (\emph{lc} for short) 
and \emph{divisorially log terminal} (\emph{dlt}) if the pair $(\cX,\cX_{0,\red})$ 
has the corresponding property, in the sense of the Minimal Model Program. 

Since the generic fiber $X$ is smooth, a model $\cX$ is thus lc
if and only if $K^\lo_\cX$ is $\Q$-Cartier, with log
discrepancy function $A_\cX\colon\cX^\div\to\Q$ taking non-negative values. 
If $\cX$ is lc, then the center $c_\cX(v)\in\cX_0$ of a valuation $v\in\cX^\div$ with $A_\cX(v)=0$ is called an \emph{lc center} of $\cX$, and an lc model $\cX$ is dlt if and only if $\cX_\snc$ contains all lc centers. The irreducible components
of each non-empty $E_J$ are then normal, with generic point contained in 
$\cX_\snc$~\cite[4.16]{KollarBook}.

\begin{ex} 
  Assume that $\dim X=1$, and let $\cX$ be a dlt model. Each
  irreducible component $E_i$ is then a smooth curve. 
  At a point $\xi\in E_i\cap E_j$, $i\ne j$, $\cX$ is snc. 
  At a closed point $\xi\in\mathring{E}_i$, 
  $\cX$ is either regular, or has a cyclic quotient singularity. 
\end{ex}
\begin{ex} If $\cX$ is toroidal, then $\cX$ is lc, and $\cX$ is dlt 
  if and only if it is snc. Following~\cite{dFKX,KNX}, we could say
  that an lc model $\cX$ is \emph{qdlt} (for \emph{quotient of dlt}) 
  if its lc centers are contained in a toroidal open subset $\cU\subset\cX$. 
\end{ex}

\begin{ex}\label{ex:invadj} 
  If $\cX$ is any model such that $\cX_0$ has klt singularities (and hence is reduced), then $\cX$ is dlt, by inversion of adjunction. 
\end{ex}
%
%
%
%
\subsection{The skeleton of a dlt model}\label{sec:skel}
The \emph{dual complex} $\D(\cX)$ of an snc model $\cX$ is defined as the dual complex of the snc divisor $\cX_0=\sum_{i\in I} b_i E_i$, as in \S\ref{sec:dual}. It is equipped with a natural integral affine structure, in which the face $\sigma$ corresponding to a component $Y$ of a non-empty $E_J$ is identified with the simplex
\begin{equation*}
  \sigma=\left\{w\in\R_+^J\mid\sum_{i\in J} b_i w_i=1\right\},
\end{equation*}
in such a way that $M_\sigma=\Z^J$. 

As explained in~\cite[\S3]{siminag} and~\cite[\S3]{MN}, there is a
natural embedding 
\begin{equation*}
  \emb_\cX\colon\D(\cX)\to\cX^\an
\end{equation*}
that takes a point $w\in\sigma$ to the corresponding monomial valuation. 
In particular, the vertex corresponding to $E_i$ is 
sent to the divisorial valuation $v_{E_i}=b_i^{-1}\ord_{E_i}$. 
The value group of a valuation $v=\emb_\cX(w)$, $w\in\sigma$,
is given by 
$$
v(F(X)^*)= M_\sigma(w):=\{f(w)\mid f\in M_\sigma\}.
$$
Further, if $w\in\mathring{\sigma}$, then $Y_\sigma$ is the 
closure of the center of $\emb_\cX(w)$.

The resulting subspace
$\Sk(\cX):=\emb_\cX(\Delta_X)\subset\cX^\an\subset X^\an$ 
is called the \emph{skeleton} of $\cX$. It is naturally a $\Z$-PA space, the $\Z$-PA functions on $\Sk(\cX)$ being precisely the restrictions of model functions $\phi_D$ determined by a Cartier divisor $D$ on some proper modification $\cX'\to\cX$. 

We further have a natural retraction $r_\cX\colon\cX^\an\to\Sk(\cX)$, mapping a valuation $v$ centered on $\cX_0$ to the monomial valuation $r_\cX(v)$ taking the same values on the $E_i$'s. 
These retractions induce a homeomorphism 
$$
X^\an\simto\varprojlim_\cX\Sk(\cX), 
$$ 
where $\cX$ runs over all proper (or projective) snc models, compare~\eqref{e306}.

If $\cX'\to\cX$ is a proper morphism of snc models, then, by~\cite[3.1.7]{MN},
$$
\Sk(\cX)\subset\Sk(\cX')\subset\cX'^\an=\cX^\an,
$$
the first inclusion being $\Z$-PA\@. Further, 
$$
\bigcup_{\cX\ \text{snc}}\Sk(\cX)\subset X^\an
$$
coincides with the set of (quasi)monomial, or Abhyankar, valuations. 

For a dlt model $\cX$, the dual complex $\D(\cX)$ and skeleton
$\Sk(\cX)\subset X^\an$ are simply defined as those of $\cX_\snc$,
cf.~\cite{NX13}. The retraction $r_\cX\colon\cX^\an\to\Sk(\cX)$ can be
defined as above when $\cX$ is $\Q$-factorial, but its existence is
otherwise unclear (at least to us!).  

By~\cite{KKMS}, any toroidal model $\cX$ has a dual complex $\D(\cX)$
endowed with a natural integral affine structure. This dual complex is
canonically realized as a subspace $\Sk(\cX)\subset X^\an$, for
instance by setting $\Sk(\cX):=\Sk(\cX')$ for any toroidal
modification $\cX'\to\cX$ with $\cX'$ snc. Thus $\Sk(\cX)$
is equipped with a $\Z$-PA structure.
%
%
%
%
\subsection{From log discrepancies to Temkin's metric}
As noted in~\cite{valtree,psh,JM} in increasing order of generality,
log discrepancy functions extend in a natural way to Berkovich
spaces. More precisely, let $\cX$ be any model of $X$ 
such that $K^\lo_\cX$ is $\Q$-Cartier, with log discrepancy function $A_\cX\colon\cX^\div\to\Q$. For each snc model $\cX'$ properly dominating $\cX$, a simple computation going back (at least) to~\cite[Lemma 3.11]{KollarPairs} shows the following:
\begin{itemize} 
\item[(i)] the restriction of $A_\cX$ to $\Sk(\cX')$ is $\Z$-affine on each face of $\D(\cX')$; 
\item[(ii)] we have $A_\cX\ge A_\cX\circ r_{\cX'}$, the inequality being strict outside $\Sk(\cX')$. 
\end{itemize}
We may thus extend $A_\cX$ to an lsc function $A_\cX\colon\cX^\an\to[0,+\infty]$ by setting 
\begin{equation}\label{equ:Asup}
A_\cX(v):=\sup_{\cX'} A_\cX(r_{\cX'}(v))
\end{equation}
for any $v\in\cX^\an$. When $\cX$ is dlt, the log discrepancy function $A_\cX$ determines the skeleton as follows. 

\begin{prop}\label{prop:skdlt} If $\cX$ is dlt, then $\Sk(\cX)=\left\{v\in\cX^\an\mid A_\cX(v)=0\right\}$. 
\end{prop}

\begin{lem}\label{lem:lccenter} Assume that $\cX$ is lc, and pick $v\in\cX^\an$ with $A_\cX(v)=0$. Then $c_\cX(v)$ is an lc center of $\cX$. 
\end{lem}
\begin{proof} We claim that, for every sufficiently high snc model $\cX'$ proper over $\cX$, $v':=r_{\cX'}(v)$ and $v$ have the same center on $\cX$. Indeed, the center of $v$ on $\cX'$ is a specialization of that of $r_{\cX'}(v)$, and hence $c_\cX(v)\in\overline{c_\cX(r_{\cX'}(v))}$. On the other hand, we have $\lim_{\cX'} r_{\cX'}(v)=v$. Since $c_\cX\colon\cX^\an\to\cX_0$ is anticontinuous, $c_\cX^{-1}(\overline{\{c_\cX(v)\}})$ is open, and hence contains $v':=r_{\cX'}(v)$ for some snc model $\cX'$ proper over $\cX$. As a result, $c_\cX(v')$ is a specialization of $c_\cX(v)$, and the claim follows. 

By~\eqref{equ:Asup}, we have $A_\cX(v')=0$, and it is thus enough to prove the result for $v'\in\Sk(\cX')$. If $\sigma$ is the unique face of $\D(\cX')$ containing $v'$ in its interior, then $A_\cX\equiv 0$ on $\sigma$, since $A_\cX$ is non-negative and affine on $\sigma$.  For any divisorial point $w$ in the relative interior of $\sigma$, we thus have $A_\cX(w)=0$ and $c_{\cX'}(v')=c_{\cX'}(w)$, which shows that $c_\cX(v')=c_\cX(w)$ is an lc center. 
\end{proof}

\begin{proof}[Proof of Proposition~\ref{prop:skdlt}] When $\cX$ is snc, the result is a direct consequence of (i) and (ii) above. When $\cX$ is dlt, we have by definition
$$
\Sk(\cX)=\Sk(\cX_\snc)\subset\cX_\snc^\an\subset\cX^\an,
$$
and $A_\cX=A_{\cX_\snc}$ on $\cX_\snc^\an$. It is thus enough to show that any $v\in\cX^\an$ with $A_\cX(v)=0$ belongs to $\cX_\snc^\an$, \ie satisfies $c_\cX(v)\in\cX_\snc$. But $c_\cX(v)$ is an lc center by Lemma~\ref{lem:lccenter}, and hence $c_\cX(v)\in\cX_\snc$ by definition of dlt singularities. 
\end{proof} 

Let $\cX$ be a proper model with $K^\lo_{\cX/S}$ $\Q$-Cartier. Viewed as a $\Q$-line bundle, the latter is then a model of $K_X$, and hence defines a model metric 
$\phi_{K^\lo_{\cX/S}}$ on $K_X^\an$. 
Further,~\eqref{equ:logdisc} shows that the lsc metric
\begin{equation}\label{equ:logdiscr}
  A_X:=\phi_{K^\lo_{\cX/S}}+A_\cX
\end{equation}
on $K_X^\an$ is independent of $\cX$. This is a special case of
Temkin's canonical metrization of the canonical bundle~\cite{TemkinMetric}.\footnote{That we obtain Temkin's metric follows from~\cite[Theorem~8.1.2]{TemkinMetric}. Note that Temkin uses multiplicative terminology.}
The \emph{weight function} of~\cite{MN}
associated to a pluricanonical form $\omega\in H^0(X,mK_X)$
is the function $A_X-\frac1m\log|\omega|$ on $X^\an$.
%
%
%
%
\subsection{The skeleton of a metric on $K_X$}
The purpose of this section is to introduce and study a slight generalization of the \emph{Kontsevich--Soibelman skeleton} introduced in~\cite{KS06} and further analyzed in~\cite{MN,NX13}. 
\begin{defi} 
  If $\p$ is a continuous (or usc) metric on $K_X^\an$, 
  set $\kappa:=A_X-\p$ and $\kappa_{\min}:=\inf_{X^\an}\kappa$. 
  The \emph{skeleton} of $\p$ is the compact set
  \begin{equation*}
    \Sk(\p)=\left\{x\in\Xan\mid \kappa(x)=\kappa_{\min}\right\}.
  \end{equation*}
\end{defi}
Note that $\kappa$ is an lsc function $X^\an\to(-\infty,+\infty]$, and
hence achieves its infimum. 

\begin{defi} 
  Let $\cL$ be a model of $K_X$ determined on a proper dlt model $\cX$. 
  We denote by $\D(\cL)$ the subcomplex of $\D(\cX)$ such that a face $\sigma$ 
  of $\D(\cX)$ is in $\D(\cL)$ if and only if each vertex of $\sigma$ achieves 
  $\min_i\kappa(v_i)$ with $\kappa=A_X-\phi_\cL$. 
\end{defi} 
Concretely, the values $\kappa(v_i)$ are computed as follows: we have 
$$
K^\lo_{\cX/S}=\cL+\sum_{i\in I} a_i E_i
$$
with $a_i\in\Q$, and $\kappa(v_{E_i})=a_i/b_i$. Note that each face of
$\D(\cX)$ contains at most one maximal face of $\D(\cL)$. 
\begin{prop}\label{prop:skelmetr} 
  Assume that $\p$ is a model metric on $K_X^\an$, 
  determined by a model $\cL$ of $K_X$ on a proper dlt model $\cX$ of
  $X$. Then $\Sk(\p)\subset\Sk(\cX)$, and $\kappa=A_X-\p$ is affine on 
  each face of $\D(\cX)$. In particular, 
  \begin{equation}\label{equ:skinf}
    \kappa_{\min}=\min_i\kappa(v_i),
  \end{equation}
  where $v_i$ runs over the vertices in $\D(\cX)$, 
  and $\Sk(\p)$ is the subset of $\Sk(\cX)\subset X^\an$ 
  corresponding to the subcomplex $\D(\cL)$ of $\D(\cX)$. 
\end{prop}

\begin{proof} 
  Since the relative log canonical divisor $K^\lo_{\cX/S}$ and $\cL$ are both models of $K_X$, $D:=K^\lo_{\cX/S}-\cL$ is a $\Q$-Cartier divisor supported on $\cX_0$. The corresponding model function $\phi_D$ satisfies
  $\kappa=A_\cX+\phi_D$,  which shows that
  $\kappa|_{\Sk(\cX)}=\phi_D|_{\Sk(\cX)}$ is affine 
  on each face of $\D(\cX)$. Now pick $v\in\Sk(\p)$. By~\eqref{equ:minmod}, we get
\begin{equation*}
  \kappa(v)
  =A_\cX(v)+\phi_D(v)
  \ge\inf\phi_D
  =\min_i\phi_D(v_i)
  =\min_i(A_\cX+\phi_D)(v_i)\ge\inf_{X^\an}\kappa.
\end{equation*}
It follows that $A_\cX(v)=0$, and hence $v\in\Sk(\cX)$, by Proposition~\ref{prop:skdlt}. 
\end{proof}
%
%
\subsection{Residual boundaries}\label{S311} 
The following construction plays a crucial role for the understanding of the limit measure appearing in Corollary~B.

Consider a model metric $\cL$ of $K_X$ defined on a proper dlt model
$\cX$. Following~\S\ref{sec:rescplex} we explain how to associate 
a subklt pair $(Y,B^\cL_Y)$ to each stratum $Y$ of $\cX_0$ 
corresponding to a maximal simplex in $\D(\cL)$.

Let us first recall a few facts about adjunction. When $\cX$ is an snc
model, each stratum $Y$ comes with a boundary $B_Y:=\sum_{i\notin J_Y}
E_i\cap Y$. Here $(Y,B_Y)$ is log smooth, and
\begin{equation}\label{equ:poinca}
  K^\lo_{\cX/S}\big|_Y=K_{(Y,B_Y)}:=K_Y+B_Y, 
\end{equation}
the identification being provided by Poincar\'e residues. When $\cX$ is merely dlt, each stratum $Y$ is normal, and comes with a canonically defined effective $\Q$-divisor $B_Y$ such that $(Y,B_Y)$ is dlt and still satisfies~\eqref{equ:poinca} (cf.~\cite[4.19]{KollarBook}). We have
$$
B_Y=\sum_{i\notin J_Y} E_i\cap Y+B'_Y
$$
where $B'_Y$ is an effective $\Q$-divisor supported in the complement of $\cX_\snc$. 

\begin{ex} For each $i$, $E_i\cap(\cX\setminus\cX_\snc)$ contains finitely many prime divisors $F_{ik}$ of $E_i$. At the generic point of $F_{ik}$, $\cX$ has cyclic quotient singularities, and
$$
B_{E_i}=\sum_{j\ne i}E_j\cap E_i+\sum_k\left(1-\frac{1}{m_{ik}}\right)F_{ik}
$$
with $m_{ik}$ the order of the corresponding cyclic groups, cf.~\cite[3.36.3]{KollarBook}.
\end{ex}
Now let $\p$ be a model metric on $K_X^\an$, determined by a model $\cL$ of $K_X$ 
on a proper dlt model $\cX$ of $X$. Introduce as before the function 
$\kappa:=A_X-\p$ on $X^\an$, and note that the $\Q$-Cartier divisor
$$
D:=K^\lo_{\cX/S}-\cL-\kappa_{\min}\cX_0=\sum_i(\kappa(v_{E_i})-\kappa_{\min})b_i E_i
$$
is effective. 

\begin{lem}\label{lem:boundaries} 
  If $Y$ is a stratum of $\cX_0$ corresponding to a face $\sigma$ of $\D(\cL)$, 
  then $Y\not\subset\supp D$. It follows that the $\Q$-Cartier divisor 
  $$
  B^\cL_Y:=B_Y-D|_Y
  $$ 
  is well-defined, and we have a canonical identification 
  $\cL|_Y=K_{(Y,B^\cL_Y)}$ as $\Q$-line bundles. 
  Further, if $\sigma$ is a maximal face of $\D(\cL)$, 
  then the pair $(Y,B^\cL_Y)$ is subklt. 
\end{lem}
We emphasize that $B^\cL_Y$ is not effective in general. 
\begin{proof} 
  The first two points are clear. When $\sigma$ is a maximal face, 
  each $E_i$ meeting $Y$ satisfies $\kappa(v_{E_i})>\kappa_{\min}$. 
  As a result, $D|_Y$ contains each lc center $E_i\cap Y$ of $(Y,B_Y)$, 
  which yields the last assertion.
\end{proof}
%
%
\subsection{Skeleta and base change}
Now we study how skeleta of snc models and of metrics behave under 
base change. 

For $m\in\Z_{>0}$ consider the Galois extension $K':=k\lau{\unipar^{1/m}}$ of
$K=k\lau{\unipar}$, with Galois group $G=\Z/m\Z$, and set $X'=X_{K'}$. 
Then $G$ acts on $X'^{\an}$ and the canonical map 
$p\colon X'^{\an}\to X^{\an}$ induces a homeomorphism 
\begin{equation*}
  X'^{\an}/G\simto X^{\an}.
\end{equation*}

If $\cX$ is a model of $X$, then its normalized base change yields a model 
$\cX'$ of $X'$ with a finite morphism $\rho\colon\cX'\to\cX$. 
 If $D$ is a $\Q$-divisor on $\cX$ defining a model function $\phi_D$ on $X^\an$, then 
\begin{equation}\label{equ:pullmodel}
  \phi_{\rho^*D}=m p^*\phi_D. 
\end{equation}

When $\cX$ is an snc model, $\cX'$ is toroidal,
by~\cite[pp.98--102]{KKMS}. 
The following rather detailed description will be useful later on. 

\begin{lem}\label{lem:base} 
  We have $p^{-1}(\Sk(\cX))=\Sk(\cX')$. 
  Further, for each face $\sigma$ of $\D(\cX)$, there exist positive
  integers $e_\sigma$, $f_\sigma$ and $g_\sigma$ satisfying 
  \begin{equation*}
    e_\sigma=\frac{m}{\gcd(m,b_\sigma)}
    \qand
    f_\sigma g_\sigma=\gcd(m,b_\sigma)
  \end{equation*}
  and such that the following properties hold:
  $p^{-1}(\sigma)$ is a union of $g_\sigma$ faces $\sigma'_\a$ of $\D(\cX')$, 
  and these are permuted by $G$. For each $\a$:
  \begin{itemize}
  \item[(a)]
      $p$ induces a $\Q$-affine isomorphism $\sigma'_\a\simto\sigma$;
  \item[(b)]
    $p$ induces a generically finite map $Y_{\sigma'_\a}\to Y_\sigma$, 
    of degree $f_\sigma$;
  \item[(c)]
    $mp^*M_\sigma\subset M_{\sigma'_\a}$, and 
    $[M_{\sigma'_\a}:mp^*M_\sigma]=e_\sigma$.
  \end{itemize}
  Furthermore, we have:
  \begin{itemize}
  \item[(i)] 
    $M_{\sigma'_\a}=p^*\left(mM_\sigma+\Z 1_{\sigma}\right)$;
  \item[(ii)] 
    $\vol(\sigma'_\a)=m^{\dim\sigma}\vol(\sigma)$;
  \item[(iii)] 
    $b_{\sigma'_\a}=b_\sigma/\gcd(m,b_\sigma)$. 
  \end{itemize}
\end{lem}

\begin{proof} 
  The proof uses the toroidal theory of~\cite{KKMS} together with 
  elementary ramification theory of valuations~\cite{ZS}.

  Let $\sigma$ be the face of $\D(\cX)$ corresponding to an irreducible component 
  $Y$ of $E_0\cap\dots\cap E_p$. Set $b_i=\ord_{E_i}(\unipar)$. 
  With the identification
  \begin{equation*}
    \sigma=\{w\in\R_+^{p+1}\mid\sum_i b_i w_i=1\},
  \end{equation*}
  the integral affine structure $M_\sigma$ is given 
  by the lattice $\Z^{p+1}$. Note that $b_\sigma=\gcd_i b_i$.

  Given a closed point $\xi\in\mathring{Y}$, 
  we can find local coordinates $z_0,\dots,z_n$ in the formal 
  completion $\widehat\cO_{\cX,\xi}\simeq k\cro{z_0,\dots,z_n}$ 
  such that $\unipar=\prod_{i=0}^p z_i^{b_i}$. 
  A toric computation (cf.~\cite[pp.98--102]{KKMS}) 
  shows that $\xi$ has $\gcd(m,b_\sigma)$ preimages $\xi'_\a$ in 
  $\cX'_0$, with $\cX'$ formally isomorphic, at each $\xi'_\a$, 
  to the product of $\A_k^{n-p}$ with the affine toric $k$-variety 
  corresponding to the cone $\R_+^{p+1}\subset\R^{p+1}$ with lattice 
  $$
  M':=\Z^{p+1}+\Z\left(\frac{b_0}{m},\dots,\frac{b_p}{m}\right). 
  $$
  It follows that $p^{-1}(\sigma)$ is the union of the corresponding
  faces $\sigma'_\a$ of $\D(\cX')$, each isomorphic to 
  $$
  \sigma'=\left\{w'\in\R_+^{p+1}\mid\sum_i b_i w'_i=m\right\},
  $$
  with integral affine structure induced by $M'$. 
  Now $p$ restricts to a homeomorphism 
  $\sigma'_\a\simto\sigma$ given by $w=w'/m$. 
  Thus $M_{\sigma'_\a}=m p^*M_\sigma+\Z 1_{\sigma'_\a}$. 
  This implies~(i), and~(ii)--(iii) easily follow. 

  Now note that 
  \begin{multline*}
    [M_{\sigma'_\a}':mp^*M_\sigma]
    =[mp^*M_\sigma+\Z 1_{\sigma'_\a}:mp^*M_\sigma]\\
    =[\Z^{p+1}+\Z(\frac{b_0}{m},\dots,\frac{b_p}{m}):\Z^{p+1}]
    =\frac{m}{\gcd(m,b_\sigma)}
    =:e_\sigma.
  \end{multline*}

  It remains to analyze the degree $f_\sigma$ of the restriction $Y_{\sigma'_\a}\to Y_\sigma$.
  For this we use ramification theory.
  
  The function field $F(X')=F(X)(\unipar^{1/m})$ is a Galois extension of $F(X)$ of degree $m$, 
  with Galois group $G$. For any valuation $v'\in X'^\val$, we have $v'|_{F(X)}=m p(v')$. 
  
  Let $v\in X^\an$ be a valuation corresponding to a point $w\in\sigma$. 
  Assume $w$ is ``general'' in the sense that $\dim_\Q\sum_{i=0}^p\Q w_i=p$. 
  The point $w$ has $g_\sigma$ preimages $w'_\a$ under $p$, one in each $\sigma'_\a$, 
  and the valuations $v'_\a:=m^{-1}w'_\a$ are all the extensions of $v$ to $F(X')$. 
  Let us compute the residue degree and ramification index of these extensions.

  The residue fields of $v$ and $v'_\a$ are exactly the function fields of $Y$ and $Y'_\a$, 
  respectively, so the residue degree of the extension $v'_\a$ of $v$ is equal to $f_\sigma$.

  The value group $\Gamma_v=v(F(X))$ of $v$ is given by $\Gamma_v=\sum_{i=0}^p\Z w_i$.
  Similarly, the value group of $v'_\a$ is given by 
  $\Gamma_{v'_\a}=\frac1m\Z+\frac1m\sum_{i=0}^p\Z w'_i=\frac1m\Z+\sum_{i=0}^p\Z w_i$.
  It follows that the ramification index of the extension $v'_\a$ of $v$ is given by 
  \begin{equation*}
    [\Gamma_{v'_\a}:\Gamma_v]
    =[\frac1m\Z+\sum_{i=0}^p\Z w_i:\sum_{i=0}^p\Z w_i]
    =\gcd(\Z\cap m\sum_{i=0}^p\Z w_i)
    =\frac{m}{\gcd(m,b_\sigma)}
    =e_\sigma.
  \end{equation*}
  By~\cite[p.77]{ZS} we now have $e_\sigma f_\sigma g_\sigma=m$,
  which completes the proof.
\end{proof}

Next we study skeleta of metrics. Generalizing~\cite[Lemma 4.1.9]{NX13}, we prove: 
\begin{lem}\label{lem:skelmetrbase} 
  Let $\p$ be a continuous metric on $K_X^\an$, 
  $\p'$ the metric on $K_{X'}^\an\simeq p^*K_X^\an$ 
  corresponding to $p^*\p$, and set $\kappa':=A_{X'}-\p'$. 
  Then $\kappa'=m p^*\kappa$. 
  As a consequence, $\Sk(\p')=p^{-1}\Sk(\p)$ and 
  $\kappa'_{\min}=m\kappa_{\min}$.
\end{lem}
\begin{proof} 
  By~\cite[Theorem~7.12]{Gub98} (see also~\cite[Corollary~2.3]{siminag}), 
  model metrics are dense in the set of continuous metrics
  on $K_X^\an$. 
  Hence we may assume $\p$ is a model
  metric. Using~\eqref{equ:Asup}, it is enough to show that
  $\kappa'(v')=m\kappa(p(v'))$ for a divisorial valuation $v'\in
  X'^\div$. Let $\cX$ be an snc model with $p(v')\in\Sk(\cX)$, and
  such that $\p=\phi_\cL$ for a model $\cL$ of $K_X$ on $\cX$. Since
  the normalized base change $\cX'$ of $\cX$ is toroidal, we can
  choose a toroidal modification $\cX''\to\cX'$ with $\cX''$ snc. The
  induced morphism $\rho\colon\cX''\to\cX$ is toroidal; 
  hence it satisfies the log ramification formula
  $$
  mK^\lo_{\cX''/S'}=\rho^*K^\lo_{\cX/S}. 
  $$
  By~\eqref{equ:pullmodel}, we infer 
  $\phi_{K^\lo_{\cX''/S'}}-\p'=p^*(\phi_{K^\lo_{\cX/S}}-\p)$, 
  which gives the desired result since $v'\in\Sk(\cX'')$, 
  $p(v')\in\Sk(\cX)$ imply $A_{\cX''}(v')=A_{\cX}(p(v'))=0$. 
\end{proof}
%
%
%
%
\section{Skeletal measures}\label{sec:skeletal}
From now on, we assume that $k=\C$, and that 
$X$ is a smooth, projective, geometrically connected
variety over the non-Archimedean field $K=\C\lau{\unipar}$. 
Our goal is to construct measures of the types appearing
in Theorem~A and Corollary~B.
%
%
\subsection{Residually metrized models}
As explained above, to any model $\cL$ of a line bundle $L$ on $X$,
defined on a proper dlt model $\cX$ of $X$, 
we can associate a skeleton $\Sk(\cL)\subset\Sk(\cX)\subset X^\an$. 
To produce a measure on $\Sk(\cL)$ we need additional data.
\begin{defi} 
  Let $L$ be a line bundle on $X$. A \emph{residually metrized model}
  of $L$ is a pair $\cL^{\#}=(\cL,\psi_0)$ where $\cL$ is a model of
  $L$, determined on a proper dlt model $\cX$ of $X$, and $\psi_0$ is a
  continuous Hermitian metric on $\cL_0:=\cL|_{\cX_0}$, viewed as a
  holomorphic line bundle over the complex space $\cX_0$. A
  \emph{residually metrized model metric} $\psi^{\#}$ on $L$ 
  is an equivalence class of such pairs, modulo pull-back to a higher model. 
\end{defi}

\begin{ex} If $L$ is trivial, then any choice of trivialization $s\in
  H^0(X,L)$ defines a residually metrized model metric $\psi^{\#}$ 
  on $L$, determined on any model $\cX$ by $\cL=\cO_{\cX}$ 
  and $\psi_0$ the trivial metric on $\cO_{\cX_0}$. 
\end{ex}
%
%
\subsection{Residual measures}\label{S309}
Let $\cL^{\#}=(\cL,\p_0)$ be a residually metrized model of $K_X$,
determined on a proper dlt model $\cX$. If $Y$ is a stratum
corresponding to a top-dimensional face of $\D(\cL)$, 
Lemma~\ref{lem:boundaries} shows that the restriction of $\p_0$ to $\cL|_Y$ 
induces a Hermitian metric $\p_Y$ on $K_{(Y,B^\cL_Y)}:=K_Y+B^\cL_Y$, 
with $(Y,B^\cL_Y)$ subklt. 
By Lemma~\ref{lem:subklt}, we may thus introduce: 
\begin{defi} 
  Let $Y$ be a stratum corresponding to a top-dimensional face of
  $\D(\cL)$. 
  The \emph{residual measure} of $\cL^{\#}$ on $Y$ is the (finite) positive measure 
  $$
  \Res_Y(\cL^{\#}):=\exp\left(2(\p_Y-\phi_{B^\cL_Y})\right). 
  $$ 
\end{defi}
This definition is of course compatible with one
in~\S\ref{sec:rescplex}, and can be more explicitly described as
follows. 
Let $\xi$ be a (closed) point of $Y\cap\cX_\snc$, index the
irreducible components $E_0,\dots,E_p$ passing through $\xi$ so that $Y$ is a component of $\bigcap_{0\le i\le d} E_i$ with $d=\dim\D(\cL)\le p$. In the notation of Example~\ref{ex:log}, the Poincar\'e residue
\begin{equation*}
  \Res_Y(\Omega)
  =\left(\frac{dz_{d+1}}{z_{d+1}}\wedge\dots\wedge\frac{dz_p}{z_p}\wedge 
    dz_{p+1}\wedge\dots\wedge dz_n\right)\bigg|_Y
\end{equation*}
is a generator of $K_{(Y,B_Y)}=K^\lo_{\cX/S}\big|_Y$. 
Setting $a_i:=\kappa(v_{E_i})b_i\in\Q$, we have 
\begin{equation*}
  K^\lo_{\cX/S}=\cL+\sum_i a_i E_i,
\end{equation*}
and we may thus view 
\begin{equation*}
  \tau
  :=\unipar^{\kappa_{\min}}\prod_{i=0}^p z_i^{a_i-\kappa_{\min}b_i}\Omega^\rel
  =\unipar^{\kappa_{\min}}\prod_{i=d+1}^p z_i^{a_i-\kappa_{\min}b_i}\Omega^\rel
\end{equation*}
as a local $\Q$-generator of $\cL$. 
Further, $B^\cL_Y=\sum_{i=d+1}^p(1-(a_i-\kappa_{\min}b_i))E_i|_Y$, 
and $\tau|_Y$ corresponds to  
\begin{equation*}
  \prod_{i=d+1}^p z_i^{a_i-\kappa_{\min} b_i}\Res_Y(\Omega)
\end{equation*}
under the identification $\cL|_Y=K_{(Y,B^\cL_Y)}$. We arrive at
\begin{align}\label{equ:res}
  \Res_Y(\cL^{\#})
  &=\frac{\prod_{i=d+1}^p|z_i|^{2(a_i-\kappa_{\min} b_i)}}
    {\left|\unipar^{\kappa_{\min}}\prod_{i=d+1}^p z_i^{a_i-\kappa_{\min}b_i}\Omega^\rel\right|^2_{\p_0}}
    \left|\Res_Y(\Omega)\right|^2\notag\\
  &=\frac{\prod_{i=d+1}^p|z_i|^{2(a_i-\kappa_{\min} b_i-1)}}
    {\left|\unipar^{\kappa_{\min}}\prod_{i=d+1}^p z_i^{a_i-\kappa_{\min}b_i}\Omega^\rel\right|^2_{\p_0}}
    |dz_{d+1}\wedge\dots\wedge dz_n|^2.
\end{align}
%
%
\subsection{Measures on dual complexes}
We now define measures associated to residually metrized model metrics.
\begin{defi} 
  Let $\cL^{\#}$ be a residually metrized model of $K_X$, 
  determined on a proper dlt model $\cX$ of $X$. 
  To $\cL^{\#}$ we associate a positive measure $\mu_{\cL^{\#}}$ on 
  $\D(\cL)\subset\D(\cX)$ defined by
  \begin{equation*}
    \mu_{\cL^{\#}}
    =\sum_\sigma\left(\int_{Y_\sigma}\Res_{Y_\sigma}(\cL^{\#})\right)
    b_\sigma^{-1}\la_\sigma, 
  \end{equation*}
  where $\sigma$ runs over the top-dimensional faces of $\D(\cL)$. 
\end{defi}
By Lemma~\ref{lem:integral}, we have
$$
\mu_{\cL^{\#}}(\sigma)=\frac{\int_{Y_\sigma}\Res_{Y_\sigma}(\cL^{\#})}{d!\prod_{i\in J} b_i}
$$
for each face $\sigma$ corresponding to a component of some $E_J$. 
%
%
%
%
\subsection{Skeletal mesures on Berkovich spaces} 
Now consider a residually metrized model metric $\p^{\#}$ on $K_X$.
Pick any representative $\cL^{\#}=(\cL,\psi_0)$ for $\p^{\#}$,
where $\cL$ is a model of $K_X$ determined on a proper dlt model
$\cX$ of $X$, and where $\p_0$ is a continuous metric on $\cL_0:=\cL|_{\cX_0}$.
\begin{defi} 
  The \emph{skeletal measure} $\mu_{\p^{\#}}$ is the image of the measure 
  $\mu_{\cL^{\#}}$ under the embedding 
  $\D(\cL)\hookrightarrow X^\an$.
  We view it as a positive measure on $X^\an$, supported on the 
  skeleton $\Sk(\p^{\#}):=\Sk(\phi_\cL)$. 
\end{defi} 
This definition makes sense, in view of the following result.
\begin{lem}\label{L305}
  The skeletal measure $\mu_{\cL^{\#}}$ is independent of the choice
  of representative $\cL^{\#}$ for $\p^{\#}$.
\end{lem}
\begin{proof}
  Let $\cX$, $\cX'$ be proper dlt models of $X$, with $\cX'$
  dominating $\cX$ via a proper birational morphism
  $\rho\colon\cX'\to\cX$. 
  Let $\cL^{\#}=(\cL,\p_0)$ be a residually metrized model of $K_X$
  consisting of a model $\cL$ of $K_X$ determined on $\cX$
  and a continuous metric $\p_0$ on $\cL_0$.
  Set $\cL'=\rho^*\cL$, $\p'_0=\rho^*\p_0$ and
  $\cL'^{\#}=(\cL',\p'_0)$. 
  We must prove that $\mu_{\cL'^{\#}}=\mu_{\cL^{\#}}$.

  Let $\sigma'$ be a top-dimensional face of $\D(\cL')$,
  $Y'$ the associated stratum of $\cX'_0$, $Y$ the 
  minimal stratum of $\cX_0$ containing $\rho(Y')$
  and $\sigma=\sigma_Y$ the associated simplex of $\Delta(\cX)$.
  Then $\sigma$ and $\sigma'$ have the same dimension, and
  if we (somewhat abusively) identify $\sigma$ and $\sigma'$ with their images 
  in $\Sk(\phi_\cL)\subset X^\an$, then $\sigma'$ is a 
  rational subsimplex of $\sigma$.
  It suffices to prove that $\mu_{\cL'^{\#}}(\sigma')=\mu_{\cL^{\#}}(\sigma')$.

  Now $\rho$ restricts to a birational morphism of 
  $Y'\to Y$, so since $\la_{\sigma}|_{\sigma'}=\la_{\sigma'}$ and
  $b_{\sigma}=b_{\sigma'}$,
  it suffices to prove that $\Res_{Y'}(\cL'^{\#})=\rho^*\Res_Y(\cL^{\#})$.  
  But this is formal. Indeed, we have 
  $(\rho|_Y)_*(B_{Y'}^{\cL'})=B_Y^\cL$ and we can identify 
  $(\rho|_Y)^*K_{(Y,B_Y^\cL)}$ with $K_{(Y',B_{Y'}^{\cL'})}$
  in such a way that the restriction of $\p'_0$ to 
  $\cL'|_{Y'}=K_{(Y',B_{Y'}^{\cL'})}$ coincides with the pullback
  under $\rho|_{Y'}$ of the restriction of $\psi_0$ to $\cL|_Y=K_{(Y,B_Y^\cL)}$.
\end{proof}
%
%
\subsection{Behavior under base change}
Fix $m\in\Z_{>0}$. 
As before, denote by $X'$ the base change of $X$ to $K'=\C\lau{\unipar^{1/m}}$, 
with induced map $p\colon X'^\an\to X^\an$. 
\begin{thm}\label{thm:skelmesfinite} 
  Let $\psi^{\#}$ be a residually metrized model metric on $K_X$, 
  and let $\psi'^{\#}$ be its pull-back to $X'$. Then
  \begin{equation*}
    p_*\mu_{\psi'^{\#}}=m^d\mu_{\psi^{\#}}
  \end{equation*}
  with $d=\dim\Sk(\psi^{\#})$. 
\end{thm}
\begin{proof} 
  Pick a representative $\cL^{\#}=(\cL,\p_0)$ of $\p^{\#}$
  such that $\cL$ is defined on a proper snc model $\cX$.
  Let $\cX'$ be the
  normalized base change by $\unipar=\unipar'^m$.
  
  Let $\sigma$ be a $d$-dimensional face of $\D(\cL)$. 
  By Lemma~\ref{lem:base}, $p^{-1}(\sigma)$ is the union of 
  $g_\sigma$ distinct isomorphic faces $\sigma'_\a$ of $\D(\cX)$ such that 
  \begin{equation}\label{equ:bprime}
    b_{\sigma'_\a}=b_\sigma/\gcd(m,b_\sigma)
  \end{equation} 
  \begin{equation}\label{equ:volprime}
    \vol(\sigma'_\a)=m^d\vol(\sigma).
  \end{equation}
  Further, the induced map $Y'_{\sigma'_\a}\to Y$ is generically
  finite, of degree $f_\sigma$ independent of $\a$, and we have 
  $f_\sigma g_\sigma=\gcd(m,b_\sigma)$.
  Pick a toroidal modification $\cX''\to\cX'$ with $\cX''$ snc, 
  denote by $\rho\colon\cX''\to\cX$ the composition, and set
  $\cL'':=\rho^*\cL$.

  Each face $\sigma'_\a$ above is subdivided into 
  simplices $\sigma''_{\a\b}$ of $\D(\cL'')$ of dimension $d$,
  each corresponding to a stratum $Y''_{\a\b}$ of $\cX''_0$,
  and $\rho|_{Y''_{\a\b}}\colon Y''_{\a\b}\to Y$ is generically finite,
  of degree $f_\sigma$.
  Further,~\eqref{equ:bprime} implies that
  \begin{equation}\label{e309}
    b_{\sigma''_{\a\b}}=b_{\sigma'_\a}=b_\sigma/\gcd(m,b_\sigma)
    \quad\text{for all $\a,\b$}.
  \end{equation}
  We shall need the following result:
  \begin{lem}\label{lem:resfinite} 
    With notation as above, we have, for all $\a$, $\b$:
    \begin{equation*}
      \Res_{Y''_{\a\b}}(\cL''^{\#})=\gcd(m,b_\sigma)^{-2}\rho^*\Res_Y(\cL^{\#}).
    \end{equation*}
  \end{lem}
  Grant this result for the moment.
  Lemma~\ref{lem:resfinite} implies
  \begin{equation*}
    \int_{Y''_{\a\b}}\Res_{Y''_{\a\b}}(\cL''^{\#})=f_\sigma\gcd(m,b_\sigma)^{-2}\int_Y\Res_Y(\cL^{\#}),
  \end{equation*}
  and hence 
  \begin{multline*}
    (p_*\mu')(\sigma) 
    =\sum_{\a,\b}\mu'(\sigma''_{\a\b}) 
    =\sum_{\a,\b}\left(\int_{Y''_{\a\b}}\Res_{Y''_{\a\b}}(\cL''^{\#})\right) 
    b_{\sigma''_{\a\b}}^{-1}\vol(\sigma''_{\a\b})\\
    =f_\sigma\gcd(m,b_\sigma)^{-2}\left(\int_Y\Res_Y(\cL^{\#})\right)b_\sigma^{-1}
    \gcd(m,b_\sigma) \sum_\a \vol(\sigma'_\a)\\
    =m^d\left(\int_Y\Res_Y(\cL^{\#})\right)b_\sigma^{-1}\vol(\sigma)=m^d\mu(\sigma),
  \end{multline*}
  thanks to~\eqref{equ:volprime} and~\eqref{e309}.
\end{proof}
\begin{proof}[Proof of Lemma~\ref{lem:resfinite}]
  Pick a closed point $\xi''\in\mathring{Y''}$ and set 
  $\xi=\rho(\xi'')\in\mathring{Y}$. 
  We use the notation at the end of~\S\ref{S309} with $p=d$.
  Namely, pick local coordinates $(z_i)_{0\le i\le n}$ at $\xi$ and
  $(z''_j)_{0\le j\le n}$ at $\xi''$ such that $E_i=\{z_i=0\}$ for $0\le i\le d$ 
  and $E''_j=\{z_j''=0\}$ for $0\le j\le d$. 
  We have $\rho^*z_i=u_i\prod_{j=0}^d(z''_j)^{c_{ij}}$ for $0\le i\le d$, 
  where  $c_{ij}\in\Z_{\ge0}$ and $u_i\in\cO_{\cX'',\xi''}$ is a unit.
  Further, by Lemma~\ref{lem:base}, 
  the matrix $(c_{ij})$ has determinant $\pm e_\sigma$,
  where $e_\sigma=m/\gcd(m,b_\sigma)$.

  Set 
  \begin{equation*}
    \Omega_1:=\frac{dz_0}{z_0}\wedge\dots\wedge\frac{dz_d}{z_d}
    \qand
    \Omega_2:=dz_{d+1}\wedge\dots\wedge dz_n,
  \end{equation*}
  and define $\Omega_1''$, $\Omega''_2$ similarly.
  Then $\Omega:=\Omega_1\wedge\Omega_2$ 
  and $\Omega'':=\Omega''_1\wedge\Omega_2''$ are 
  local $\Q$-generators of $K^\lo_\cX$ and $K^\lo_{\cX''}$
  at $\xi$ and $\xi''$, respectively.
  Further, 
  \begin{equation*}
    \Res_Y(\Omega)=\Omega_2|_Y
    \qand
    \Res_{Y''}(\Omega'')=\Omega''_2|_{Y''}.
  \end{equation*}
  Now
  \begin{equation*}
    \rho^*\Omega_1
    =\pm e_\sigma\Omega''_1
    +\frac1{z_0''\dots z''_d}\tilde\Omega''_1,
  \end{equation*}
  where $\tilde\Omega''_1$ is a regular $(d+1)$-form vanishing at
  $\xi''$,
  and 
  \begin{equation*}
    \rho^*\Omega_2
    =q\Omega''_2+\tilde\Omega''_2,
  \end{equation*}
  where $q\in\cO_{\cX,\xi''}$ and $\tilde\Omega''_2$ is a regular 
  $(n-d)$-form at $\xi''$ satisfying $\Omega''_1\wedge\tilde\Omega''_2=0$.
  On the one hand, this leads to 
  \begin{equation*}
    (\rho|_{Y''})^*\Res_Y(\Omega)
    =(\rho|_{Y''})^*(\Omega_2|_Y)
    =q\Omega''_2|_{Y''}
    =q\Res_{Y''}(\Omega'').
  \end{equation*}
  On the other hand, we also get
  \begin{equation*}
    \rho^*\Omega=\pm qe_\sigma(1+h)\Omega'',
  \end{equation*}
  with $q$ as above and $h\in\cO_{\cX'',\xi''}$ vanishing along $Y''$.

  Define $\Omega^\rel$ and $\Omega^{''\rel}$ by 
  $\frac{d\unipar}{\unipar}\otimes\Omega^\rel=\Omega$ and 
  $\frac{d\unipar'}{\unipar'}\otimes\Omega^{''\rel}=\Omega''$,
  respectively.
  Then
  \begin{equation*}
    m\frac{d\unipar'}{\unipar'}\otimes\rho^*\Omega^\rel
    =\rho^*(\frac{d\unipar}{\unipar})\otimes\rho^*\Omega^\rel
    =\rho^*\Omega
    =\pm qe_\sigma(1+h)\Omega''
    =\pm qe_\sigma(1+h)\frac{d\unipar'}{\unipar'}\otimes\Omega^{''\rel},
  \end{equation*}
  so that
  \begin{equation*}
    \rho^*\Omega^\rel=\pm\frac{e_\sigma}{m}q(1+h)\Omega^{''\rel}.
  \end{equation*}
  As a consequence, 
  \begin{equation*}
    \rho^*|\unipar^{\kappa_{\min}}\Omega^\rel|_{\psi_0}
    =\frac{e_\sigma}{m}|q||(1+h)||(\unipar')^{\kappa'_{\min}}\Omega^{''\rel}|_{\psi'_0}.
  \end{equation*}
  Since $h$ vanishes along $Y''$, this finally leads to 
  \begin{multline*}
    (\rho|_{Y''})^*\Res_Y(\cL^\#)
    =\frac{(\rho|_{Y''})^*|\Res_Y(\Omega)|^2}
    {(\rho|_{Y''})^*|\unipar^{\kappa_{\min}}\Omega^\rel|^2_{\psi_0}}  
    =\frac{|(\rho|_{Y''})^*(\Omega_2|_Y)|^2}
    {\frac{e^2_\sigma}{m^2}|q|^2|(\unipar')^{\kappa'_{\min}}\Omega^{''\rel}|_{\psi'_0}^2}\\
    =\left(\frac{m}{e_\sigma}\right)^2
    \frac{|\Res_{Y''}(\Omega'')|^2}
    {|(\unipar')^{\kappa'_{\min}}\Omega^{''\rel}|_{\psi'_0}^2}
    =\left(\frac{m}{e_\sigma}\right)^2\Res_{Y''}(\cL^{''\#}),
  \end{multline*}
  which completes the proof since $e_\sigma=\frac{m}{\gcd(b_\sigma,m)}$.
\end{proof}
%
%
%
%
\section{The Calabi--Yau case}\label{S317}
As in~\S\ref{sec:skeletal}, we assume that $X$ is a smooth, projective, 
geometrically connected variety over $\C\lau{\unipar}$.
Now we further assume that $K_X$ is trivial.
Pick a trivializing section $\eta\in H^0(X,K_X)$,
and denote by $\log|\eta|$ the associated model metric on $K_X^\an$, 
determined on any model $\cX$ by $\cL=\cO_\cX$, with $\eta$ providing the identification $\cL|_X\simeq K_X$. Denote also by $\log|\eta|^{\#}$ the residually metrized model metric induced by the trivial Hermitian metric $\p_0=0$ on $\cO_{\cX_0}$. 

The function $\kappa:=A_X-\log|\eta|=-\log|\eta|_{A_X}$ 
coincides with the weight function of~\cite{MN,NX13}. 
By definition, the \emph{Kontsevich--Soibelman skeleton} of $X$ is 
$\Sk(X):=\Sk(\log|\eta|^{\#})$. It is indeed independent of the choice of 
$\eta$, since any other trivializing section of $K_X$ is of the form 
$\eta'=f\eta$ with $f\in\C\lau{\unipar}^*$, and hence $\kappa'=\kappa+\ord_0(f)$. 
%
%
\subsection{Topology of the skeleton}
By~\cite[Theorem 4.2.4]{NX13}, the $\Z$-PA-space $\Sk(X)$ is connected, of pure dimension $d$, and is a deformation retract of $X^\an$. Further, 
$\Sk(X)$ is a \emph{pseudomanifold with boundary}, \ie for some (or, equivalently, any) triangulation $\D$ of $\Sk(X)$, we have: 
\begin{itemize}
\item[(a)] Non-branching property: every $(d-1)$-simplex of $\D$ is contained in at most two $d$-simplices 
\item[(b)] Strong connectedness: every pair of $n$-simplices $\sigma$, $\sigma'$ is joined by a chain of $n$-simplices $\sigma=\sigma_1,\dots,\sigma_N=\sigma'$ with $\sigma_i$ and $\sigma_{i+1}$ sharing a common $(n-1)$-face. 
\end{itemize}
In the maximally degenerate case $d=n$, if $X$ has semistable reduction, then 
$\Sk(X)$ is even a \emph{pseudomanifold}, \ie~(a) is replaced by
\begin{itemize}
\item[(a')] 
  every $(n-1)$-simplex of $\D$ is contained in exactly two $n$-simplices.
\end{itemize}
See also~\cite{KX} for even more precise results on the structure of $\Sk(X)$.
For example, $\Sk(X)$ is homeomorphic to a sphere if $n\le 3$
(still in the maximally degenerate case and $X$ having semistable reduction).
%
%
\subsection{The skeletal measure}
Consider the skeletal measure $\mu_{\log|\eta|^{\#}}$ on $\Sk(X)$. Choose an snc model $\cX$, and write as usual $\cX_0=\sum_{i\in I} b_i E_i$. The form $\eta$ defines an identification $K^\lo_{\cX/S}=\sum_{i\in I} a_i E_i$, and Proposition~\ref{prop:skelmetr} yields 
\begin{equation}\label{equ:kappamin}
\kappa_{\min}=\min_i\frac{a_i}{b_i}.
\end{equation}
If $\kappa_{\min}\in\Z$, then 
$$
\om:=\frac{d\unipar}{\unipar^{\kappa_{\min}+1}}\wedge\eta
$$ 
is a logarithmic form on $\cX$. For each face $\sigma$ of $\D(\cL)$,
ordering the set $J\subset I$ of components cutting out the stratum $Y=Y_\sigma$
yields a well-defined Poincar\'e residue
$\Res_Y(\om)$.
By Lemma~\ref{lem:boundaries}, $\Res_Y(\om)$ is a rational section of $K_Y$, with divisor 
$$
-B^\cL_Y=\sum_{i\notin J}(a_i-\kappa_{\min}b_i-1)E_i|_Y
$$ 
When $\sigma$ is a maximal face, $\Res_Y(\om)$ is thus a holomorphic
form on $Y$; using the formulas in~\S\ref{S309}, 
it is easy to see that the residual measure on $Y$ is given by
$$
\Res_Y(\log|\eta|^{\#})=|\Res_Y(\om)|^2. 
$$
The following result corresponds to Theorem~C in the introduction.
\begin{thm}\label{thm:resCY} 
  Assume that $X$ is maximally degenerate, \ie $\dim\Sk(X)=n$, and that 
  $X$ has semistable reduction, 
  Then the skeletal measure $\mu_{\log|\eta|^{\#}}$ is a multiple of the integral Lebesgue 
  measure of $\Sk(X)$. 
\end{thm}
\begin{proof} Let $\cX$ be a semistable model, \ie $\cX$ is snc with $\cX_0$ reduced. By~\eqref{equ:kappamin}, we have $\kappa_{\min}\in\Z$. Since some non-empty $E_J$ might have several components, 
the dual complex $\D(\cX)$ is possibly not a triangulation of $\Sk(X)$. However, the barycentric subdivision $\D'$ of $\D(\cX)$ is a triangulation; the corresponding toroidal modification $\cX'$ is snc, with $\cX'_0$ is possibly non-reduced, but $b_{\sigma}=1$ for each $n$-simplex $\sigma$ of $\D'$. Applying the above discussion to $\cX'$, we infer
$$
\mu_{\log|\eta|^{\#}}=\sum_{\sigma}|\Res_{y_\sigma}(\om)|^2\la_\sigma,
$$
with $\sigma$ ranging over the $n$-dimensional faces of $\D'$, with corresponding strata $y_\sigma\in\cX'_0$ reduced to single points. It will thus be enough to show that $|\Res_{y_\sigma}(\om)|$ is independent of $\sigma$. 

By the strong connectedness property, any two $n$-simplices $\sigma$, $\sigma'$ of $\D'$ can be joined by a chain of $n$-simplices $\sigma=\sigma_1,\dots,\sigma_N=\sigma'$ with $\sigma_i$ and $\sigma_{i+1}$ sharing a common $(n-1)$-face $\tau_i$. Denoting by $y_i=y_{\sigma_i}$ and $Y_i=Y_{\tau_i}$ the corresponding strata in $\cX'$, we thus have $y_i,y_{i+1}\in Y_i$. Further, the Poincar\'e residue $\Res_{Y_i}(\om)$ has poles precisely at $y_i,y_{i+1}$, since any other pole would correspond to an $n$-simplex of $\D'$ containing $\tau_i$, contradicting the non-branching property. Since $\Res_{y_i}\Res_Y(\om)=\Res_{y_i}(\om)$, the residue theorem applied to the Riemann surface $Y_i$ yields $\Res_{y_i}(\om)+\Res_{y_{i+1}}(\om)=0$, and hence $|\Res_{y_1}(\om)|=\dots=|\Res_{y_N}(\om)|$.
\end{proof}

\begin{rmk} Theorem~\ref{thm:resCY} fails in general when $X$ does not
  have semistable reduction. Indeed, the semistable reduction theorem~\cite{KKMS} shows that the base change $p\colon X'\to X$ to $\C\lau{\unipar^{1/m}}$ has semistable reduction for some $m$ divisible enough. By Lemma~\ref{lem:skelmetrbase}, $\dim\Sk(X')=n$, and $\mu_{\log|\eta'|^{\#}}$ is thus a multiple of the integral Lebesgue measure $\la'$ of $\Sk(X')$, by Theorem~\ref{thm:resCY}. By Theorem~\ref{thm:skelmesfinite}, $\mu_{\log|\eta|^{\#}}=m^{-n}p_*\la'$. However, $p_*\la'$ is not proportional to the integral Lebesgue measure $\la$ of $\Sk(X)$ in general. Indeed, for each $n$-simplex $\sigma$ of $\D(\cL)$, Lemma~\ref{lem:base} shows that $(p_*\la')_{\sigma}=m^n b_\sigma\la_{\sigma}$, and $b_\sigma$ is in general not independent of $\sigma$. 
\end{rmk}
%
%
%
%
\section{Extensions}\label{S312}
In this section we extend the main results in various directions.
%
%
\subsection{A singular version of Theorem A}
Let $\pi\colon\cX\to\DD$ be a projective, flat holomorphic map
of a normal complex space onto the disc, with $X:=\pi^{-1}(\DD^*)$
smooth over $\DD^*$. 
Since $\pi$ is projective, it defines a smooth projective
variety $X_{\C\lau{\unipar}}$ over $\C\lau{\unipar}$, as well as a model
$\cX_{\C\cro{\unipar}}$.

Let $\cL$ be a $\Q$-line bundle on
$\cX$ extending $K_{X/\DD^*}$, and $\p$ a continuous Hermitian metric on $\cL$. 
This data induces a continuous Hermitian metric $\psi_t$
on $K_{X_t}$ for $t\in\DD^*$, as well as a residually metrized model $\cL^{\#}$ of
$K_{X_{\C\lau{\unipar}}}$, the model given by $\cL_{\C\cro{\unipar}}$ and the metric 
by the restriction of $\p$ to $\cL_0=\cL|_{\cX_0}$.
Thus we obtain a skeletal measure $\mu_{\cL^{\#}}$ on $X^\an_{\C\lau{\unipar}}$.

Denote by $\cL'$ (resp.\ $\p'$) the pull-back of $\cL$ (resp.\ $\p$) to a log resolution $\cX'\to\cX$. By invariance of skeletal measures under pull-back, we have $\mu_{\cL'^{\#}}=\mu_{\cL^{\#}}$, and Theorem~\ref{thm:genconv} therefore implies: 

\begin{thm}\label{thm:gencvsing} The rescaled measures
$$
\mu_t:=\frac{e^{2\p_t}}{|t|^{2\kappa_{\min}}(2\pi\log|t|^{-1})^d}, 
$$
viewed as measures on $\cX'^\hyb$, converge weakly to $\mu_{\cL^{\#}}$. 
\end{thm}

\begin{cor}\label{cor:mass} If $\cX$ (\ie the pair $(\cX,\cX_{0,\red})$) is dlt, then 
$$
\lim_{t\to 0}\frac{\int_{\cX_t} e^{2\psi_t}}{|t|^{2\kappa_{\min}}(2\pi\log|t|^{-1})^d}=\sum_{\sigma}\left(\int_{Y_\sigma}\Res_{Y_\sigma}(\cL^{\#})\right) b_\sigma^{-1}\vol(\sigma), 
$$
where $\sigma$ runs over the $d$-dimensional faces of $\D(\cL)$. 
\end{cor}

When $d=0$, this implies the following slight generalization of~\cite[Lemma 1]{Li}. 

\begin{cor} Assume that $\cX_0$ has klt singularities (and hence $\cX$ is dlt by inversion of adjunction). Let $\psi$ be a continuous metric on $K_{\cX/\DD}$. Then $t\mapsto\int_{\cX_t} e^{2\psi_t}$ is continuous at $t=0$. 
\end{cor}
%
%
\subsection{Corollary~B for pairs}
Suppose $(X,B)$ is a projective subklt pair over $\DD^*$ that is meromorphic
at $0\in\DD$.

By Bertini's theorem (see~\cite[4.8]{KollarPairs} and also below), 
the pair $(X_t,B_t)$ is subklt for all $t\in\DD^*$ outside a discrete 
subset $Z$.
Let $\p$ be a continuous metric on $K_{{(X,B)}/\DD^*}$. 
As explained in~\S\ref{S310}, $\p$ induces a finite positive measure 
$e^{2(\p_t-\phi_{B_t})}$ on $X_t$ for $t\in\DD^*\setminus Z$.

Assume that $\p$ has analytic singularities in the sense that 
there exists a flat projective map $\cX\to\DD$ extending $X\to\DD^*$,
with $\cX$ normal, and a $\Q$-line bundle $\cL$ on $\cX$ extending 
$K_{(X,B)/\DD^*}$ such that $\p$ extends continuously to $\cL$. 

Our assumptions imply that $X$ is defined over the Banach ring $A_r$
described in Appendix~\ref{S308} for $0<r\ll1$. 
Let $X^\hyb$ be the analytification of 
the base change $X_{A_r}$. Recall that $X^\hyb$ naturally fibers over
$\overline{\DD}_r$, with $X^\hyb_{\overline{\DD}^*_r}\simeq X_{\overline{\DD}^*_r}$
and $X^\hyb_0\simeq X_{\C\lau{\unipar}}^\an$.
\begin{thm}\label{T301}
  The pair $(X_t,B_t)$ is subklt for $0<|t|\ll1$. Further, there 
  exist $\kappa_{\min}\in\Q$ and $d\in\N^*$ such that 
  the rescaled measures
  \begin{equation*}
    \mu_t:=\frac{e^{2\p_t}}{|t|^{2\kappa_{\min}}(2\pi\log|t|^{-1})^d}, 
  \end{equation*}
  viewed as measures on $X^\hyb$, converge weakly, as $t\to0$, to 
  a finite positive measure $\mu_0$ on $X^\hyb_0=X_{\C\lau{\unipar}}^\an$.
\end{thm}
A special case of Theorem~\ref{T301} is the log Calabi--Yau setting,
when the $\Q$-line bundle $K_{(X,B)/\DD^*}$ is trivial. 
In general, we are not able to give a very precise description of the limit measure 
$\mu_0$, but the proof will show that $\mu_0$ is a skeletal measure when the pair
$(X,B)$ is log smooth.
\begin{proof}[Proof of Theorem~\ref{T301}]
  Let us first treat the case when $(X,B)$ is log smooth. 
  In this case we need not assume that $X\to\DD^*$ is projective. 
  It follows from the normal crossings condition that 
  $(X_t,B_t)$ is subklt for $0<|t|\ll1$. After reparametrizing we may
  assume this is true for all $t\in\DD^*$, that is, $Z=\emptyset$.
  Set 
  \begin{equation*}
    \nu_t=e^{2(\p_t-\phi_{B_t})}.
  \end{equation*}
  This is a positive measure on $X_t$, smooth outside the support of $B_t$.
  Pick an snc model $(\cX,\cB)$ of $(X,B)$, where $\cB$ is the 
  closure of $B$ in $\cX$, such that $\p$ extends to a continuous metric
  on a $\Q$-line bundle $\cL$ on $\cX$ extending $K_{(X,B)/\DD^*}$.

  We can then prove a version of Theorem~A inside the hybrid
  space $\cX^\hyb$. By letting $\cX$ vary, we obtain 
  Theorem~\ref{T301} as a consequence, just as Corollary~B follows
  from Theorem~A. 

  The proof is very similar to the proof of Theorem~A, so we will only indicate the 
  modifications needed.
  Let us write
  \begin{equation*}
    K^\lo_{(\cX,\cB)/\DD}=\cL+\sum_{i\in I}a_i E_i
  \end{equation*}
  with $a_i\in\Q$. 
  Set $\kappa_i:=a_i/b_i$ and $\kappa_{\min}:=\min_i\kappa_i$. 
  Here $\cX_0=\sum_ib_iE_i$ as before.

  Define $\D(\cL)$ as the subcomplex of $\D(\cX)$ spanned by 
  the vertices such that $\kappa_i=\kappa_{\min}$. This will be the 
  support of the measure $\mu_0$. 
  For every stratum $Y$ corresponding to a maximal face of $\Delta(\cL)$,
  define a subklt pair $(Y,B_Y^\cL)$ using
  \begin{equation*}
    B^\cL_Y:=\cB|_Y+\sum_{i\notin J}(1-(a_i-\kappa_{\min}b_i))E_i|_Y
  \end{equation*}
  The residual measure $\Res_Y(\p)$ is given by
  \begin{equation*}
    \Res_Y(\p):=\exp\left(2(\p|_Y-\phi_{B^\cL_Y})\right).
  \end{equation*}
  Finally set 
  \begin{equation*}
    \mu_0:=\sum_{\sigma}\left(\int_{Y_\sigma}\Res_{Y_\sigma}(\p)\right)b_\sigma^{-1}\la_\sigma, 
  \end{equation*}
  where $\sigma$ ranges over the $d$-dimensional faces of $\D(\cL)$, 
  with $d=\dim\D(\cL)$.
  
  We then prove a version of Theorem~\ref{thm:genconv}. Namely,  if 
  \begin{equation*}
    \mu_t:=\frac{\la(t)^d}{(2\pi)^d|t|^{2\kappa_{\min}}}e^{2(\p_t-\phi_{B_t})}.
  \end{equation*}
  then we show that  $\mu_t$ converges to $\mu_0$ in $\cX^\hyb$ as $t\to0$.
  This is done via a local convergence result as in Lemma~\ref{lem:gencv}. 
  Namely, given a point $\xi\in\cX_0$, we choose local coordinates 
  $(z_0,\dots,z_n)$ at $\xi$ as in~\S\ref{S320}, 
  but further require that these coordinates 
  also cut out the irreducible components of $\cB$ containing $\xi$.
  More precisely, there exist $m$ with $p\le m\le n$ such that 
  these irreducible components are given by $B_i=\{z_i=0\}$
  for $p<i\le m$. 
  Also set $c_i:=\ord_{B_i}(\cB)<1$.

  A local $\Q$-generator for $\cL$ at $\xi$ is then given by 
  \begin{equation*}
    \tau=\prod_{i=0}^pz_i^{a_i}\prod_{i=p+1}^mz_i^{-c_i}\ \Omega^\rel,
  \end{equation*}
  with $\Omega^\rel$ as before. For a stratum $Y$ corresponding to a $d$-dimensional 
  simplex in $\D(\cL)$, the residual measure is given by
  \begin{equation}
    \Res_Y(\p)=|\tau|^{-2}_\p
    \prod_{i=d+1}^p|z_i|^{2(a_i-\kappa_{\min} b_i-1)}
    \prod_{i=p+1}^m|z_i|^{-2c_i}
    \bigg|\bigwedge_{i=d+1}^ndz_i\bigg|^2.
  \end{equation}
  
  The measure $\mu_t$ can be written near $\xi$ as 
  \begin{equation*}
    \mu_t=\frac{\la(t)^d}{(2\pi)^d}
    \frac{\prod_{i=p+1}^m|z_i|^{-2c_i}|\Omega_t|^2}
    {\big|\prod_{i=p+1}^mz_i^{-c_i}\Omega_t\big|^2_{\psi_t}}.
  \end{equation*}
  The proof now proceeds exactly as in~\S\ref{S318} except that we need to insert a factor 
  $\prod_{i=p+1}^m|z_i|^{-2c_i}$ 
  in the last two lines of~\eqref{equ:normmeas} and~\eqref{equ:normbis},
  the second line and the second factor of the last line of~\eqref{e303}, 
  and the right-hand sides of~\eqref{equ:rescplexbis}
  and~\eqref{e310}.
  This completes the proof in the log smooth case.

  \medskip
  Now we consider the general case, assuming $X\to\DD^*$ is projective.
  Pick a log resolution $q\colon(X',B')\to(X,B)$. 
  Since $(X'_t,B'_t)$ is subklt for $0<|t|\ll1$, the same is true for $(X_t,B_t)$.
  We have an induced continuous map $q^\hyb\colon (X')^\hyb\to X^\hyb$.
  By what precedes, there exist $\kappa\in\Q$ and $d\in\N$ such that the
  measure 
  $\mu'_t:=\frac{e^{2(\p'_t-\phi_{B'_t})}}{|t|^{2\kappa_{\min}}(2\pi\log|t|^{-1})^d}$
  on $X'_t=X^\hyb$ converges to a nonzero positive measure $\mu'_0$ on $(X')^\hyb$.
  By continuity, it follows that $\mu_t=q^\hyb_*\mu'_t$ converges to the nonzero
  positive measure $\mu_0=q^\hyb_*\mu'_0$ on $X^\hyb$.
  This completes the proof.
\end{proof}
%
%
\subsection{Degenerations of Ricci-flat K\"ahler manifolds}
Let $M$ be a Ricci-flat K\"ahler manifold, \ie a compact K\"ahler manifold with trivial first Chern class $c_1(M)\in H^2(M,\C)$. Then $M$ carries a canonical probability measure $\mu$, given by $\mu=e^{2\psi}/\int_M e^{2\psi}$ where $\psi$ is a Hermitian metric on $K_M$ with curvature $0$ (and hence unique up to a constant). 

By the Calabi-Yau theorem, each K\"ahler $(1,1)$-class on $M$ further contains a unique Ricci-flat K\"ahler metric $\om$, characterized by
$$
\frac{\om^n}{\int_M\om^n}=\mu.
$$
Recall also that $K_M$ is torsion, \ie $r K_M\simeq\cO_M$ for some positive integer $r$. Indeed, this is a consequence of the Beauville-Bogomolov theorem~\cite{Beau,Bog}, which implies that $M$ admits a finite \'etale cover $p:M'\to M$ with $K_{M'}=p^*K_M$ trivial. A trivializing section $\eta$ of $rK_M$ defines a metric $\psi=\tfrac 1 r\log|\eta|$ on $K_M$ as above, and hence $\mu=|\eta|^{2/r}/\int|\eta|^{2/r}$. 

As a consequence of Theorem A, we shall prove: 
\begin{thm}\label{thm:CYkahler}
  Let $\pi\colon X\to\DD^*$ be a holomorphic family of Calabi-Yau K\"ahler manifolds $X_t$, meromorphic at $t=0$, and let $\mu_t$ be the corresponding family of canonical probability measures. For any snc model $\cX$, $\mu_t$ converges in $\cX^\hyb$  to a skeletal measure $\mu_0$ supported in $\D(\cX)$. 
\end{thm}
\begin{proof} As recalled above, $K_{X_t}$ is torsion for each fixed $t$. Equivalently, $h^0(X_t,rK_{X_t})=1$ for some positive integer $r$. Since $t\mapsto h^0(X_t,rK_{X_t})$ is upper semicontinuous in the Zariski topology, it follows that $rK_{X_t}$ is trivial for a fixed $r$ independent of $t$. Given any snc model $\pi\colon\cX\to\DD$, $\pi_*\cO(rK_{\cX/\DD})$ is torsion free of rank one, and hence a line bundle. The choice of a trivializing section yields a holomorphic section $\eta$ of $K_{\cX/\DD}$, inducing a holomorphic family $\eta_t$ of trivializing sections of $rK_{X_t}$ for $t\ne 0$. As a consequence,  the family of volume forms $\nu_t:=|\eta_t|^{2/r}$ has analytic singularities at $t=0$, and the result is thus a consequence of Theorem A, since $\mu_t=\nu_t/\nu_t(X_t)$. 
\end{proof}
%
%
%
%
\appendix
\section{Berkovich spaces over Banach rings}\label{S308}
In this appendix we review the construction of the analytification of a scheme of finite
type defined over a Banach ring. The main reference for this 
is~\cite{BerkHodge}; see also~\cite{PoineauAsterisque,PoineauLocale,amoebae}.
For suitable choices of Banach rings, this leads to spaces that
contain both Archimedean and non-Archimedean data.
%
%
\subsection{Berkovich spectra}
Let $A$ be a Banach ring, that is, a commutative ring that is complete with respect to 
a submultiplicative norm $\|\cdot\|$.
The Berkovich spectrum $\cM(A)$ is the set of all bounded multiplicative seminorms 
on $A$. In other words, a point $x\in\cM(A)$ corresponds to a function 
$|\cdot|_x\colon A\to\R_{\ge0}$ such that $|\cdot|_x\le\|\cdot\|$, 
$|1|_x=1$, 
$|f+g|_x\le|f|_x+|g|_x$ 
and $|fg|_x=|f|_x|g|_x$ for $f,g\in A$.
The spectrum is a nonempty, compact Hausdorff space 
with respect to the topology of pointwise convergence.

For $x\in\cM(A)$, denote by $\fp_x$ the kernel of $|\cdot|_x$.
This is a prime ideal of $A$, and 
$|\cdot|_x$ defines a multiplicative norm on $A/\fp_x$.
The completion of the fraction field of $A/\fp_x$ with respect to this norm is a
valued field $\cH(x)$. We write $f(x)$ for the image of $f\in A$ in $\cH(x)$;
then $|f(x)|=|f|_x$.
The assignment $x\mapsto\fp_x$ yields 
a map $\cM(A)\to\Spec(A)$ that is continuous for the Zariski topology $\Spec(A)$. 

\begin{ex}\label{ex:field} 
  If $k$ is a valued field (\ie a field with a multiplicative norm),
  then $\cM(k)$ is a singleton.
\end{ex}

\begin{ex}\label{ex:complexbanach} 
  When $A$ is a complex Banach algebra, the Gelfand-Mazur Theorem
  implies that the Berkovich spectrum agrees with the maximal ideal  spectrum.
\end{ex}
%
%
\subsection{Analytification of a scheme} 
To any scheme $X$ of finite type over a Banach ring $A$, Berkovich associates 
an \emph{analytification}\footnote{We use the term analytification even though we shall only consider $X^{\An}$ as a topological space. In particular, $X^{\An}$ only depends on the reduced scheme structure of $X$.}
$X^{\An}$, a locally compact topological space with a continuous morphism $X^{\An}\to\cM(A)$, defined as follows. 

When $X=\Spec B$ is affine, with $B$ a 
finitely generated $A$-algebra, $X^{\An}$ is defined as the set of
multiplicative seminorms $|\cdot|_x$ on $B$ 
whose restriction to $A$ is bounded
by the given norm on $A$, \ie belongs to $\cM(A)$. The topology on $X^{\An}$ is the weakest one for
which $x\mapsto|f|_x=|f(x)|$ is continuous for every $f\in B$. 

In the general case, the analytification $X^{\An}$ is defined by gluing together the analytifications of an affine open cover, and yields a covariant functor $X\mapsto X^{\An}$. If $X\hookrightarrow Y$ is an open (resp.\ closed) embedding, then so is $X^{\An}\hookrightarrow Y^{\An}$. If $X\to Y$ is surjective, then so is $X^{\An}\to Y^{\An}$.

The topological space $X^{\An}$ is Hausdorff (resp.\ compact) if $X$ is separated
(resp.\ projective). 
The assignment 
$x\mapsto\fp_x$ above globalizes to a continuous map 
\begin{equation*}
  \pi\colon X^{\An}\to X,
\end{equation*}
where $X$ is equipped with the Zariski topology.

When $A$ is a valued field, it is more common to write $X^\an$ instead
of $X^\An$~\cite{BerkBook}.
\begin{ex} 
  For $A=\C$, the Gelfand-Mazur theorem shows that $X^{\An}$ coincides with the usual analytification of $X$, \ie the set $X(\C)$ of complex points of $X$ endowed with the euclidean topology. 
\end{ex}
%
%
\subsection{The hybrid norm on $\C$}\label{S107}
Denote by $\C_{\hyb}$ the Banach field $(\C,\|\cdot\|_{\hyb})$, 
where the \emph{hybrid norm} is defined as
\begin{equation*}
  \|\cdot\|_{\hyb}:=\max\{|\cdot|_0,|\cdot|_\infty\},
\end{equation*}
with $|\cdot|_0$ the trivial absolute value and $|\cdot|_\infty$ the
usual absolute value. 

The elements of the Berkovich spectrum $\cM(\C_{\hyb})$ are of the form 
$|\cdot|_\infty^\rho$ for $\rho\in[0,1]$, 
interpreted as the trivial absolute value $|\cdot|_0$ for $\rho=0$.
This yields a homeomorphism 
$\cM(\C_{\hyb})\simeq [0,1]$.
%
%
\subsection{Hybrid geometry over $\C$}
If $X$ is a scheme of finite type over $\C$, we denote by
$X^{\hol}=X(\C)$ its analytification with respect to the usual
absolute value $|\cdot|_\infty$, by $X^{\an}_0$ its analytification
with respect to the trivial absolute value, and by $X^{\hyb}$ its
analytification with respect to the hybrid norm
$\|\cdot\|_{\hyb}$. 

From the structure morphism $X\to\Spec\C$ we obtain a 
continuous map $\la\colon X^{\hyb}\to\cM(\C_{\hyb})\simeq[0,1]$.
The fiber $\lambda^{-1}(\rho)$ is equal to the analytification of
$X$ with respect to the multiplicative norm $|\cdot|_\infty^\rho$ 
on $\C$. In particular, we have canonical identifications $\lambda^{-1}(1)\simeq X^{\hol}$ and $\la^{-1}(0)\simeq X^{\An}_0$. 
For $0<\rho\le1$, the fiber $\lambda^{-1}(\rho)$ is also 
homeomorphic to $X^{\hol}$. In fact, we have a a homeomorphism
\begin{equation*}
  \lambda^{-1}\left((0,1]\right)\simeq(0,1]\times X^{\hol},
\end{equation*}
see~\cite[Lemma~2.1]{BerkHodge}. 
%
%
\subsection{The hybrid circle}
Now consider the \emph{hybrid circle} of radius $r\in(0,1)$,
that is, $C_\hyb(r):=\{|\unipar|=r\}\subset\A^{1,\hyb}=(\Spec\C[\unipar])^\hyb$.
By~\cite[Prop 2.1.1]{PoineauAsterisque}, this is compact and 
realized as the Berkovich spectrum of the Banach ring
\begin{equation*}
  A_r:=\left\{f=\sum_{\a\in\Z} c_\a \unipar^\a\in\C\lau{\unipar}\ \bigg|\ 
    \|f\|_{\hyb}:=\sum_{\a\in\Z}\|c_\a\|_{\hyb}r^\a<+\infty\right\}.
\end{equation*}
Since $\|c_\a\|_{\hyb}\ge|c_\a|_\infty$,
every $f\in A_r$ defines a continuous function $f^{\hol}$ on the
punctured closed disc $\overline{\DD}^*_r$ that is holomorphic on $\DD^*_r$
and meromorphic at 0.
\begin{prop}\label{prop:hybcirc} 
  There is a homeomorphism $\overline{\DD}_r\simto\cM(A_r)\simeq C_{\hyb}(r)$, 
  that maps $z\in\overline{\DD}_r\subset\C$ to the seminorm on $A_r$ defined by
  \begin{equation}\label{e201}
    |f|=
    \begin{cases}
      r^{\ord_0(f)}&\ \text{if $z=0$} \\ 
      r^{\frac{\log|f^{\hol}(z)|_\infty}{\log|z|_\infty}}\ &\text{otherwise},
    \end{cases}
  \end{equation}
  and via which the map $\la\colon C_{\hyb}(r)\to[0,1]$ is given by 
  $\la(z)=\frac{\log r}{\log|z|_\infty}$. 
\end{prop}
\begin{proof} 
  The map $\tau\colon\overline{\DD}_r\to\cM(A_r)$ given by~\eqref{e201} is
  clearly well defined. It is also continuous on $\overline{\DD}^*_r$.
  To prove continuity at $0$, we note that for 
  each $f\in A_r$, we can write $f^{\hol}=z^{\ord_0(f)}u$, 
  where $u$ is a continuous function on $\overline{\DD}_r$ that is holomorphic
  on $\DD_r$ with $u(0)\ne0$. 
  As a consequence, we get 
  $\lim_{z\to0}\frac{\log|f^{\hol}(z)|_\infty}{\log|z|_\infty}=\ord_0(f)$.
  
  Now, for each $\rho\in(0,1]$, $\la^{-1}(\rho)\subset C_{\hyb}(r)$
  can be identified with the circle of radius $r$ with respect to the 
  absolute value $|\cdot|_\infty^\rho$, while $\la^{-1}(0)$ is the
  non-Archimedean absolute value $r^{-\ord_0}$ on $\C\lau{\unipar}$. 
  This proves that the map $\tau$ above is bijective,
  and hence a homeomorphism by compactness. 
\end{proof}
\begin{rmk}
  When $r<s$, the identity gives a bounded map from $A_s$ to $A_r$,
  and $\varinjlim_{r\to0}A_r$ is the fraction field of $\cO_{\C,0}$, \ie the ring of
  meromorphic germs at the origin of $\C$. 
\end{rmk}
%
%
\subsection{Geometry over the hybrid circle}\label{S319}
Let now $X$ be a scheme of finite type over $A_r$. 
We will associate to $X$ three kinds of analytic spaces. 

First, since $X$ is obtained by gluing together finitely many affine
schemes cut out by polynomials with coefficients holomorphic 
on $\DD^*_r\subset\C$ and meromorphic at $0$, 
we can associate to $X$ in a functorial way a complex analytic space
$X^{\hol}$ over $\DD^*_r$, which we call its \emph{holomorphic analytification}. 

Second, since $A_r$ is contained in $\C\lau{\unipar}$, we may also 
consider the base change $X_{\C\lau{\unipar}}$ and its 
\emph{non-Archimedean analytification} 
$X^{\an}_{\C\lau{\unipar}}$
with respect to the non-Archimedean absolute value $r^{\ord_0}$ on
$\C\lau{\unipar}$.

Third, we denote by $X^{\hyb}$ the analytification of $X$ as a
scheme of finite type over the Banach ring $A_r$, 
and call it the \emph{hybrid analytification} of $X$. 
In view of Proposition~\ref{prop:hybcirc}, 
it comes with a continuous structure map 
\begin{equation*}
  \pi\colon X^{\hyb}\to\overline{\DD}_r\simeq\cM(A_r), 
\end{equation*}
Recall further that $X^{\hyb}$ is locally compact, Hausdorff if $X$ is
separated, and compact if $X$ is proper over $A_r$. 
The discussion above implies:
\begin{lem}\label{lem:hyban} 
  We have canonical homeomorphisms
  \begin{equation}\label{equ:hybident}
    \pi^{-1}(0)\simeq X_{\C\lau{\unipar}}^\an
    \qand
    \pi^{-1}(\DD^*_r)\simeq X^{\hol}
  \end{equation}
  compatible with the projection to $\DD_r$. 
\end{lem}
In~\S\ref{S315} we give a topological description of $X^\hyb$.

\end{document}

%% file: preamb.tex
\numberwithin{equation}{section}       

\theoremstyle{plain}
\newtheorem{thm}{Theorem}[section]
\newtheorem{prop}[thm]{Proposition}
\newtheorem{lem}[thm]{Lemma}
\newtheorem{cor}[thm]{Corollary}
\newtheorem{prop-def}[thm]{Proposition-Definition}

\newtheorem*{ThmA}{Theorem A}

\newtheorem*{ThmC}{Theorem C}

\newtheorem*{CorB}{Corollary B}

\theoremstyle{definition}

\newtheorem{defi}[thm]{Definition}

\theoremstyle{remark}
\newtheorem{ex}[thm]{Example}

\newtheorem{rmk}[thm]{Remark}

\newcounter{alphasect}
\def\alphainsection{0}

\let\oldsection=\section
\def\section{%
  \ifnum\alphainsection=1%
\e    \addtocounter{alphasect}{1}
  \fi%
\oldsection}%

\renewcommand\thesection{%
 \ifnum\alphainsection=1%
   \Alph{alphasect}%
 \else
   \arabic{section}%
 \fi%
}%

\newcommand{\A}{{\mathbb{A}}}
\newcommand{\C}{{\mathbb{C}}}
\newcommand{\DD}{{\mathbb{D}}}

\newcommand{\G}{{\mathbb{G}}}
\newcommand{\N}{{\mathbb{N}}}
\renewcommand{\P}{{\mathbb{P}}}
\newcommand{\Q}{{\mathbb{Q}}}
\newcommand{\R}{{\mathbb{R}}}
\newcommand{\Z}{{\mathbb{Z}}}

\newcommand{\bF}{{\bar{F}}}

\newcommand{\cB}{{\mathcal{B}}}

\newcommand{\cH}{{\mathcal{H}}}

\newcommand{\cL}{{\mathcal{L}}}
\newcommand{\cM}{{\mathcal{M}}}

\newcommand{\cO}{{\mathcal{O}}}

\newcommand{\cU}{{\mathcal{U}}}
\newcommand{\cV}{{\mathcal{V}}}

\newcommand{\cX}{{\mathcal{X}}}

\newcommand{\fp}{{\mathfrak{p}}}

\newcommand{\simto}{\overset\sim\to}

\renewcommand{\a}{\alpha}
\renewcommand{\b}{\beta}

\newcommand{\la}{\lambda}
\newcommand{\om}{\omega}

\newcommand{\e}{\varepsilon}
\newcommand{\f}{\varphi}
\newcommand{\p}{\psi}
\newcommand{\D}{\Delta}

\newcommand{\eg}{e.g.\ }
\newcommand{\ie}{i.e.\ }

\newcommand{\loccit}{\textit{loc.\,cit.\ }}

\newcommand{\qand}{{\quad\text{and}\quad}}

\newcommand{\Xan}{{X^\mathrm{an}}}

\newcommand{\red}{\mathrm{red}}
\newcommand{\rel}{\mathrm{rel}}
\newcommand{\snc}{\mathrm{snc}}
\newcommand{\lo}{\mathrm{log}}
\renewcommand{\div}{\mathrm{div}}

\newcommand{\emb}{\operatorname{emb}}

\newcommand{\id}{\operatorname{id}}

\newcommand{\MA}{\operatorname{MA}}

\newcommand{\PA}{\operatorname{PA}}

\newcommand{\hol}{\operatorname{hol}}
\newcommand{\Hom}{\operatorname{Hom}}

\newcommand{\Ker}{\operatorname{Ker}}

\newcommand{\Log}{\operatorname{Log}}

\newcommand{\ord}{\operatorname{ord}}

\newcommand{\Sk}{\operatorname{Sk}}

\newcommand{\Spec}{\operatorname{Spec}}
\newcommand{\supp}{\operatorname{supp}}

\newcommand{\val}{\operatorname {val}}

\newcommand{\vol}{\operatorname{Vol}}

\newcommand{\Res}{\operatorname{Res}}

\newcommand{\reg}{\mathrm{reg}}

\newcommand{\hyb}{\mathrm{hyb}}
\newcommand{\an}{\mathrm{an}}
\newcommand{\An}{\mathrm{An}}

\newcommand{\cro}[1]{[\![#1]\!]}
\newcommand{\lau}[1]{(\!(#1)\!)}

\newcommand{\unipar}{{t}}